%% file: ex_article.tex
\documentclass[final,hidelinks,onefignum,onetabnum]{siamart190516}

\input{ex_shared}

\ifpdf
\hypersetup{
  pdftitle={Compression, inversion, and approximate {PCA} of dense kernel matrices at near-linear computational complexity},
  pdfauthor={Florian Sch{\"a}fer, T.J. Sullivan, and Houman Owhadi}
}
\fi

\begin{document}

\maketitle

% REQUIRED
\begin{abstract}
	Dense kernel matrices $\KM \in \Reals^{N \times N}$ obtained from point evaluations of a covariance function $\K$ at locations $\{ x_{i} \}_{1 \leq i \leq N} \subset \Reals^{d}$ arise in statistics, machine learning, and numerical analysis.
	For covariance functions that are Green's functions of elliptic boundary value problems and homogeneously-distributed sampling points, we show how to identify a subset $S \subset \{ 1 , \dots , N \}^2$, with $\# S = \BigO ( N \log (N) \log^{d} ( N  /\epsilon ) )$, such that the zero fill-in incomplete Cholesky factorisation of the sparse matrix $\KM_{ij} \one_{( i, j ) \in S}$ is an $\epsilon$-approximation of $\KM$.
	This factorisation can provably be obtained in complexity $\BigO ( N \log( N ) \log^{d}( N /\epsilon) )$ in space and $\BigO ( N \log^{2}( N ) \log^{2d}( N /\epsilon) )$ in time, improving upon the state of the art for general elliptic operators;
	we further present numerical evidence that $d$ can be taken to be the intrinsic dimension of the data set rather than that of the ambient space.
	The algorithm only needs to know the spatial configuration of the $x_{i}$ and does not require an analytic representation of $\K$.
	Furthermore, this factorization straightforwardly provides an approximate sparse PCA with optimal rate of convergence in the operator norm.
	Hence, by using only subsampling and the incomplete Cholesky factorization, we obtain, at nearly linear complexity, the compression, inversion and approximate PCA of a large class of covariance matrices.
	By inverting the order of the Cholesky factorization we also obtain a solver for elliptic PDE with complexity $\BigO ( N \log^{d}( N /\epsilon) )$ in space and $\BigO ( N  \log^{2d}( N /\epsilon) )$ in time, improving upon the state of the art for general elliptic operators.
\end{abstract}

% REQUIRED

% REQUIRED

\section{Introduction}
\label{sec:introduction}
\input{./sec-introduction.tex}

\section{Overview of the algorithm and its setting}
\label{ssec-simpleAlg}
\input{./sec-overview.tex}

\section{Why it works --- justification of the method}
\label{ssec-justification}
\input{./sec-whyitworks.tex}

\section{Implementation and numerical results}
\label{sec:numerics}
\input{./sec-numerics.tex}

\section{Analysis of the algorithm}
\label{sec:analysis}
\input{./sec-analysis.tex}

\section{Extensions and byproducts}
\input{./sec-extensions.tex}

\section{Comparison to related work}
\input{./sec-comparison.tex}

\section{Conclusions}
\label{sec:conclusions}
\input{./sec-conclusions.tex}

\section*{Acknowledgments}
FS and HO gratefully acknowledge support by the Air Force Office of Scientific Research and the DARPA EQUiPS Program (award number FA9550-16-1-0054, Computational Information Games) and the Air Force Office of Scientific Research  (award number FA9550-18-1-0271, Games for Computation and Learning).
TJS has been supported by the Freie Universit{\"a}t Berlin within the Excellence Initiative of the German Research Foundation.
This collaboration has been facilitated by the Statistical and Applied Mathematical Sciences Institute through the National Science Foundation award number DMS-1127914.
Any opinions, findings, and conclusions or recommendations expressed in this material are those of the authors and do not necessarily reflect the views of the above-named institutes and agencies.
We would like to thank C.~Oates and P.~Schr{\"o}der for helpful discussions, and C.~Scovel for many helpful comments and suggestions.

\bibliographystyle{siamplain}
\bibliography{references}
\appendix

\section{Correctness and computational complexity of the maximum-minimum distance ordering}
\label{apsec-compSortSparse}

Recall the variables used in \cref{alg-sortSparse}:
the integer array $P$ contains the minimax ordering;
the real array $l[i]$ contains the distances of each point to the points that are already included in the minimax ordering;
and the arrays of integer arrays $c$ and $p$ will contain the entries of the sparsity pattern in the sense that
\begin{equation}
	( j \in c[i] \text{ and } i \in p[i] ) \iff \disttt(i,j) \leq \rho l[i].
\end{equation}
We begin by showing correctness of the algorithm.

\begin{theorem}
	\label{thm-corrSortSparse}
	The ordering and sparsity pattern produced by \cref{alg-sortSparse} coincide with those described in \cref{sec:introduction}.
	Furthermore, whenever the while-loop in \cref{line-whileSortSparse} is entered,
	\begin{enumerate}[label=(\arabic*)]
		\item for all $i \in P$, $l[i]$ is as defined in \cref{sec:introduction};
		\item the array $c[P[1]]$ contains all $1 \leq j \leq N$ and for all other $i$ in $P$, $c[i]$ contains exactly those $1\leq j \leq N$ that satisfy $\disttt( i, j ) \leq \rho l[i]$;
		and
		\item for all $1 \leq j \leq N$, $p[j]$ consists of $P[1]$ and all those $i \in P$ that satisfy $\disttt( i, j ) \leq \rho l[i]$.
	\end{enumerate}
\end{theorem}

\begin{proof}
	It is easy to see that if the for-loop in \cref{line-forSortSparse} were running over all $1 \leq j \leq N$, then the algorithm would yield the correct result.
	We claim that the restriction of the running variable to $\{ j \in c[k] \mid \disttt(j,k) \leq \disttt(i,k) + \rho l[i] \}$ does not change the result of the algorithm.
	The proof will proceed by induction.
	Let us assume that the \cref{alg-sortSparse} has been correct up to a given time that \cref{line-forSortSparse} is visited.
	Then, by choice of $k$ and the triangle inequality, any $j$ that is omitted by the for-loop must satisfy $\disttt(i,j) > \rho l[i]$.
	Since $i$ was chosen to have maximal minimal distance among the points remaining in $H$, and $\rho > 1$, this means that adding $i$ to the maximin ordering can not decrease the maximal minimal distance of $j$.
	Thus, skipping the $\mathtt{decrease!}$ operation does not change the choice of $P$ and $l$.
	Similarly, $\disttt(i,j) > \rho l[i]$ implies that the if-statement inside of the for-loop is false, meaning that skipping $j$ does not change the update of $c$ or $p$, from which the result follows.
\end{proof}

Having established \cref{thm-corrSortSparse}, we will now use $\prec$, $l[i]$, and $i_k$ to refer to the maximin ordering, the length-scale of the point with index $i$, and the $k$\textsuperscript{th} index in the maximin ordering.
We will now bound the complexity of \cref{alg-sortSparse} in the setting of \cref{thm-decayApproxIChol}.

\begin{theorem}
	\label{thm-compSortSparseOmega}
	In the setting of \cref{thm-decayApproxIChol}, there exists a constant depending $C$ depending only on $d$, $\Omega$, and $\delta$, such that, for $\rho > C$, \cref{alg-sortSparse} has computational complexity
	$C \rho^d N \log N$ in space and $C N \bigl( \rho^d \log^{2} N + C_{\disttt_{\partial \Omega}} \bigr)$ in time, where $C_{\disttt_{\partial \Omega}}$ is the computational complexity of invoking the function $\disttt_{\partial \Omega}$.
\end{theorem}

\begin{proof}
	As a first step, we will upper-bound the number of iterations of the for-loop in \cref{line-forSortSparse} throughout the algorithm.
	To simplify the notation, $C$ will denote a positive constant that depends on $d$, $\Omega$ and $\delta$ that may change throughout the proof.
	We claim that there exists $1 \leq k_{\min} \leq N$ depending only on $d$, $\Omega$, and $\delta$, such that, for all $i \succ i_{k_{\min}}$, by the time it appears in the while-loop at \cref{line-whileSortSparse}, there exists an index $k \prec i$ such that $l[k] \geq 2 l[j]$ and $\disttt(i,k) \leq C l = C l[i]$.
	Indeed, since $\Omega$ has Lipschitz boundary, it satisfies an interior cone condition \cite{adams2003sobolev} in the sense that there exist $\theta \in (0,2\pi]$ and $r > 0$ such that every point $x \in \Omega$ is the tip of a spherical cone within $\Omega$ with opening angle $\theta$ and radius $r$.
	This spherical cone contains a ball with radius $r_{\gamma}$, which depends only on $\theta$ and $r$.
	Let $\gamma_i$ be such a cone with tip $x_i$.
	By a scaling argument, the spherical cone $\gamma_i \cap B_{\tilde{r}}(x_i)$ then contains a ball of radius $r_{\gamma} ( \tilde{r} / r)$, for all $\tilde{r} < r$.
	For any $i \in \I$ and any ball $B \subset \Omega$ with radius at least $4 l[i] / \delta$, there exists a $k \prec i$ such that $l[k] \geq 2 l[i]$ and $x_k \in B$.
	Thus, for $l[i] \leq \delta r_{\gamma}/4$, there exists a $k \prec i$ with $x_k \subset \gamma_{i} \cap B_{2 l[i]}(x_i)$.
	By a sphere-packing argument, we can find a $k_{\min}$ such that, for all $i \succ i_{k_{\min}}$, $l[i] \leq \delta r_{\gamma}/4$, which yields the claim.
	Because of the above, for $\rho > C$, there exists a point satisfying the constraint in \cref{line-kArgminSortSparse} with $\disttt(k,i) \leq 2 Cl[i]$.
	Thus, the number of times the for-loop in \cref{line-forSortSparse} is visited for a given index $i$ is bounded above by $C_i \defeq \# \{ j \in \I \mid \disttt(i,j) \leq 2 (C + \rho ) l[i] \}$.
	By a sphere-packing argument, $C_{i_m} \leq C ( N/m ) \rho^d$, for a constant $C$ depending only on $d$, $\Omega$, and $\delta$.
	Summing the above over $1 \leq m \leq N$ yields the upper bound $C \rho^d N \log N$.
	The most costly step in the for-loop in \cref{line-forSortSparse} is the $\texttt{decrease!}$ operation requiring the restoration of the heap property, which has computational complexity $\BigO(\log N)$.
	Thus, the overall computational complexity is at most $C N \bigl( \rho^d \log^{2} N + C_{\disttt_{\partial \Omega}} \bigr)$.
	The bound on the space complexity follows, since each iteration of the for-loop	consumes $\BigO(1)$ memory.
\end{proof}

\begin{proof}[Proof of \cref{thm-compSortSparse}]
	\Cref{thm-compSortSparse} follows from \cref{thm-corrSortSparse,thm-compSortSparseOmega}.
\end{proof}

\Cref{alg-sortSparse} uses only pairwise distances between points, and thus automatically adapts to low-dimensional structure in the $\{ x_i \}_{i \in \I}$.
Indeed, for $\Omega = \Reals^d$, the computational complexity of \cref{alg-sortSparse} depends only on the intrinsic dimension of the dataset.

\begin{condition}[Intrinsic dimension]
	\label{cond-intdim}
	There exist constants $C_{\intd}, \intd > 0$, independent of $N$, such that, for all $r,R > 0$ and $x \in \Reals^{d}$,
	{\small
	\begin{equation*}
		\max \left\{ |A| \,\middle|\, \vphantom{\big|} A \subset \I, i,j \in A \implies \dist ( x_{i}, x ), \dist ( x_{j}, x ) \leq R , \dist (x_{i}, x_{j} ) \geq r \right\} \leq C_{\intd} \left( \frac{R}{r} \right)^{\intd}.
	\end{equation*}}
	We say that the point set $\{ x_{i} \}_{i \in \I}$ has \emph{intrinsic dimension} $\intd$.
\end{condition}

\begin{condition}[Polynomial Scaling]
	\label{cond-polyScale}
	There exists a polynomial $\boldsymbol{p}$ for which
	\begin{equation*}
		\frac{\max_{i \neq j \in \I} \dist( x_{i}, x_{j} )}{\min_{i \neq j \in \I} \dist( x_{i}, x_{j} )} \leq \boldsymbol{p}( N ) .
	\end{equation*}
\end{condition}

\begin{theorem}
	\label{thmjhgfhgfty}
	Let $\Omega = \Reals^d$ and $\rho \geq 2$.
	Then the computational complexity of \cref{alg-sortSparse} is at most $C \rho^{\intd} N \log N$ in space and $C N \bigl( \log(N) \rho^{\intd} (\log N + C_{\disttt} ) + C_{\disttt_{\partial \Omega}} \bigr)$ in time, for a constant $C = C \bigl( C_{\intd}, \intd, \boldsymbol{p} \bigr)$ depending only on the constants in \cref{cond-intdim,cond-polyScale}.
\end{theorem}

\begin{proof}
	The proof is analogous to that of \cref{thm-compSortSparseOmega}.
	The main difference is that the claim on the existence of $k_{\min}$ is replaced by the fact --- which follows directly from the definition of the maximin ordering --- that, for all $i$, there exists a $k \prec i$ such that $l[k] \geq 2 l[i]$ and $\disttt(k,i) \leq 2 l[i]$.
	In particular, any $\rho \geq 2$ leads to near-linear computational complexity.
\end{proof}

\section{Proofs of \cref{sec:analysis}}
\label{apsec-abstractDecay}

\begin{proof}[Proof of \cref{lem-BlockCholesky}]
	The main idea is to recursively apply \cref{lem-blockChol2Scale}.
	First, applying \cref{lem-blockChol2Scale} with first block $\KM_{1:q-1,1:q-1}$ and second block $\KM_{q,q}$ yields
	{\tiny \begin{align}
		\KM^{(q)} =&
		\begin{pmatrix}
			&                           &                        &  0  \\
			&  \Id    &                & \vdots            \\
			&                           &                        & 0  \\
			-\IKMC^{(q),-1} \IKM^{(q)}_{q,1} & \cdots  & -\IKMC^{(q),-1} \IKM^{(q)}_{q,q-1}                     & \Id
		\end{pmatrix}\\
		&\quad \quad \quad \quad \quad \quad \quad \quad \quad \quad \quad \quad \quad \quad \quad \quad \begin{pmatrix}
			&             &                &  0     \\
			&\KM^{(q-1)}  &                & \vdots \\
			&             &                & 0      \\
			0 & \cdots      & 0              & \IKMC^{(q),-1}
		\end{pmatrix}
		\begin{pmatrix}
			&         &           &-\IKM^{(q)}_{1,q} \IKMC^{(q),-\top}      \\
			&  \Id    &           & \vdots                                  \\
			&         &           & -\IKM^{(q)}_{q-1,q} \IKMC^{(q),-\top}  \\
			0 & \cdots  & 0         & \Id
		\end{pmatrix}.
	\end{align} }
	We now repeat this operation recursively.
	After the $k$\textsuperscript{th} step, the central matrix has an upper-left block consisting of $\KM^{(q-k)}$.
	We then apply \cref{lem-blockChol2Scale} to this upper-left block, with the splitting given by $\KM_{1:q-k-1,1:q-k-1}$ and $\KM_{q-k,q-k}$.
	This reduces the central matrix more and more towards the block-diagonal matrix $D$, while splitting off a triangular factor to either side.
	Doing this up to the $(q-1)$\textsuperscript{th} step yields the following identity:
	{\tiny \begin{align}
		& \begin{pmatrix}
			\KM^{(q)}_{1 , 1}   & \cdots             & \KM^{(q)}_{1 , q-1}       & \KM^{(q)}_{1 , q}   \\
			\vdots            & \ddots            & \vdots                   & \vdots            \\
			\KM^{(q)}_{q-1,1}   & \cdots            & \KM^{(q)}_{q-1 , q-1}     & \KM^{(q)}_{2 , q}  \\
			\KM^{(q)}_{q , 1}   & \cdots             & \KM^{(q)}_{q , q-1}       & \KM^{(q)}_{q , q}
		\end{pmatrix}
		\\
		& \quad =
		\begin{pmatrix}
			&                           &                        &  0  \\
			&  \Id    &                & \vdots            \\
			&                           &                        & 0  \\
			-\IKMC^{(q),-1} \IKM^{(q)}_{q,1}           & \cdots                     & -\IKMC^{(q),-1} \IKM^{(q)}_{q,q-1}                     & \Id
		\end{pmatrix} \\
		&\quad \phantom{=} \phantom{cdots} \quad \quad
		\begin{pmatrix}
			&                           &                        &  0               & 0 \\
			&  \Id                      &                        & \vdots           & \vdots \\
			&                           &                        & 0                & \vdots \\
			-\IKMC^{(q-1),-1} \IKM^{(q-1)}_{q-1,1}           & \cdots                     & -\IKMC^{(q-1),-1} \IKM^{(q-1)}_{q-1,q-2}                     & \Id       & 0 \\
			0                   &           \cdots           &           \cdots        &        0     &       \Id         &
		\end{pmatrix}
		\cdots \\
		& \quad \phantom{=} \quad \cdots
		\begin{pmatrix}
			\Id                 &             0             &                          0             \\
			-\IKMC^{(2),-1} \IKM^{(2)}_{2,1}            &  \Id                                        & \vdots          \\
			0                  &                \cdots           &              \Id       &
		\end{pmatrix}\\
		&\quad \phantom{=} \phantom{cdots} \quad \quad
		\begin{pmatrix}
			\IKMC^{( 1 ),-1} & 0 & \cdots & \cdots & 0\\
			0 & \IKMC^{(2),-1} & \ddots & 0 &   \vdots\\
			\vdots & \ddots  & \ddots & \ddots& \vdots \\
			\vdots & 0 & \ddots & \IKMC^{(q-1),-1} & 0 \\
			0 & 0 & \cdots & 0 & \IKMC^{(q),-1}
			\end{pmatrix}
			\begin{pmatrix}
			\Id                 &            -\IKMC^{(2)}_{1,2} \IKMC^{(2),-T}_{2,2}         &                          0             \\
			0                  &  \Id                                        & \vdots          \\
			0                  &                \cdots           &              \Id       &
		\end{pmatrix} \cdots
		\\
		& \quad \phantom{=} \quad \cdots
		\begin{pmatrix}
			&                           &                        &  -\IKM^{(q-1)}_{1,q-1} \IKMC^{(q-1),-\top} & 0  \\
			&  \Id                      &                        & \vdots            & \vdots \\
			&                           &                        & -\IKM^{(q-1)}_{q-2,q-1} \IKMC^{(q-1),-\top}  & \vdots \\
			0           & \cdots                     & 0                      & \Id                  &  0 \\
			0           & \cdots                     & \cdots                  & 0                    & \Id
		\end{pmatrix}
		\begin{pmatrix}
			&                           &                        &  -\IKM^{(q)}_{1,q} \IKMC^{(q),-\top}  \\
			&  \Id                      &                        & \vdots            \\
			&                           &                        & -\IKM^{(q)}_{q-1,q} \IKMC^{(q),-\top}  \\
			0           & \cdots                     & 0                      & \Id
		\end{pmatrix}.
	\end{align} }
	We now combine the lower-triangular factors, obtaining
	{\tiny
	\begin{align}
		& \begin{pmatrix}
			\Id                 &            -\IKM^{(2)}_{1,2} \IKMC^{(2),-\top}         &                          0             \\
			0                  &  \Id                                        & \vdots          \\
			0                  &                \cdots           &              \Id       &
		\end{pmatrix}
		\cdots
		% \begin{pmatrix}
		% 	&                           &                        &  -\IKM^{(q-1)}_{1,q-1} \IKMC^{(q-1),-\top} & 0  \\
		% 	&  \Id                      &                        & \vdots            & \vdots \\
		% 	&                           &                        & -\IKM^{(q-1)}_{1,q-1} \IKMC^{(q-1),-\top}  & \vdots \\
		% 	0           & \cdots                     & 0                      & \Id                  &  0 \\
		% 	0           & \cdots                     & \cdots                  & 0                    & \Id
		% \end{pmatrix}
		\begin{pmatrix}
			&                           &                        &  -\IKM^{(q)}_{1,q} \IKMC^{(q),-\top}  \\
			&  \Id                      &                        & \vdots            \\
			&                           &                        & -\IKM^{(q)}_{1,q} \IKMC^{(q),-\top}  \\
			0           & \cdots                     & 0                      & \Id
		\end{pmatrix} =
		\\
		& \left(
		\begin{pmatrix}
			&                           &                        &  \IKM^{(q)}_{1,q} \IKMC^{(q),-\top}  \\
			&  \Id                      &                        & \vdots            \\
			&                           &                        & \IKM^{(q)}_{q-1,q} \IKMC^{(q),-\top}  \\
			0           & \cdots                     & 0                      & \Id
		\end{pmatrix}
		% \
		% \begin{pmatrix}
		% 	&                           &                        &  \IKM^{(q-1)}_{1,q-1} \IKMC^{(q-1),-\top} & 0  \\
		% 	&  \Id                      &                        & \vdots            & \vdots \\
		% 	&                           &                        & \IKM^{(q-1)}_{q-2,q-1} \IKMC^{(q-1),-\top}  & \vdots \\
		% 	0           & \cdots                     & 0                      & \Id                  &  0 \\
		% 	0           & \cdots                     & \cdots                  & 0                    & \Id
		% \end{pmatrix}
		\cdots
		\begin{pmatrix}
			\Id                 &            \IKM^{(2)}_{1,2} \IKMC^{(2),-\top}         &                          0             \\
			0                  &  \Id                                        & \vdots          \\
			0                  &                \cdots           &              \Id       &
		\end{pmatrix}
		\right)^{-1} \\
		&=
		\begin{pmatrix}
			\Id  & 0 & \cdots & \cdots & 0\\
			\IKMC^{( 2 ),-1} \IKM^{( 2 )}_{2,1} & \Id & \ddots & 0 &   \vdots\\
			\vdots & \IKMC^{(3),-1} \IKM^{(3)}_{3,2}   & \ddots & \ddots& \vdots \\
			\vdots & \vdots & \ddots & \Id & 0 \\
			\IKMC^{( q ),-1} \IKM^{(q)}_{q,1} & \IKMC^{( q ),-1} \IKM^{( q )}_{q,2} & \cdots &\IKMC^{( q ),-1} \IKM^{( q )}_{q,q-1}&\Id
		\end{pmatrix}^{-\top}.
	\end{align}}
	Here, we have used the formulae for the inverses and products of elementary lower-triangular matrices \cite[pp.150--151]{trefethen1997numerical},
	\begin{align}
		& \left( \Id + \left( 0 , \dots, 0, a_{k+1}, \dots , a_N \right)^{\top} \otimes \boldsymbol{e}_k \right)^{-1} = \Id - \left( 0 , \dots, 0, a_{k+1}, \dots , a_N \right)^{\top} \otimes \boldsymbol{e}_k , \\
		& \left( \Id + \left( 0 , \dots, 0, a_{k+1}, \dots , a_N \right)^{\top} \otimes \boldsymbol{e}_k \right)
		\left( \Id + \left( 0 , \dots, 0, b_{l+1}, \dots , b_N \right)^{\top} \otimes \boldsymbol{e}_l \right) \\
		& \quad = \Id + \left( 0 , \dots, 0, a_{k+1}, \dots , a_N \right)^{\top} \otimes \boldsymbol{e}_k + \left( 0 , \dots, 0, b_{l+1}, \dots , b_N \right)^{\top} \otimes \boldsymbol{e}_l ,
	\end{align}
	where $\boldsymbol{e}_k$ is the $k$\textsuperscript{th} standard Euclidean basis row vector, with $k<l$.
\end{proof}

\begin{proof}[Proof of \cref{lem-lowerBound}]
	\label{prv-lowerBound}
	We prove the result in the setting of \cref{examp-subsamp}, since the proof for \cref{examp-average} is similar.
	The inequality $\| \phi \|_{\ast}^2\geq \frac{1}{\| \IK \|} \| \phi \|_{H^{-s}(\Omega)}^2$ and \eqref{eqn-duality} imply that
	\begin{align}
		\left\| \phi \right\|_{\ast}^2
		& \geq \frac{1}{\| \IK \|} \sup_{v \in H_0^s( \Omega )} \sum_{ i \in \I^{(k)}} 2 \dualprod{ \x_i \phi_i }{ v } - \norm{ v }_{H_0^s( \Omega )}^2 \\
		& \geq \frac{1}{\| \IK \|} \sum_{i \in \I^{(k)}} \sup_{v \in H_0^s\left(B_{(\delta/2) h^k }( \xloc_i )\right)} 2 \dualprod{ \x_i \phi_i }{ v } - \norm{ v }_{H_0^s\left(B_{(\delta/2) h^k }( \xloc_i )\right)}^2 \\
		&= \frac{1}{\| \IK \|} \sum_{i \in \I^{(k)}} \left|\x_i\right|^2 \norm{ \phi_i }_{H^{-s}\left(B_{(\delta/2) h^k } ( \xloc_i )\right)}^2\\
		& \geq \frac{1}{\| \IK \|} \absval{ \x }^2 \inf_i \norm{ \phi_i }_{H^{-s}\left(B_{(\delta/2) h^k } ( \xloc_i )\right)}^2\,.
	\end{align}
	The identity $\norm{ \phi_i }_{H^{-s}\left(B_{(\delta/2) h^k }^2 ( \xloc_i )\right)}^2=h^{2sk} (\delta/2)^{2s-d} \|\boldsymbol{\delta}(\quark-0)\|_{H^{-s}\left(B_{1}(0)\right)}^2 $ concludes the proof with $C_{\Phi} \defeq \|\IK\| (\delta/2)^{d-2s} \left\| \boldsymbol{\delta}(\quark-0) \right\|_{H^{-s} \left(B_{1}(0)\right)}^{-1}$
	and $H \defeq h^s$.
\end{proof}

\begin{proof}[Proof of \cref{lem-condForW} in the case of \cref{examp-average}]
	Let $\zeta$ be a set of points such that $\left\{B_{\rho h^k}(z)\right\}_{z \in \zeta}$ covers $\Omega$, and such that $\sup_{x \in \Omega} \# \left\{ z \in \zeta: x \in B_{2 \rho h^k}(z) \right\} \leq C(d)$.
	For $i \in \J^{(l)}$ and $z \in \zeta$, we write $i \leadsto z$ if $z$ is the element of $\zeta$ closest to $i$ (using an arbitrary way to break ties). For $1 \leq k < l \leq q$, $\phi \defeq \sum_{i \in \J^{(l)}} \x_i \phi_i$, $\varphi \defeq \sum_{ i \in \J^{(l)}} \x_i \varphi_i$ and $\varphi_i \defeq \sum_{j \in \I^{(k)}} w_{ij} \phi_j^{(k)}$ we have
	\begin{equation}
		\|\phi - \varphi\|_{H^{-s}(\Omega)}^2 = \sup_{v \in H_0^s(\Omega)} \left( \sum_{z \in \zeta} \sum_{i \leadsto z} \int_{B_{2\rho h^k}(z)} 2 \x_i ( \phi_i - \varphi_i ) v(x) \dx \right) - \|v\|_{H_0^s(\Omega)}^2.
	\end{equation}
	The Bramble--Hilbert lemma \cite{dekel2004bramble_supplement} and the vanishing moment property \eqref{eqjkgfkhkjhgh} of $\phi_i - \varphi_i$ yield that
	\begin{align}
		& \sum_{i \leadsto z} \int_{B_{2 \rho h^k}(z)} 2 \x_i ( \phi_i - \varphi_i )v(x) \dx \\
		& \quad \leq 2 \left( 2 \rho h^k\right)^s \left \| \sum_{i \leadsto z} \x_i \left(\phi_i - \varphi_i\right) \right\|_{L^2(B_{2 \rho h^k}(z))} \|D^sv\|_{L^2(B_{2 \rho h^k}(z))}  \\
		& \quad \leq 2 C \left(2\rho h^k\right)^{2s} \left \| \sum_{i \leadsto z} \x_i \left(\phi_i - \varphi_i\right) \right\|_{L^2(B_{2 \rho h^k}(z))}^2 + \frac{\|D^sv\|^2_{L^2(B_{2 \rho h^k}(z))}}{2 C}.
	\end{align}
	Summing over all $z \in \zeta$ and choosing the constant $C$ appropriately yields
	\begin{equation}
		\|\phi - \varphi\|_{H^{-s}(\Omega)}^2 \leq C \rho^{2s} h^{2ks} \sum_{z \in \zeta} \left\| \sum_{i \leadsto z} \alpha_i \left( \phi_i - \varphi_i\right) \right\|_{L^2(B_{2 \rho h^k}(z))}^2 .
	\end{equation}
	Since the $\phi_i$ are $L^2$-orthogonal to each other and $\|\phi_i\|_{L^2}^2\leq C$,
	{\small
	\begin{align}
		\sum_{z \in \zeta} \left\| \sum_{i \leadsto z} \alpha_i \left( \phi_i - \varphi_i\right) \right\|_{L^2(B_{2 \rho h^k}(z))}^2
		&\leq 2
		\sum_{z \in \zeta} \left( \left[\sum_{i \leadsto z} \alpha_i^2 \|\phi_i\|_{L^2}^2\right] + \left\| \sum_{i \leadsto z} \alpha_i \varphi_i \right\|_{L^2(B_{2 \rho h^k}(z))}^2\right)\\
		&\leq C \left( |\x|^2 + \sum_{z \in \zeta} \left\| \sum_{i \leadsto z} \alpha_i \varphi_i \right\|_{L^2}^2\right) .
	\end{align}}
	Inserting the definition of the $\varphi_i$ yields
	\begin{align}
		\sum_{z \in \zeta} \left\| \sum_{i \leadsto z} \x_i \varphi_i \right\|_{L^2}^2
		= \sum_{z \in \zeta} \left\| \sum_{j \in \tilde{I}^{(k)}} \sum_{i \leadsto z} \alpha_i w_{ij} \phi_j^{(k)} \right\|_{L^2}^2
		&\leq \sum_{z \in \zeta} \sum_{j \in \tilde{I}^{(k)}} \left( \sum_{i \leadsto z} \alpha_i w_{ij}\right)^2 \left\| \phi_j^{(k)} \right\|_{L^2}^2 \\
		&\leq C h^{-kd} \sum_{z \in \zeta} \sum_{j \in \tilde{I}^{(k)}} \left( \sum_{i \leadsto z} |\alpha_i| |w_{ij}|\right)^2 .
	\end{align}
	We will now use the fact that on $\Reals^n$, we have the norm inequalities $n^{-1/2} |\quark|_1 \leq |\quark|_2 \leq |\quark|_1$.
	By a sphere-packing argument, for any $z \in \zeta$, we have $\left\{ i \in \J^{(l)} \,\middle|\, i \leadsto z \right\} \leq C(d) (\rho/\delta)^d h^{d(k-l)}$
	Thus, the number of summands in the innermost sum is at most $C(d) (\rho/\delta)^d h^{(k-l)d}$ and using the above norm inequalites, we obtain
	\begin{align}
		&h^{-kd} \sum_{z \in \zeta} \sum_{j \in \tilde{I}^{(k)}} \left( \sum_{i \leadsto z} |\alpha_i| |w_{ij}|\right)^2\\
		\leq& C (\rho/\delta)^d h^{-ld} \sum_{z \in \zeta} \sum_{i \leadsto z} \sum_{j \in \tilde{I}^{(k)}} \left( |\alpha_i| |w_{ij}|\right)^2
		\leq C (\rho/\delta)^d h^{-ld} \omega_{l,k}^2 |\x|^2.
	\end{align}
	Putting the above together yields the result.
\end{proof}

\begin{proof}[Proof of \cref{lem-decayInverse}]
	Define
	\begin{align*}
		R & \defeq \Id - \frac{2}{\|A\| + \|A^{-1}\|^{-1}} A, &
		r & \defeq \frac{1 - \frac{1}{\|A^{-1}\|\|A\|}}{1 + \frac{1}{\|A^{-1}\|\|A\|}} ,
	\end{align*}
	and observe that $\|R\| = r$.
	Since $A = \frac{\|A\| + \|A^{-1}\|^{-1}}{2} \left( \Id - R \right)$, it follows from a Neumann series argument that $ A^{-1} = \frac{2}{\|A\| + \|A^{-1}\|^{-1}} \sum_{k=0}^{\infty} R^{k}$.
	The positive definiteness of $A$ implies that
	\begin{equation*}
		| R_{i,j} | \leq \max\left\{1 , \frac{2 C}{\|A\| + \| A^{-1}\|^{-1}} \right\} \exp ( -\gamma d( i , j ) ) .
	\end{equation*}
	Let $C_{R} \defeq \max\left\{1 , \frac{2 C}{\|A\| + \| A^{-1}\|^{-1}} \right\}$.
	\cref{lem-decayProduct} implies that
	\begin{equation*}
		| R^{k}_{i,j} | \leq \left(c_{d}\left( \gamma/2 \right)\right)^{k-1} C_{R}^{k} \exp\left( -\frac{\gamma}{2} d( i , j ) \right) .
	\end{equation*}
	Combining the above estimates yields
	\begin{align}
		\frac{\|A\| + \|A^{-1}\|^{-1}}{2} \left| \left( A^{-1} \right)_{i,j} \right|
		& \leq \left( n + 1\right) \left( c_{d}\left( \gamma/2 \right) \right)^{n-1} C_{R}^{n} \exp\left( -\frac{\gamma}{2} d( i , j ) \right) + \frac{r^{n+1}}{1-r}\\
		& \leq \exp \left( \left(1 + \log\left( c_{d}\left( \gamma/2 \right) \right) + \log( C_{R} ) \right) n - \frac{\gamma}{2} d( i , j ) \right) \\
		&\phantom{=} \quad + \exp \left( -\log( 1-r ) + \log( r ) ( n+1 )\right) .
	\end{align}
	By choosing	
	\begin{equation}
		\label{eqkedhdkjd}
		\nu \defeq \frac{\frac{\gamma}{2}d( i , j ) - \log( 1-r )}{\left( 1 + \log\left( c_{d}\left( \gamma/2 \right) \right) +\log( C_{R} ) \right) - \log( r )} ,
	\end{equation}
	and $n + 1 \defeq \lceil \nu \rceil$, we obtain
	\begin{align}
		\label{eqkjhgjhggygubh}
		& \exp \left( \left(1 + \log\left( c_{d}\left( \gamma/2 \right) \right) + \log( C_{R} ) \right) n - \frac{\gamma}{2} d( i , j ) \right) + \exp \left( -\log( 1-r ) + \log( r ) ( n+1 )\right)\\
		& \quad \leq \exp \left( \left(1 + \log\left( c_{d}\left( \gamma/2 \right) \right) + \log( C_{R} ) \right) \nu - \frac{\gamma}{2} d( i , j ) \right) + \exp \left( -\log( 1-r ) + \log( r ) \nu \right)\\
		& \quad = 2\exp\left( \frac{ -\log( 1-r ) \left(1 + \log\left( c_{d}\left( \gamma/2 \right) \right) + \log( C_{R} ) \right) + \log( r ) \frac{\gamma}{2} d( i , j )} {\left(1 + \log\left( c_{d}\left( \gamma/2 \right) \right) + \log( C_{R} ) \right) - \log( r )} \right) .
	\end{align}
	This yields the upper bound
	{\small
	\begin{align}
		\left| \left( A^{-1} \right)_{i,j} \right|
		& \leq \frac{4 \cdot \exp\left( \frac{ -2 \log( 1-r ) \left(1 + \log\left( c_{d}\left( \gamma/2 \right) \right) + \log( C_{R} ) \right) +	\log( r ) \frac{\gamma}{2} d( i , j )} {\left(1 + \log\left( c_{d}\left( \gamma/2 \right) \right) + \log( C_{R} ) \right) - \log( r )} \right)}{\|A\| + \|A^{-1}\|^{-1}}  \\
		& = \frac{4}{\|A\| + \|A^{-1}\|^{-1}} \cdot \exp\left( \frac{\log( r ) }{\left(1 + \log\left( c_{d}\left( \gamma/2 \right) \right) + \log( C_{R} ) \right) - \log( r )}\frac{\gamma}{2} d( i , j )\right) \\
		\label{prf-decayInverse-term-to-optimise}
		& \phantom{=} \quad \cdot \exp \left( \frac{ -2\log( 1-r ) \left(1 + \log\left( c_{d}\left( \gamma/2 \right) \right) + \log( C_{R} ) \right) } {\left(1 + \log\left( c_{d}\left( \gamma/2 \right) \right) + \log( C_{R} ) \right) - \log( r )} \right) .
	\end{align}}
	Optimising the term on line \eqref{prf-decayInverse-term-to-optimise} over $\left( 1 + \log\left( c_{d}\left( \gamma/2 \right) \right)+ \log( C_{R} ) \right)$ yields
	\begin{align}
		\left| ( A^{-1} )_{i,j} \right| & \leq \frac{4}{\left(\|A\| + \|A^{-1}\|^{-1}\right)( 1-r )^{2}} \exp\left( \frac{\frac{\gamma}{2} d( i , j ) \log( r ) }{\left(1 + \log\left( c_{d}\left( \gamma/2 \right) \right) + \log( C_{R} ) \right) - \log( r )}\right) .
		% \qedhere
	\end{align}
\end{proof}

\begin{proof}[Proof of \cref{lem-decayCholesky}]
	\label{prv-decayCholesky}
	In this proof we will use the notation $k{:}l$ to denote the individual indices from $k$ to $l$, as opposed to matrix blocks.
	We will establish the result by showing that, for all $1\leq k\leq N$, the $k$\textsuperscript{th} column of $L$ (when considered as an element of
	$\Reals^{I\times I}$ by zero padding) satisfies the exponential decay stated in the lemma.
	Let $S^{( k )} \defeq B_{k:N , k:N} - B_{k:N , 1:k-1} ( B_{1:k-1, 1:k-1} )^{-1} B_{1:k-1, k:N}$.
	Then $L_{k:N,k} = S^{( k )}_{:,1} / \sqrt{S^{( k )}_{k,k}}$.
	\Cref{lem-blockChol2Scale} implies that $S^{( k )} = ( B_{ k:N , k:N } )^{-1}$, and hence \cref{lem-decayInverse} yields that
	{\small
	\begin{align}
		\left| \bigl( S^{(k)} \bigr)_{i,j} \right|
		\leq \frac{4}{\left(\|B\| + \|B^{-1}\|^{-1}\right)( 1-r )^{2}} \exp\left( \frac{\log( r ) \frac{\gamma}{2} d( i , j )}{1 + \log\left( c_{d}\left( \gamma/2 \right) \right) + \log( C_{R} ) - \log( r )}\right).
	\end{align}}
	Here we used the facts that the spectrum of $B_{k:n,k:n}$ is contained in $[ \lambda_{\min} (B), \lambda_{\max} (B) ]$ and that the right-hand side of the above estimate is increasing in $r$ and $C_{R}$.
	The estimate $S^{( k )}_{k,k} \geq \frac{1}{\|S^{( k ),-1} \|} \geq \frac{1}{\|B\|}$ completes the proof.
\end{proof}

\begin{proof}[Proof of \cref{lem-decayTriang}]
	\label{prv-decayTriang}
	For any matrix $T$ for which the Neumann series $\sum_{k=0}^\infty T^k$ converges in the operator norm, we have $\left( \Id - T \right)^{-1} = \sum_{k=0}^\infty T^{k}$.
	Therefore, $L^{-1} = \sum_{k=0}^\infty ( \Id - L )^k$ if the right-hand side series is convergent.
	Since $ \Id - L $ has the block-lower-triangular structure
	\begin{equation}
		( \Id - L ) =
		\begin{pmatrix}
			0         & 0       & 0         & 0      \\
			-L_{2,1}  &\ddots   & 0         & 0      \\
			\vdots    &\ddots   & \ddots    & \vdots \\
			-L_{q,1}  &\dots    &-L_{q,q-1} & 0      \\
		\end{pmatrix},
	\end{equation}
	it follows that $ \Id - L $ is $q$-nilpotent, i.e.\ $( \Id - L )^{q} = 0$ and the Neuman series terminates after the first $q$ summands.
	Using this we will now show that the exponential decay of $L$ is preserved under inversion.
	To this end, consider the $(k,l)$\textsuperscript{th} block of $( \Id - L )^p$ and observe that
	\begin{align}
		\label{eqkjgghjgh}
		\left| \left(\left( ( \Id - L )^p \right)_{k,l} \right)_{ij} \right|
		& = \left| (-1)^p \sum_{k = s_1 > s_2 > \dots > s_p > s_{p+1} = l} \left(\prod_{m=1}^{p} L_{s_m,s_{m+1}}\right)_{ij} \right| \\
		& \leq \sum_{k = s_1 > s_2 > \dots > s_p > s_{p+1} = l} \left( c_d\left(\gamma/2\right) C \right)^p \exp\left( - \frac{\gamma}{2} d(i,j) \right) \\
		& = \binom{k-l-1}{ p-1} \left( c_d\left(\gamma/2\right) C \right)^p
		\exp\left( - \frac{\gamma}{2} d(i,j) \right) ,
	\end{align}
	where the inequality follows from \cref{lem-decayProduct}.
	Summing \eqref{eqkjgghjgh} over $p$, we obtain, for $i \neq j$,
	\begin{align}
		\left| \left( L^{-1}_{k,l} \right)_{ij} \right|
		& \leq \sum_{p=1}^{k-l-1} \binom{k-l-1 }{ p-1} \left( c_d\left(\gamma/2\right) C \right)^p \exp\left( - \frac{\gamma}{2} d(i,j) \right)\\
		& \leq \left( 1+c_d\left(\gamma/2\right) C \right)^{k-l} \exp\left( - \frac{\gamma}{2} d(i,j) \right) \\
		& \leq 2^q \left( c_d\left(\gamma/2\right) C \right)^{q} \exp\left( - \frac{\gamma}{2} d(i,j) \right),
	\end{align}
	which concludes the proof of the lemma.
\end{proof}

With the above results on the propagation of exponential decay we can now conclude the proof of \cref{thm-decayAbstractCholesky}.

\begin{proof}[Proof of \cref{thm-decayAbstractCholesky}]
	\label{prv-decayAbstractCholesky}
	Applying \cref{lem-decayInverse}, \cref{cond-spatloc}, and the condition number bound in \cref{cond-specloc} yields the following estimate for $\IKMC^{(k),-1}$:
	\begin{equation}
		\left| \bigl( \IKMC^{(k),-1} \bigr)_{ij} \right| \leq \frac{4 \exp\left(\frac{\log(r)} {\left(1 + \log\left( c_d \left( \gamma/2\right)\right) + \log\left(C_R\right) -\log( r )\right)} \frac{\gamma}{2} d( i, j ) \right)}{\left( \bigl\| \IKMC^{(k)} \bigr\| + \bigl\| \IKMC^{(k),-1} \bigr\|^{-1} \right)\left(1-r\right)^2} ,
	\end{equation}
	with $C_R = \max\left\{1,\frac{2C_\gamma \bigl\| \IKMC^{(k),-1} \bigr\|}{1 + \kappa}\right\}$ and $r = \frac{1- \kappa^{-1}}{1 + \kappa^{-1}}$.
	\Cref{lem-blockChol2Scale} and \cref{cond-specloc} yield
	\begin{align}
		\lambda_{\max} \bigl( \IKMC^{(k)} \bigr) & \leq \lambda_{\max} \bigl( \IKM^{(k)} \bigr) \leq C_{\Phi} H^{-2k} , \\
		\lambda_{\min} \bigl( \IKMC^{(k)} \bigr) & \geq \frac{1}{C_{\Phi}} H^{-2(k-1)} .
	\end{align}
	Using these estimates, we obtain
	\begin{equation}
		\left| \bigl( \IKMC^{(k),-1} \bigr)_{ij} \right| \leq \frac{2 C_{\Phi} H^{2(k-1)}}{\left(1-r\right)^2} \exp\left( -\tilde{\gamma}d( i, j ) \right),
	\end{equation}
	where $\tilde{C}_R = \max\left\{1,\frac{2C_\gamma C_{\Phi}}{1 + \kappa}\right\}$, $r = \frac{1- \kappa^{-1}}{1 + \kappa^{-1}}$ and $\tilde{\gamma} \defeq \frac{-\log(r)} {\left(1 + \log\left( c_d \left( \gamma/2\right)\right) + \log\left(\tilde{C}_R\right) - \log( r )\right)} \frac{\gamma}{2}$.
	Applying \cref{lem-decayProduct} to the products  $B^{(i),-1}A^{(i)}_{ij}$ appearing in the definition of $\bar{L}^{-1}$ in Lemma \cref{lem-BlockCholesky}, we obtain
	\begin{equation}
		\left| \bigl( \bar{L}^{-1} \bigr)_{ij} \right| \leq \frac{2 C_{\Phi} C_{\gamma} \left(c_d\left(\tilde{\gamma}/2\right)\right)^2} {\left(1-r\right)^2} \exp\left( -\frac{\tilde{\gamma}}{2}d( i, j ) \right).
	\end{equation}
	\Cref{lem-decayTriang} now yields the following decay bound for $\bar{L}$:
	\begin{equation}
		| \bar{L}_{ij} | \leq \left(4 c_d\left(\tilde{\gamma}/4\right)\frac{C_{\Phi} C_{\gamma} \left(c_d\left(\tilde{\gamma}/2\right)\right)^2} {\left(1-r\right)^2} \right)^q \exp\left( -\frac{\tilde{\gamma}}{4}d( i, j ) \right).
	\end{equation}
	For a positive-definite matrix $M$, let $\chol\left( M\right)$ denote its lower-triangular Cholesky factor and set $L^{(k)} \defeq \chol\left( \IKMC^{(k),-1}\right)$.
	Following the same procedure as in the bound of the decay of $\IKMC^{(k)}$ yields the decay bound
	\begin{equation}
		\left| L^{(k)}_{ij} \right| \leq \frac{2 C_{\Phi} H^{\left(k-1\right)}}{\left(1-r\right)^2} \exp\left( -\tilde{\gamma}d( i, j ) \right).
	\end{equation}
	Applying \cref{lem-decayProduct} to the product $\bar{L} \chol(D) = \chol(\KM)$ yields the decay bound
	\begin{align}
		\bigl| ( \chol( \KM ) )_{ij} \bigr|
		& \leq \frac{2 C_{\Phi} c_d\left(\tilde{\gamma}/8\right)^2}{\left(1-r\right)^2} \left(4 c_d\left(\tilde{\gamma}/4\right)\frac{C_{\Phi} C_{\gamma} \left(c_d\left(\tilde{\gamma}/2\right)\right)^2} {\left(1-r\right)^2} \right)^q \exp\left( -\frac{\tilde{\gamma}}{8}d( i, j ) \right).
		% \qedhere
	\end{align}
\end{proof}

\section{Proof of accuracy of incomplete Cholesky factorization in the supernodal multicolor ordering}
\label{apsec-superMulti}

We will now bound the approximation error of the Cholesky factors obtained from \cref{alg-ICholesky}, using the supernodal multicolor ordering and sparsity pattern described in \cref{const-superMulti}.
For $\ti, \tj \in \tI$, let $\KM_{\ti, \tj}$ be the submatrix $(\KM_{ij})_{i \in \ti, j \in \tj}$ and let $\sqrt{M}$ be the (dense and lower-triangular) Cholesky factor of a matrix $M$.

First observe that \cref{alg-ICholesky} with supernodal multicolor ordering $\prec_{\rho}$ and sparsity pattern $S_{\rho}$ is equivalent to the block-incomplete Cholesky factorization described in \cref{alg-superICholesky} where the function $\texttt{Restrict!}(\KM, S_{\rho})$ sets all entries of $\KM$ outside of $S_{\rho}$ to zero.

\bigskip

\begin{algorithm}[H]
	\textbf{Input:} $\KM \in \Reals^{I \times I}$ symmetric \\
  	\textbf{Output:} $L\in \Reals^{I \times I}$ lower triangular \\
  \begin{algorithmic}
	\setcounter{ALC@unique}{0}
	\STATE $\texttt{Restrict!}(\KM, S_{\rho})$
	\FOR{$\ti \in \tI$}
		\STATE $L_{:,\ti} \leftarrow \KM_{:,\ti} / \sqrt{\KM_{\ti,\ti}}^{\top}$\;
		\FOR{$\tj \tsucc \ti$ : $(\ti, \tj) \in \tS$}
			\FOR{$ k \tsucc \tj$ : $( \tk, \ti ), ( \tk, \tj ) \in \tS$}
				\STATE $\KM_{\tk,\tj} \leftarrow \KM_{\tk,\tj} - \KM_{\tk,\ti} (\KM_{\ti,\ti})^{-1} \KM_{\tj, \ti}$
			\ENDFOR
		\ENDFOR
	\ENDFOR
  \RETURN $L$
  \end{algorithmic}
	\caption{Supernodal incomplete Cholesky factorization}
	\label{alg-superICholesky}
\end{algorithm}

\bigskip

We will now reformulate the above algorithm using the fact that the elimination of nodes of the same color, on the same level of the hierarchy, happens consecutively.
Let $p$ be the maximal number of colors used on any level of the hierarchy.
We can then write $\I = \bigcup_{1 \leq k \leq q, 1 \leq l \leq p} \J^{(k,l)}$, where $\J^{(k,l)}$ is the set of indices on level $k$ colored in the color $l$.
Let $\KM_{( k, l ),( m, n )}$ be the restriction of $\KM$ to $\J^{(k,l)} \times \J^{(m,n)}$ and write $(m,n) \prec (k,l) \iff m < k \text{ or } (m = k \text{ and } n < l )$.
We can then rewrite \cref{alg-superICholesky} as

\bigskip

\begin{algorithm}[H]
	\textbf{Input:} $\KM \in \Reals^{I \times I}$ symmetric \\
	\textbf{Output:} $L\in \Reals^{I \times I}$ lower triangular \\
  \begin{algorithmic}
	\setcounter{ALC@unique}{0}
	  \FOR{$1 \leq k \leq q$}
	  	\FOR{$1 \leq l \leq p$}
	  		\STATE $\texttt{Restrict!}( \KM, S_{\rho})$
	  		\STATE $L_{( :, : ),( k, l )}
	  		\leftarrow \KM_{( :, : ),( k, l )}
	  		/ \sqrt{\KM_{( k, l ),( k, l )}}^{\top} $
	  		\STATE $ \KM
	  		\leftarrow \KM - \KM_{( :, : ), ( k, l )}
	  		\left(\KM_{( k, l ), ( k, l )}\right)^{-1}
	  		\KM_{( k, l ), ( :, : )}$\;
	  	\ENDFOR
	  \ENDFOR
	  %$\texttt{Restrict!}( \KM, S_{\rho})$\;
	  \RETURN $L$
  \end{algorithmic}
	\caption{Supernodal incomplete Cholesky factorization}
	\label{alg-superMultiICholesky}
\end{algorithm}

\bigskip

For $1 \leq k \leq q, 1 \leq l \leq p$ and a matrix $M \in \Reals^{\I \times \I}$
with
$M_{( :, : ), ( m, n )}, M_{( m, n ), ( :, : )} = 0$ for all $( m, n ) \prec ( k, l )$,
let $\cS\left[ M \right]$ be the matrix obtained by applying $\texttt{Restrict!}(M,S_{\rho})$ followed by the Schur complementation $M \leftarrow M - M_{( :, : ), ( k, l )} \bigl( M_{( k, l ), ( k, l )} \bigr)^{-1} M_{( k, l ), ( :, : )}$.
We now prove a stability estimate for the operator $\cS$.
Let $M_{k, ( m, n )}$ be the restriction of a matrix $M \in \Reals^{\I \times \I}$ to
$\J^{(k)} \times \J^{(m,n)}$.

\begin{lemma}
	\label{lem-pertSchur}
	For $1 \leq k^{\circ} \leq q$ and $1 \leq l^{\circ} \leq p$ let
	$\KM,E \in \Reals^{I \times I}$ be such that
	\begin{equation}
		\KM_{( :, : ), ( m, n )},
		\KM_{( m, n ), ( :, : )} = 0
		\text{ for all } ( m, n ) \prec ( k^\circ, l^\circ ),
	\end{equation}
	and (writing $\Theta_{k,l}$ for the $J^{(k)}\times J^{(l)}$ submatrix of $\Theta$ and $\lambda_{\max}$ for maximal singular values) define
	\begin{align}
		\lambda_{\min} & \defeq \lambda_{\min} ( \KM_{k^{\circ},k^{\circ}} ) , &
		\lambda_{\max} & \defeq \max_{k^{\circ} \leq k \leq q} \lambda_{\max} ( \KM_{k^{\circ},k} ) .
	\end{align}
	If
	\begin{equation}
		\max_{k^{\circ} \leq k,l \leq q} \| E_{k,l} \|_{\FRO} \leq \epsilon \leq
		\frac{\lambda_{\min}}{2} ,
	\end{equation}
	then the following perturbation estimate holds:
	\begin{equation}
		\max_{k^{\circ} \leq k,l \leq q} \bigl\| \left( \cS [ \KM ] - \cS[ \KM + E ] \bigr)_{k,l} \right\|_{\FRO} \leq \left( \frac{3}{2} + 2 \frac{\lambda_{\max}}{\lambda_{\min}} + 8 \frac{\lambda_{\max}^2}{\lambda_{\min}^2} \right) \epsilon .
	\end{equation}
\end{lemma}

\begin{proof}
	Write $\tKM$, $\tE$ for the versions of $\KM$, $E$ set to zero outside of $S_{\rho}$.
	For $k^{\circ} \leq k,l \leq q$,
	{\small
	\begin{align}
		& \left(\cS[ \KM + E ] - \cS[ \KM ]\right)_{k,l} \\
		& \quad =
		\tKM_{k,l} + \tE_{k,l} - \bigl( \tKM + \tE \bigr)_{k,(k^{\circ},l^{\circ})} \bigl( \tKM + \tE \bigr)_{(k^{\circ}, l^{\circ}),(k^{\circ}, l^{\circ})}^{-1}
		\bigl( \tKM + \tE \bigr)_{(k^{\circ},l^{\circ}),l} \\
		& \quad \phantom{=} \quad - \tKM_{k,l} + \tKM_{k,(k^{\circ},l^{\circ})} \tKM_{(k^{\circ},l^{\circ}),(k^{\circ},l^{\circ})}^{-1} \tKM_{(k^{\circ},l^{\circ}), l} \\
		& \quad = \tE_{k, l} + \bigl( \tKM + \tE \bigr)_{k,(k^{\circ},l^{\circ})} \bigl( \tKM + \tE \bigr)_{(k^{\circ},l^{\circ}),(k^{\circ},l^{\circ})}^{-1} \tE_{(k^{\circ},l^{\circ}),(k^{\circ},l^{\circ})} \tKM_{(k^{\circ},l^{\circ}),(k^{\circ},l^{\circ})}^{-1} \bigl( \tKM + \tE \bigr)_{(k^{\circ},l^{\circ}), l} \\
		& \quad \phantom{=} \quad - \bigl( \tKM + \tE \bigr)_{k,(k^{\circ},l^{\circ})}\tKM_{(k^{\circ},l^{\circ}),(k^{\circ},l^{\circ})}^{-1} \bigl( \tKM + \tE \bigr)_{(k^{\circ},l^{\circ}),l} + \tKM_{k,(k^{\circ},l^{\circ})} \tKM_{(k^{\circ},l^{\circ}),(k^{\circ},l^{\circ})}^{-1} \tKM_{(k^{\circ},l^{\circ}), l} \\
		& \quad = \tE_{k, l}
		+ \bigl( \tKM + \tE \bigr)_{k,(k^{\circ},l^{\circ})} \bigl( \tKM + \tE \bigr)_{(k^{\circ},l^{\circ}),(k^{\circ},l^{\circ})}^{-1} \tE_{(k^{\circ},l^{\circ}),(k^{\circ},l^{\circ})} \tKM_{(k^{\circ},l^{\circ}),(k^{\circ},l^{\circ})}^{-1} \bigl( \tKM + \tE \bigr)_{(k^{\circ},l^{\circ}), l} \\
		& \quad \phantom{=} \quad - \tE_{k, (k^{\circ},l^{\circ})} \tKM_{(k^{\circ},l^{\circ}),(k^{\circ},l^{\circ})}^{-1} \tKM_{(k^{\circ},l^{\circ}),l}
		- \tKM_{k, (k^{\circ},l^{\circ})} \tKM_{(k^{\circ},l^{\circ}),(k^{\circ},l^{\circ})}^{-1} \tE_{(k^{\circ},l^{\circ}),l} \\
		& \quad \phantom{=} \quad - \tE_{k, (k^{\circ},l^{\circ})} \tKM_{(k^{\circ},l^{\circ}),(k^{\circ},l^{\circ})}^{-1} \tE_{(k^{\circ},l^{\circ}),l},
	\end{align}}
	where the second equality follows from the matrix identity
	\begin{equation}
		( A + B )^{-1} = A^{-1} - \left( A + B \right)^{-1} B A^{-1}.
	\end{equation}
	Now recall that, for all $A \in \Reals^{n \times m}, B \in \Reals^{m \times s}$, $\|M\| \leq \|M\|_{\FRO}$ and $\|A B\|_{\FRO} \leq \|A\| \|B\|_{\FRO}$.
	Therefore, $\| ( A + E )^{-1} \| \leq 2/\lambda_{\min}$ and $\| A + E \| \leq 2 \lambda_{\max}$.
	Combining these estimates and using the triangle inequality yields
	\begin{align}
		& \bigl\| ( \cS [ A + E ] - \cS[ A ] )_{k,l} \bigr\|_{\FRO}\\
		& \quad \leq \| E_{k,l} \|_{\FRO} + 8 \frac{\lambda_{\max}^2}{\lambda_{\min}^2} \| E_{k^{\circ},k^{\circ}} \|_{\FRO} + \frac{\lambda_{\max}}{\lambda_{\min}} ( \|E_{k,l}\|_{\FRO} + \|E_{l,k}\|_{\FRO} )\\
		& \quad + \lambda_{\min}^{-1} \| E_{k,k^{\circ}} \|_{\FRO} \| E_{k^{\circ},l} \|_{\FRO} \\
		& \quad \leq \left( 1 + 8 \frac{\lambda_{\max}^2}{\lambda_{\min}^2} + 2 \frac{\lambda_{\max}}{\lambda_{\min}} + \frac{\epsilon}{\lambda_{\min}}\right) \epsilon \\
		& \quad \leq \left( \frac{3}{2} + 2 \frac{\lambda_{\max}}{\lambda_{\min}} + 8 \frac{\lambda_{\max}^2}{\lambda_{\min}^2} \right) \epsilon .
		% \qedhere
	\end{align}
\end{proof}

Recursive application of the above lemma gives a stability result for the incomplete Cholesky factorization.

\begin{lemma}
	\label{lem-stabICholesky}
	For $\rho > 0$, let $\prec_{\rho}$ and $S_{\rho}$ be a supernodal ordering and sparsity pattern such that the maximal number of colors used on each level is at most $p$.
	Let $L^{S_\rho}$ be an invertible lower-triangular matrix with nonzero pattern $S_{\rho}$ and define $M \defeq L^{S_{\rho}} L^{S_{\rho},\top}$.
	Assume that $M$ satisfies \cref{cond-specloc} with constant $\kappa$.
	Then there exists a universal constant $C$ such that, for all $0 < \epsilon< \frac{\lambda_{\min}( M )}{2 q^2 ( C\kappa )^{2q p}}$ and all $E \in \Reals^{\I \times \I}$ with $\|E\|_{\FRO} \leq \epsilon$,
	\begin{equation}
		\bigl\| M - \tL^{S_{\rho}} \tL^{S_{\rho}, \top} \bigr\|_{\FRO} \leq q^2 ( C \kappa )^{2qp} \epsilon ,
	\end{equation}
	where $\tL^{(S_\rho)}$ is the Cholesky factor obtained by applying \cref{alg-superMultiICholesky} to $M + E$.
\end{lemma}

\begin{proof}
	The result follows from applying \cref{lem-pertSchur} at each step of \cref{alg-superMultiICholesky}.
\end{proof}

We can prove \cref{thm-accuracyICholesky} by using the stability result obtained above.

\begin{proof}[Proof of \cref{thm-accuracyICholesky}]
	\Cref{thm-decayCholesky} implies that by choosing $\rho \geq \tilde{C} \log(N/\epsilon)$ there exists a lower-triangular matrix $\tL^{S_\rho}$ with sparsity pattern $S_{\rho}$ such that
	$\bigl\|\KM - \tL^{S_{\rho}} \tL^{S_{\rho}, \top} \bigr\|_{\FRO} \leq \epsilon$.
	\Cref{thm-specloc} implies that the \cref{examp-subsamp} and \cref{examp-average} satisfy $\lambda_{\min} \geq 1/ \poly(N)$.
	Therefore, choosing $\rho \geq \tilde{C} \log N$ ensures that $\epsilon < \frac{\lambda_{\min} ( \KM )}{2}$ and thus that $\tKM \defeq \tL^{S_{\rho}} \tL^{S_{\rho}, \top}$ satisfies
	\Cref{cond-specloc} with constant $2C_{\Phi}$, where $C_{\Phi}$ is the corresponding constant for $\KM$.
	By possibly changing the constant $\tilde{C}$ again, $\rho \geq \tilde{C} \log N$	also ensures that
	\[
		\epsilon \leq \frac{\lambda_{\min}( \KM )}{2 q^2 \bigl( C \kappa \bigl( \tilde{\KM} \bigr) \bigr)^{2 q p}} ,
	\]
	where $C$ is the constant of \cref{lem-stabICholesky}, since
	$q \approx \log N$ and, by \cref{lem-constSuperMulti}, $p$ is bounded independently of $N$.
	Thus, by \cref{lem-stabICholesky}, the Cholesky factor $L^{S_{\rho}}$ obtained from applying
	\cref{alg-superMultiICholesky} to $\KM = \tKM + \bigl( \KM - \tKM \bigr)$
	satisfies
	\begin{equation}
		\left\| \tKM - L^{S_{\rho}} L^{S_\rho,\top} \right\|_{\FRO} \leq q^2 \left( 4 C \kappa \right)^{2 q p} \epsilon \leq \poly(N) \epsilon,
	\end{equation}
	where $\kappa$ is the constant with which $\KM$ satisfies \cref{cond-specloc} and the polynomial depends only on $C$, $\kappa$, and $p$.
	Since, for the ordering $\prec_{\rho}$ and sparsity pattern $S_{\rho}$, the Cholesky factors obtained via \cref{alg-ICholesky,alg-superMultiICholesky} coincide, we obtain the result.
\end{proof}

\end{document}

%% file: ex_shared.tex
\usepackage{amsmath, amssymb, color, url, supertabular, bbm}

\usepackage{enumitem}
\setlist[enumerate]{topsep=0pt,itemsep=-1ex,partopsep=1ex,parsep=1ex}
\setlist[itemize]{topsep=0pt,itemsep=-1ex,partopsep=1ex,parsep=1ex}

\allowdisplaybreaks

\usepackage{mathtools}
\mathtoolsset{showonlyrefs} %% This prevents LaTeX from printing out equation numbers if the numbered equation is not actually referenced in the text.

\makeatletter
\renewcommand{\paragraph}{%
  \@startsection{paragraph}{4}%
  {\z@}{1.0ex \@plus .5ex \@minus .2ex}{-.7em}%
  {\normalfont\normalsize\bfseries}%
}
\makeatother

\numberwithin{equation}{section}
\numberwithin{figure}{section}
\numberwithin{table}{section}

\usepackage{lipsum}
\usepackage{amsfonts}
\usepackage{graphicx}
\usepackage{epstopdf}
\usepackage{algorithmic}
\usepackage{booktabs}

\ifpdf
  \DeclareGraphicsExtensions{.eps,.pdf,.png,.jpg}
\else
  \DeclareGraphicsExtensions{.eps}
\fi

\newcommand*{\arXiv}[1]{\bgroup\color{blue}\href{https://arxiv.org/abs/#1}{arXiv:#1}\egroup}

\newcommand*{\BigO}{\mathcal{O}}
\newcommand*{\defeq}{\coloneqq}

\newcommand*{\Reals}{\mathbb{R}}
\newcommand*{\Naturals}{\mathbb{N}}
\DeclareMathOperator{\logdet}{log\,det}
\newcommand*{\one}{\mathbf{1}}
\newcommand*{\dx}{\,\mathrm{d}x}

\newcommand*{\dz}{\delta_{z}}

\DeclareMathOperator{\supp}{supp}
\newcommand*{\norm}[1]{\Vert #1 \Vert}
\newcommand*{\bignorm}[1]{\bigl\Vert #1 \bigr\Vert}
\newcommand*{\dualprod}[2]{[ #1 , #2 ]}
\newcommand*{\bigdualprod}[2]{\bigl[ #1 , #2 \bigr]}
\newcommand*{\innerprod}[2]{\langle #1 , #2 \rangle}
\newcommand*{\biginnerprod}[2]{\bigl\langle #1 , #2 \bigr\rangle}
\DeclareMathOperator{\dist}{dist}
\DeclareMathOperator{\disttt}{\mathtt{dist}}
\newcommand*{\absval}[1]{\vert #1 \vert}
\newcommand*{\Absval}[1]{\left\vert #1 \right\vert}
\newcommand*{\bigabsval}[1]{\bigl\vert #1 \bigr\vert}
\newcommand*{\quark}{\setbox0\hbox{$x$}\hbox to\wd0{\hss$\cdot$\hss}}
\newcommand*{\argmax}{\mathop{\textup{arg\,max}}}
\newcommand*{\argmin}{\mathop{\textup{arg\,min}}}
\newcommand*{\Expect}{\mathbb{E}}
\newcommand*{\Cov}{\operatorname{Cov}}
\DeclareMathOperator{\spn}{span}
\newcommand*{\Kernel}{\mathop{\textup{Ker}}}
\newcommand*{\Image}{\mathop{\textup{Im}}}

\newcommand*{\GMat}{G^{\text{Mat{\'e}rn}}}
\newcommand*{\GCauchy}{G^{\text{Cauchy}}}
\DeclareMathOperator{\diam}{diam}
\newcommand*{\poly}{\operatorname{poly}}
\DeclareMathOperator{\cond}{cond}
\DeclareMathOperator{\loc}{loc}
\renewcommand*{\H}{\mathcal{H}}

\newcommand*{\Id}{\textup{Id}}
\newcommand*{\K}{G}
\newcommand*{\IK}{\mathcal{L}}
\newcommand*{\IKM}{A}
\newcommand*{\IKMC}{B}
\newcommand*{\KM}{\Theta}
\newcommand*{\tKM}{\tilde{\Theta}}
\newcommand*{\I}{I}
\newcommand*{\J}{J}
\newcommand*{\B}{\mathcal{B}}
\newcommand*{\N}{\mathcal{N}}
\newcommand*{\ICH}{\texttt{ICHOL(0)}}
\newcommand*{\FRO}{\operatorname{Fro}}
\newcommand*{\V}{\mathfrak{V}}
\newcommand*{\W}{\mathfrak{W}}
\newcommand*{\unif}{\mathop{\textup{UNIF}}}
\newcommand*{\ti}{\tilde{i}}
\newcommand*{\tI}{\tilde{\I}}
\newcommand*{\rP}{P^{\updownarrow}}
\DeclareMathOperator{\chol}{chol}
\newcommand*{\tj}{\tilde{j}}
\newcommand*{\tJ}{\tilde{\J}}
\newcommand*{\tS}{\tilde{S}}
\newcommand*{\x}{\alpha}
\newcommand*{\y}{\beta}
\newcommand*{\rD}{\mathrm{D}}
\newcommand*{\bI}{\bar{\I}}
\newcommand*{\bJ}{\bar{\J}}
\newcommand*{\nnz}{\mathop{\textup{nnz}}}
\newcommand*{\rank}{\mathop{\textup{rank}}}
\newcommand*{\tsos}{t_{\texttt{SortSparse}}}
\newcommand*{\tent}{t_{\texttt{Entries}}}
\newcommand*{\tich}{t_{\texttt{ICHOL(0)}}}

\newcommand*{\intd}{\tilde{d}}
\newcommand*{\xloc}{x}
\newcommand*{\cS}{\mathbb{S}}

\newcommand*{\tsucc}{\mathop{\tilde{\succ}}}
\newcommand*{\tk}{\tilde{k}}
\newcommand*{\tE}{\tilde{E}}
\newcommand*{\tL}{\tilde{L}}

\newsiamremark{remark}{Remark}
\newsiamremark{hypothesis}{Hypothesis}
\crefname{hypothesis}{Hypothesis}{Hypotheses}
\newsiamthm{claim}{Claim}
\newsiamthm{example}{Example}
\crefname{example}{Example}{Examples}
\newsiamthm{condition}{Condition}
\crefname{condition}{Condition}{Conditions}
\newsiamthm{construction}{Construction}
\crefname{construction}{Construction}{Constructions}
\Crefname{ALC@unique}{Line}{Lines}

\headers{Inversion of dense kernel matrices
}{Florian Sch{\"a}fer, T. J. Sullivan, and Houman Owhadi}

\title{Compression, inversion, and approximate PCA of dense kernel matrices at near-linear computational complexity 
}

\author{%
  Florian\ Sch{\"a}fer\thanks{California Institute of Technology, MC 305-16, 1200 East California Boulevard, Pasadena, CA~91125, USA, \newline \email{florian.schaefer@caltech.edu}, Phone: (626) 395-3531,  Fax: (626) 578-0124, \newline Corresponding Author} 
  \and
  T. J. Sullivan\thanks{Mathematics Institute and School of Engineering, The University of Warwick, Coventry, CV4~7AL, UK, \email{t.j.sullivan@warwick.ac.uk}; and Zuse Institute Berlin, Takustra{\ss}e 7, 14195 Berlin, Germany, \email{sullivan@zib.de}}
  \and
  Houman Owhadi\thanks{California Institute of Technology}
}

\usepackage{amsopn}

%% file: sec-introduction.tex
\subsection{Dense kernel matrices and the \texorpdfstring{$N^3$-}{cubic }bottleneck}
\label{ssec-cubic-bottleneck}

Kernel matrices, i.e.\ square matrices $\KM$ of the form
\begin{equation}
	\label{eq-kernel-matrix}
	\KM_{ij} \defeq \K ( x_i, x_j ) ,
\end{equation}
obtained from pointwise evaluation of a symmetric positive-definite kernel $\K$ at a collection of points $\{ x_i \}_{i \in \I}$ in a domain $\Omega \subset \Reals^d$, play an important role in statistics, machine learning, and scientific computing.
In statistics, they are used as covariance matrices of Gaussian process priors.
In machine learning, they equip the feature space with a meaningful inner product via the \emph{kernel trick} \cite{hofmann2008kernel}.
In scientific computing, they appear as Green's functions (i.e.\ fundamental solutions) of linear elliptic partial differential equations (PDEs).

For all these applications, it is usually necessary to perform some or all of the following
tasks:
\begin{enumerate}[label=(\arabic*)]
	\item compute $v \mapsto \KM v$, given $v \in \Reals^{\I}$;
	\item compute $v \mapsto \KM^{-1}v$, given $v \in \Reals^{\I}$;
	\item compute $\logdet \KM$;
	\item sample from the normal/Gaussian distribution $\N(0,\KM)$;
	\item approximate eigenspaces corresponding to the leading eigenvalues of $\KM$.
\end{enumerate}

The first four of these tasks can be performed by computing the Cholesky factorization of $\KM$ (i.e.\ the decomposition $\KM = L L^T$ where $L$ is lower triangular).
For many popular covariance functions, most notably those of smooth random processes, the matrices $\KM$ will be \emph{dense}.
For large $N \defeq \# \I$ this results in a computational complexity of
$\BigO(N^3)$ for the Cholesky factorization and a complexity of $\BigO(N^2)$ to even store the
matrix.
When $\KM$ is \emph{sparse}, i.e.\ has relatively few non-zero entries, better complexity can be achieved --- the obvious limiting case being $\BigO(N)$ (i.e.\ linear) complexity if $\KM$ is diagonal.
However, for practical problems, the cubic scaling restricts dense Cholesky factorization to problems with $N \lessapprox 10^5$.
The breadth of kernel matrices' uses means that there is correspondingly high interest in achieving approximate Cholesky factorization of $\KM$ at linear or near-linear cost.

\subsection{Existing approaches}

Many fast methods are available for approximating dense kernel matrices and their applicability depends on specific assumptions made on $\KM$.
If the precision matrix $\KM^{-1}$ is sparse and can be approximated directly (e.g.\ by discretizing a PDE), then sparse linear solvers can be used.
These include multigrid solvers \cite{fedorenko1961relaxation,brandt1977multi,hackbusch1978fast,hackbusch2013multi} and sparse Cholesky factorization methods with nested dissection ordering \cite{george1989evolution,george1973nested,lipton1979generalized,gilbert1987analysis}.
This approach has been proposed for problems arising in spatial statistics \cite{lindgren2011explicit,roininen2011correlation,roininen2013constructing,roininen2014whittle}.
In other situations, available methods directly approximate the covariance matrix based on low-rank approximations, sparsity, and hierarchy.
Low-rank techniques such as the Nystr{\"o}m approximation \cite{williams2001using,smola2001sparse,fowlkes2004spectral} or rank-revealing Cholesky factorization \cite{bach2002kernel,fine2001efficient} seek to approximate $\KM$ by low-rank matrices whereas sparsity-based methods like
\emph{covariance tapering} \cite{furrer2012covariance} seek to approximate $\KM$ with a sparse matrix by setting entries corresponding to long-range interactions to zero.
These two approximations can also be combined to obtain \emph{sparse low-rank}
approximations \cite{sang2012full,quinonero2005unifying,schwaighofer2002transductive,
banerjee2008gaussian,snelson2005sparse}, which can be interpreted as imposing a particular graphical structure on the Gaussian process.
When $\KM$ is neither sufficiently sparse nor of sufficiently low rank, these approaches can be implemented in a hierarchical manner.
For low-rank methods, this leads to \emph{hierarchical} ($\mathcal{H}$- and $\mathcal{H}^2$-) matrices \cite{hackbusch2000sparse,hackbusch1999sparse,hackbusch2002data}, \emph{hierarchical off-diagonal low rank} (HODLR) matrices \cite{ambikasaran2013mathcal,ambikasaran2014fast}, and hierarchically semiseparable (HSS) matrices \cite{chandrasekaran2004fast,xia2010fast,li2012new}
that rely on computing low-rank approximations of sub-blocks of $\KM$ corresponding to far-field interactions on different scales. 
The interpolative factorization developed by \cite{ho2016hierarchical} combines hierarchical low-rank structure with the sparsity obtained from an elimination ordering of nested-dissection type.
Hierarchical low-rank structure was originally developed as an algebraic abstraction of the fast multipole method of \cite{greengard1987fast}.
In order to construct hierarchical low-rank approximations from entries of the kernel matrix efficiently, both deterministic and randomized algorithms have been proposed \cite{bebendorf2003adaptive,martinsson2016compressing}.
For many popular covariance functions, including Green's functions of elliptic PDEs \cite{bebendorf2003existence}, 
hierarchical matrices allow for (near-)linear-in-$N$ complexity algorithms for the inversion and approximation of $\KM$, at exponential accuracy.
Wavelet-based methods \cite{beylkin1991fast,gines1998lu}, using the separation and truncation of  interactions on different scales, can be seen as a hierarchical application of sparse approximation approaches.
The resulting algorithms have near-linear computational complexity and rigorous error bounds for asymptotically smooth covariance functions.
\cite{feischl2018sparse} use operator-adapted wavelets to compress the expected solution operators of random elliptic PDEs.
In \cite{katzfuss2016multi}, although no rigorous accuracy estimates are provided, the authors establish the near-linear computational complexity of algorithms resulting from the multi-scale generalization of probabilistically motivated sparse and low-rank approximations \cite{sang2012full,quinonero2005unifying,schwaighofer2002transductive,banerjee2008gaussian,snelson2005sparse}.

\subsection{Our main result and and overview of the paper}

Our main result is to show that a small modification of the Cholesky factorization algorithm is both accurate and scalable, when applied to kernel matrices obtained from kernels $\K$ identified as Green's functions of elliptic PDEs and a (roughly) homogeneously distributed cloud of points.
Such kernels are oftentimes used as covariance functions of smooth Gaussian processes (to enforce a smoothness prior on the function to be recovered/interpolated) and
therefore a large class of popular kernels fall into this category.
The cheap, accurate, approximate Cholesky factors provided by our method thereby serve tasks (1--4) from \cref{ssec-cubic-bottleneck}.
We furthermore show that by reversing the elimination order we obtain a fast direct solver for elliptic PDEs.

Contrary to the present belief that fast solvers for elliptic integral operators require the use of hierarchical low-rank structure or wavelets with a high order of vanishing moments, we show that state-of-the-art performance can be obtained just by zero fill-in Cholesky factorization (which just amounts to skipping some steps in the Cholesky factorization algorithm --- wavelets are only used in the detailed rigorous analysis of the algorithm). 
While there is a huge literature on the sparse Cholesky factorization of \emph{sparse} matrices, we are not aware of any prior literature on the sparse Cholesky factorization of \emph{dense} matrices.

For elliptic PDEs with arbitrary $L^{\infty}$-coefficients, $\mathcal{H}$-matrices can be used to compute $\epsilon$-approximate Cholesky factors of both differential and integral operators in computational complexity $\mathcal{O} \left(N \log^2\left(N\right) \log^{2d+2}\left(\epsilon^{-1} \right) \right)$ \cite{bebendorf2003existence,hackbusch1999sparse,hackbusch2000sparse,bebendorf2008hierarchical}.
$\mathcal{H}^2$-matrices can improve these complexities to $\mathcal{O} \left(N \log\left(N\right) \log^{2d + 2}\left(\epsilon^{-1} \right) \right)$ \cite{hackbusch2002data,borm2010approximation,borm2010efficient}.
The ``fast gamblet transform'' of \cite{owhadi2015multigrid,owhadi2017universal} can invert stiffness matrices of arbitrary elliptic operators in computational complexity $\mathcal{O}\left(\log^{2d + 1}\left(\epsilon^{-1} \right) \right)$.
Our computational complexities of $\mathcal{O} \left(N \log^{2d}\left(N/\epsilon \right) \right)$ for the Cholesky factorization of differential operators and $\mathcal{O} \left(N \log^{2}(N) \log^{2d}\left(N/\epsilon \right) \right)$ for the Cholesky factorization of integral operators improve upon the state of the art while using a much simpler algorithm.

Our method relies upon a cleverly-constructed elimination ordering and sparsity pattern, which we use in the incomplete Cholesky factorization of the matrix $\KM$.
Simplified versions of these constructions are given in \cref{ssec-simpleAlg};
\Cref{ssec-justification} gives a overview, without detailed proof, of why the method yields the desired results.
In particular, \cref{ssec-sparse-approximate-PCA} shows how the method provides a sparse approximate principal component analysis (PCA), thereby serving task (5).

\Cref{sec:numerics} presents detailed numerical experiments that illustrate the power of our method, and \cref{sec:analysis} gives the mathematical proofs of correctness and accuracy vs.\ complexity.
\Cref{sec:conclusions} contains concluding remarks, and some technical results are deferred to an Appendix.

%% file: sec-overview.tex
In this introductory section we give a brief overview of the setting in which our theoretical results apply (the class of kernels associated to elliptic operators) and highlight its main features.
All detailed numerical experiments and analysis will be deferred to \cref{sec:numerics,sec:analysis} respectively.

\subsection{The class of elliptic operators}
\label{sssec-classElliptic}

In order to establish rigorous, a priori, complexity-vs.-accuracy estimates in \cref{sec:analysis} we will assume that $\K$ is the Green's function of an elliptic operator $\IK$ of order $2s>d$ ($s,d \in \Naturals$), defined on a bounded domain $\Omega \subset \Reals^d$ with Lipschitz boundary, and acting on $H^s_0(\Omega)$, the Sobolev space of (zero boundary value) functions having derivatives of order $s$ in $L^2(\Omega)$.
More precisely, writing $H^{-s}(\Omega)$ for the dual space of $H^s_0(\Omega)$ with respect to the $L^2(\Omega)$ scalar product, our rigorous estimates will be stated for an arbitrary linear bijection
\begin{equation}
	\label{eq-op-between-sobolev}
	\IK \colon H_0^s( \Omega ) \to H^{-s}( \Omega )
\end{equation}
that is \emph{symmetric} (i.e.\ $\int_{\Omega} u \IK v \dx = \int_{\Omega} v \IK u \dx$), \emph{positive} (i.e.\ $\int_{\Omega} u \IK u \dx \geq 0$), and \emph{local} in the sense that
\begin{equation}
	\int_{\Omega} u \IK v \dx = 0 \text{ for all $u,v \in H_0^s( \Omega )$ such that $\supp u \cap \supp v = \emptyset$.}
\end{equation}
Let $\norm{ \IK } \defeq \sup_{u\in H_0^s} \norm{ \IK u }_{H^{-s}}/ \norm{ u }_{H^s_0}$ and $\norm{ \IK^{-1} } \defeq \sup_{f\in H^{-s}} \norm{ \IK^{-1} f }_{H^{s}_0}/ \norm{ f }_{H^{-s}}$ denote the operator norms of $\IK$ and $\IK^{-1}$.
The complexity and accuracy estimates for our algorithm will depend on (and only on) $d,s,\Omega, \norm{ \IK }$, $\norm{ \IK^{-1} }$, and the parameter
\begin{equation}
	\label{eq-homogeneity-param}
	\delta \defeq
	\frac{\min_{i \neq j \in \I}
	\dist\left(x_i, \{x_j\} \cup \partial \Omega \right)}
	{\max_{x \in \Omega} \dist\left(x, \{ x_i \}_{i \in \I}
	\cup \partial \Omega \right)}
\end{equation}
which is a measure of the homogeneity of the distribution of the cloud of points $x_i$.

Since our algorithm only requires the locations of the points $x_i$ and is oblivious to the exact knowledge of $\K$, for our numerical experiments in \cref{sec:numerics} we will consider \eqref{eq-op-between-sobolev}, general elliptic operators with or without boundary conditions (these include Mat\'ern kernels with fractional values of $s$) and exponential kernels.

\subsection{Zero fill-in incomplete Cholesky factorization (\ICH)}

\begin{figure}[t]
	\begin{minipage}[t]{0.40\textwidth}
		\vspace{0pt}
		\begin{algorithm}[H]
			\textbf{Input:} $A \in \Reals^{N \times N}$ symmetric\\
			\textbf{Output:} $L\in \Reals^{N \times N}$ lower triang. \\
			% \vspace{1.8em}
			\begin{algorithmic}[1]
			\setcounter{ALC@unique}{0}
			\FOR{$i \in \{ 1, \dots , N \}$}
				\STATE $L_{:i} \gets A_{:i} / \sqrt{A_{ii}}$
				\FOR{$j \in \{ i + 1, \dots, N \}$}
					\FOR{$ k \in \{ j, \dots, N \}$}
						\STATE $A_{kj} \gets A_{kj} - \frac{A_{ki} A_{ji}}{A_{ii}}$
					\ENDFOR
				\ENDFOR
			\ENDFOR
			\RETURN $L$\;
			\end{algorithmic}
			\caption{Standard dense Cholesky factorization.}
			\label{alg-Cholesky}
		\end{algorithm}
	\end{minipage}
	\begin{minipage}[t]{0.55\textwidth}
		\vspace{0pt}
		\begin{algorithm}[H]
		\setcounter{ALC@unique}{0}
			\textbf{Input: }$A \in \Reals^{N \times N}$ symmetric, \textcolor{red}{$ \operatorname{nz}(A) \subset S$}\\
			\textbf{Output:} $L\in \Reals^{N \times N}$ lower triang.\ \textcolor{red}{$ \operatorname{nz}(L) \subset S$}\\
			% \BlankLine
			\begin{algorithmic}[1]
				{\color{red}
				\FOR{$(i,j) \notin S$}
					\STATE $ A_{ij} \gets 0$
				\ENDFOR}
				\FOR{$i \in \{ 1, \dots, N \}$}
				\STATE $L_{:i} \gets A_{:i} / \sqrt{A_{ii}}$
					\FOR{$j \in \{ i + 1, \dots, N \}$ \textcolor{red}{: $(i, j) \in S$}}
						\FOR{$ k \in \{ j, \dots, N \}$ \textcolor{red}{: $( k, i ), ( k, j ) \in S$}}
							\STATE $A_{kj} \gets A_{kj} - \frac{A_{ki} A_{ji}}{A_{ii}}$
						\ENDFOR
					\ENDFOR
				\ENDFOR
				\RETURN $L$\;
			\end{algorithmic}
			\caption{Incomplete Cholesky factorization with sparsity pattern $S$.}
			\label{alg-ICholesky}
	  \end{algorithm}
	\end{minipage}
	\caption{Comparison of ordinary and incomplete Cholesky factorization.
	Here, for a matrix $A$, $\operatorname{nz}(A) \defeq \{ (i, j) \mid A_{ij} \neq 0 \}$ denotes the index set of the non-zero entries of $A$.}
	\label{fig-cholesky}
\end{figure}

A simple approach to decreasing the computational complexity of Cholesky factorization is the \emph{zero fill-in incomplete Cholesky factorization} \cite{meijerink1977iterative} (\ICH).
When performing Gaussian elimination using \ICH, we treat all entries of both the input matrix and the output factors outside a prescribed \emph{sparsity pattern} $S \subset \I \times \I$ as zero and correspondingly ignore all operations in which they are involved.
\Cref{fig-cholesky} shows a comparison of ordinary Cholesky factorization and \ICH.
Our approach to kernel matrices consists of applying \cref{alg-ICholesky} with an elimination ordering $\prec$ and a sparsity pattern $S$ that are chosen based on the locations of the $x_i$;
\Cref{const-superMulti} gives the details of this elimination ordering and sparsity pattern.

Write $\norm{ \quark }_{\FRO}$ for the Frobenius matrix norm and $C$ for a constant depending only on $d$, $\Omega$, $s$, $\norm{ \IK }$, $\|\IK^{-1}\|$, and $\delta$.
To simplify notation, the asymptotic bounds in this paper are stated in the case where the logarithmic factors are at least one.
Our main result is the following:

\begin{theorem}
	\label{thm-decayApproxIChol}
	Let $\IK$ and $\delta$ be defined as in \eqref{eq-op-between-sobolev} and \eqref{eq-homogeneity-param}.
	For $\rho \geq C \log(N/\epsilon)$, the sparse Cholesky factor $L^\rho$, obtained from \cref{alg-ICholesky} with the elimination ordering $\prec_\rho$ and sparsity pattern $\tilde{S}_{\rho} \subset \I \times \I$ described in \cref{const-superMulti}, satisfies
	\begin{equation}
		\label{eq-decayApproxIChol}
		\bignorm{ \KM - L^{\rho} L^{\rho, \top} }_{\FRO} \leq \epsilon.
	\end{equation}
	The selection of the ordering and sparsity pattern, as well as \cref{alg-ICholesky}, can be performed in computational complexity $C \rho^{2d} N \log^2 N$ in time and $C \rho^{d} N \log N$ in space.
	In particular, we can obtain an $\epsilon$-accurate approximation in Frobenius norm in complexity $C N \log^2(N) \log(N/\epsilon)^{2d}$ in time and $C N \log(N) \log(N/\epsilon)^{d}$ in space.
\end{theorem}

\begin{remark}
	For problems arising in Gaussian process regression, there will typically be no domain $\Omega$ on the boundary of which the process is conditioned to be zero;
	equivalently, $\Omega$ will be all of $\Reals^d$.
	This introduces an additional error, but we still observe good approximation of the covariances even of points close to the boundary (see \cref{ssec-noBoundary} for a detailed discussion).
\end{remark}

We will now present a simplified version of the elimination ordering and sparsity pattern (compared to the one mentioned in \cref{thm-decayApproxIChol}).
Although the proof of \cref{thm-decayApproxIChol} does not cover the stability of {\ICH} under this simplified version (rather, it covers the one described in \cref{const-superMulti}), extensive numerical experiments suggest that {\ICH} remains stable under this simplified version, and since it is also \emph{user-friendly} we recommend this as the ``go-to'' version for a simple, practical implementation.\footnote{Although more complex, the ordering used in \cref{thm-decayApproxIChol} has more potential for optimization by exploiting parallelism and dense linear algebra operations.}

\begin{figure}[t]
	\centering
	\includegraphics[width=0.24\textwidth]{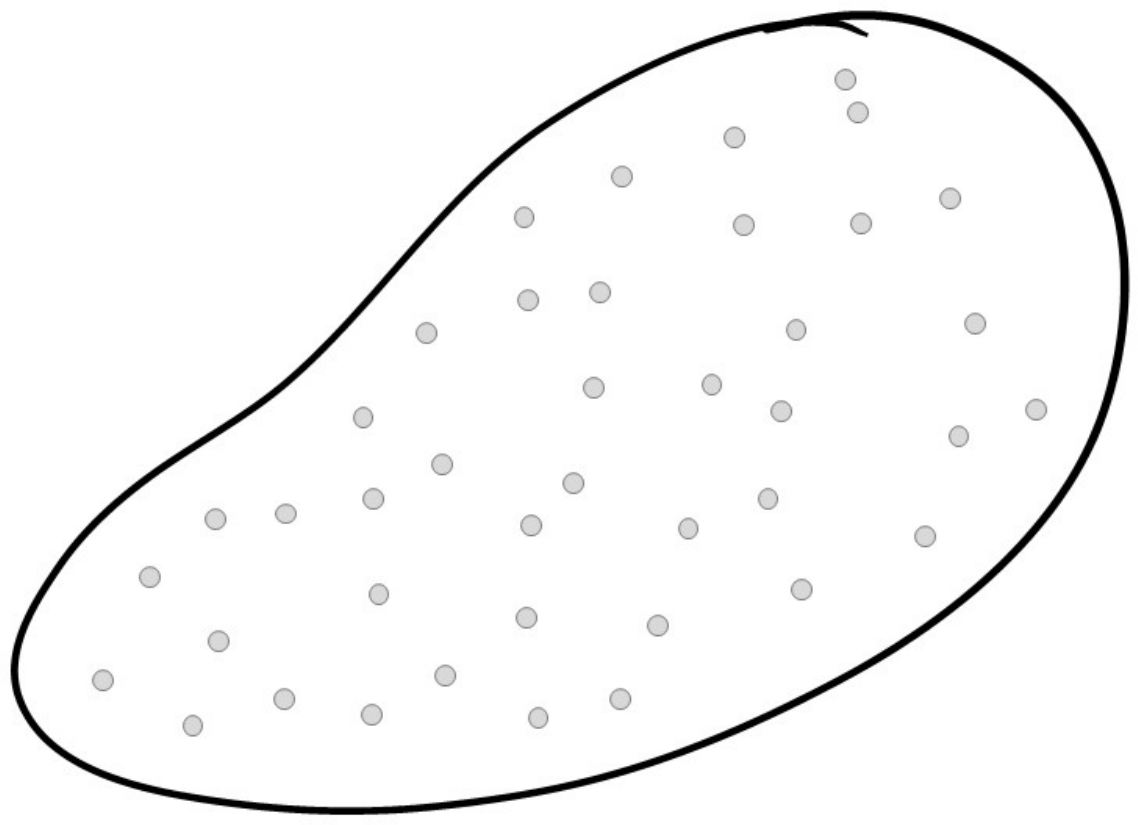}
	\includegraphics[width=0.24\textwidth]{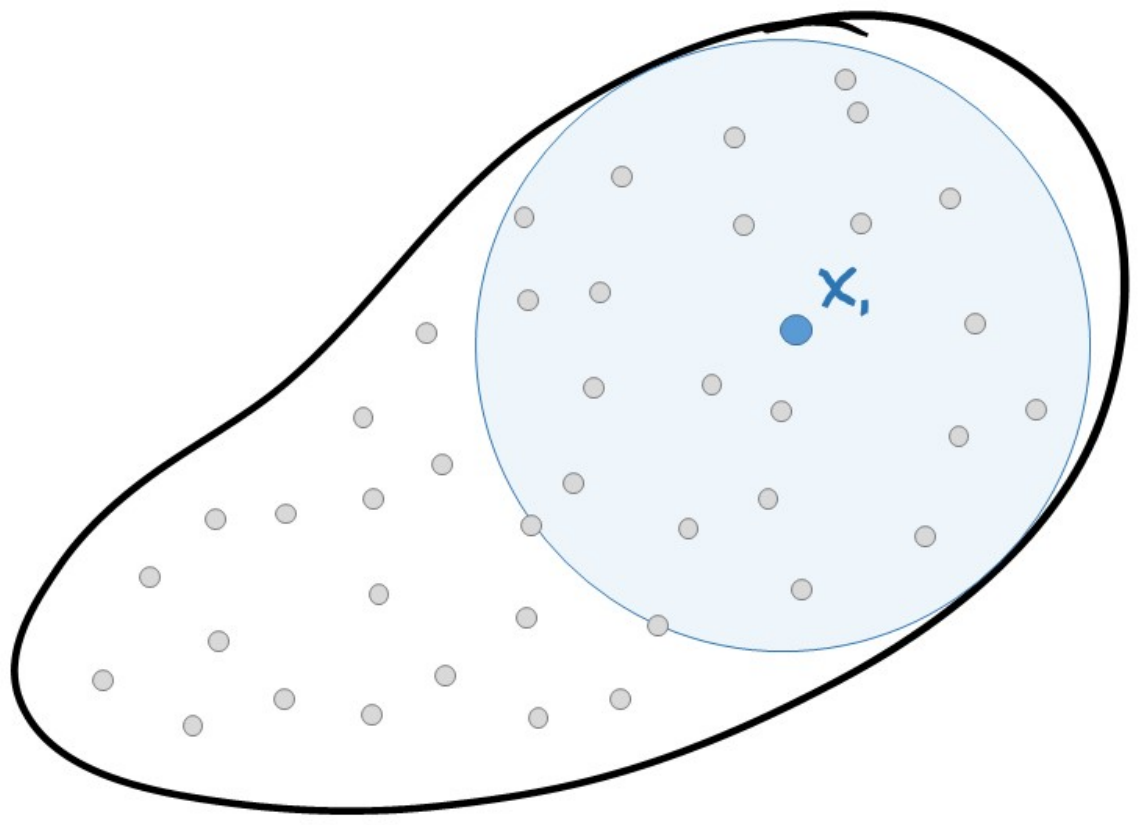}
	\includegraphics[width=0.24\textwidth]{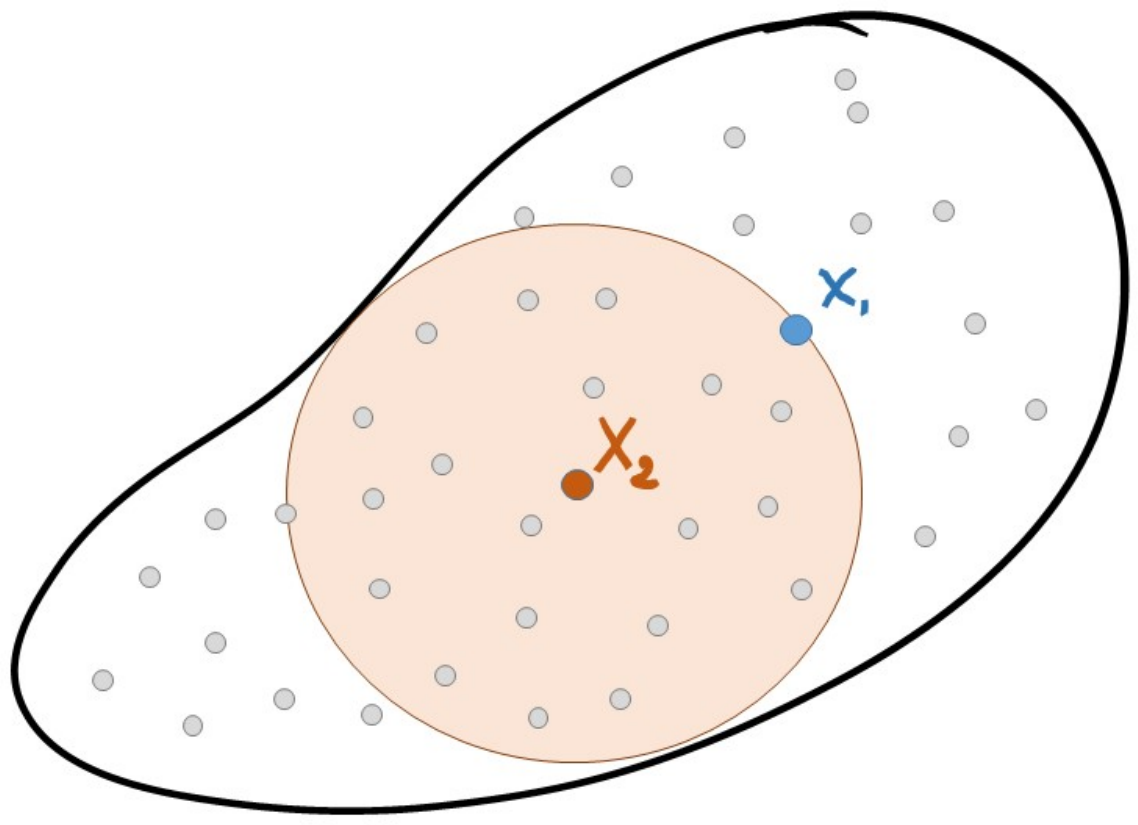}
	\includegraphics[width=0.24\textwidth]{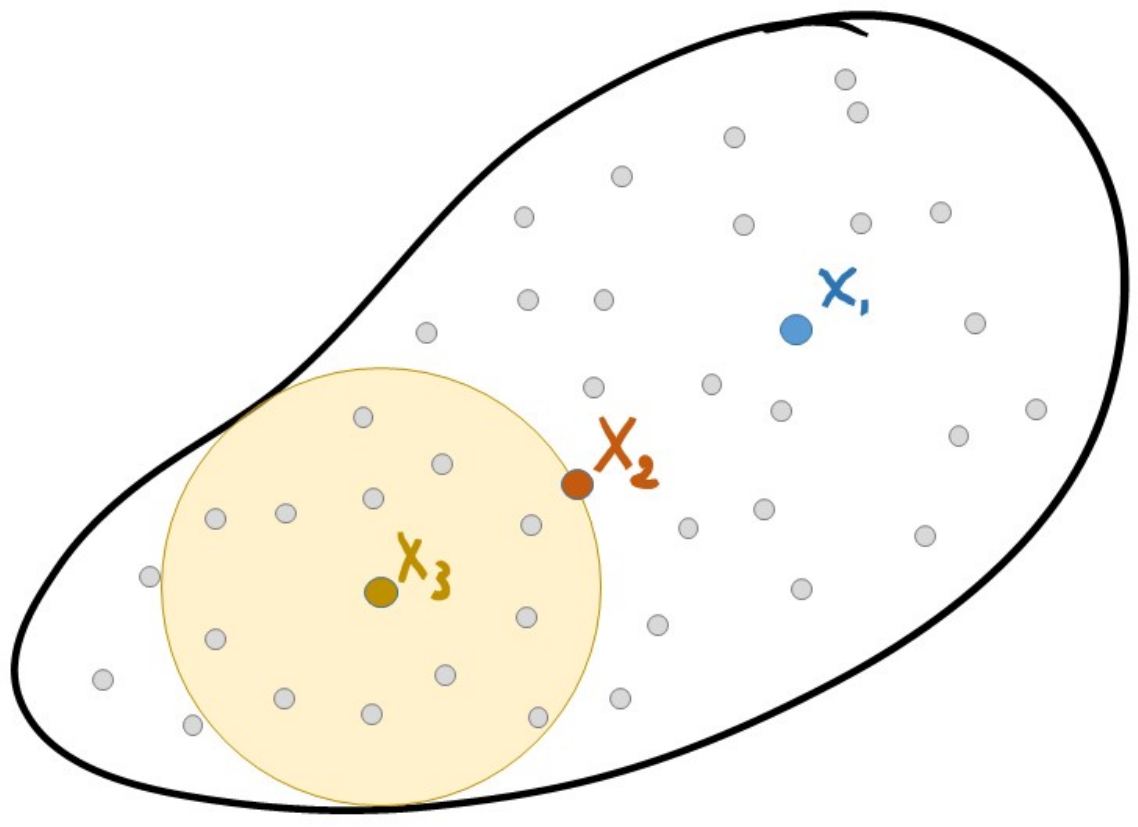}
	\includegraphics[width=0.24\textwidth]{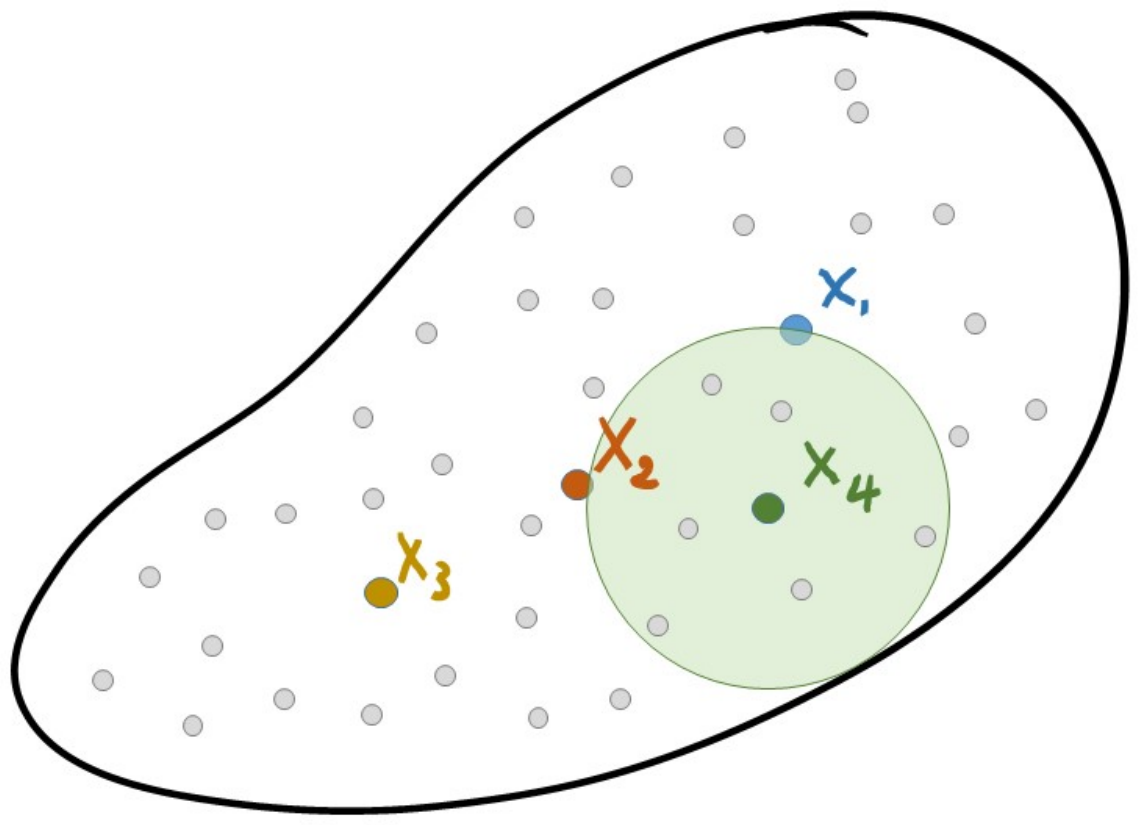}
	\includegraphics[width=0.24\textwidth]{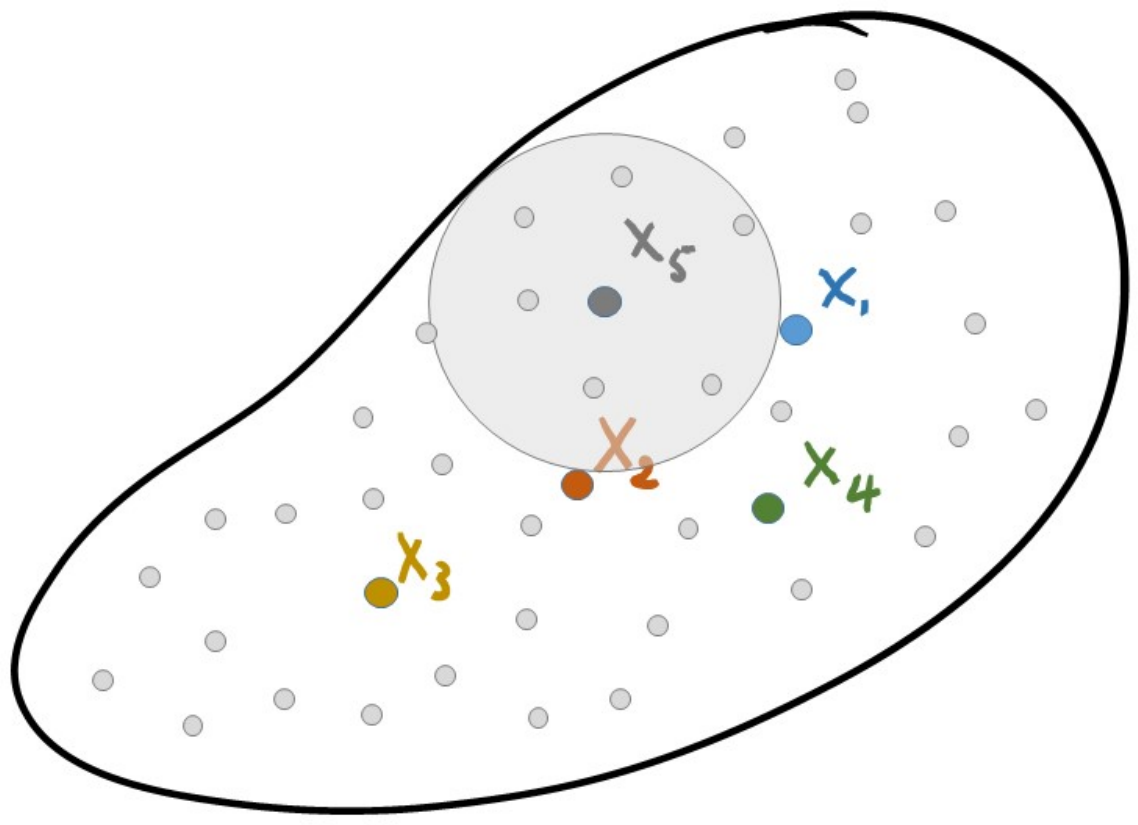}
	\includegraphics[width=0.24\textwidth]{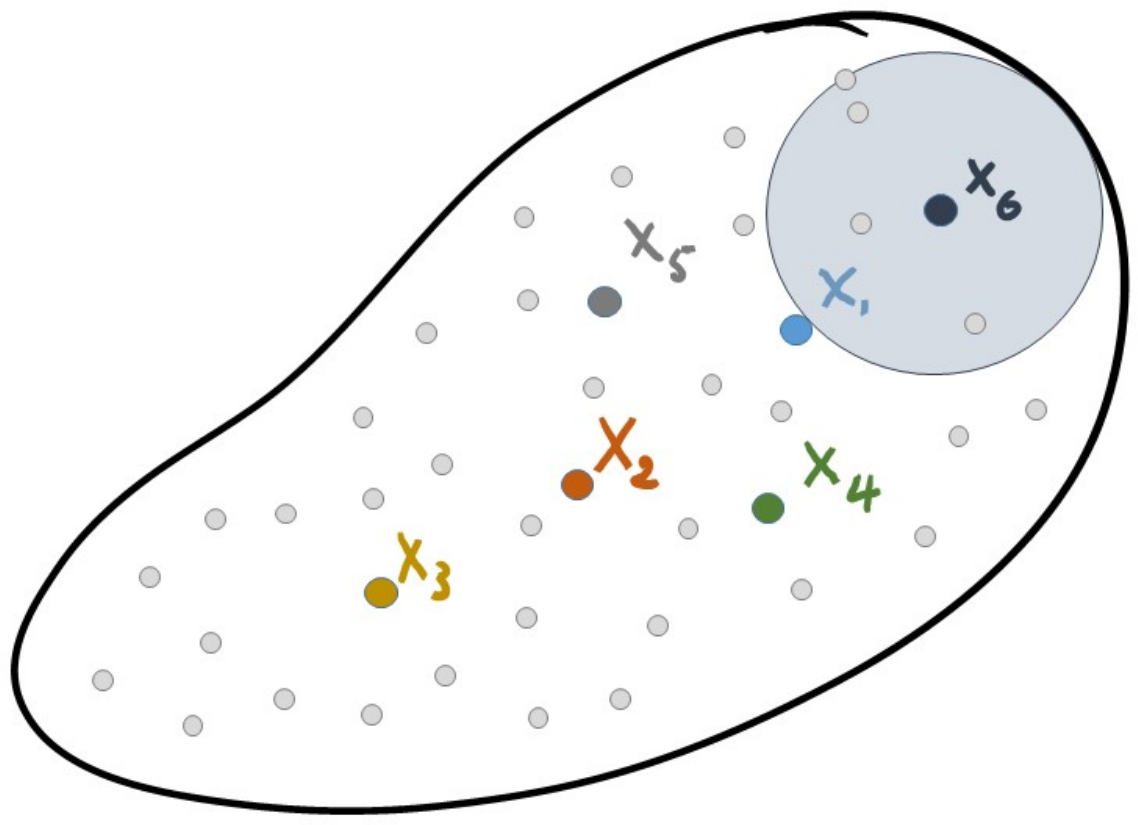}
	\includegraphics[width=0.24\textwidth]{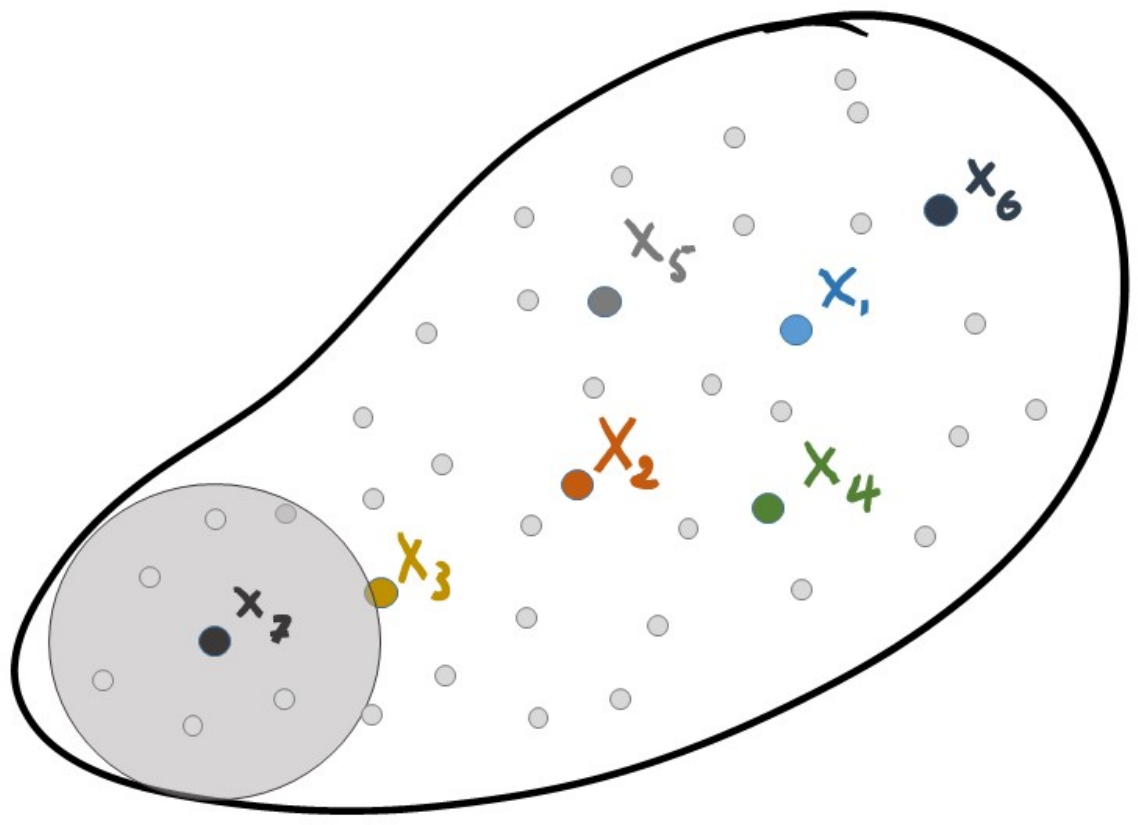}
	\caption{The maximin order successively adds the point that is furthest away from both $\partial \Omega$ and the set of points already added.
	The radius of the shaded circle is $l[i]$.}
	\label{fig-maximinOrder}
\end{figure}

\subsection{The elimination ordering and sparsity pattern}
\label{sssec-chooseOrder}

We use a \emph{maximum-minimum distance ordering} (maximin ordering) \cite{guinness2016permutation} as the elimination ordering.
This ordering is obtained by successively picking the point $x_{i}$ that is furthest away from $\partial \Omega$ and the points that were already picked.
If $\partial \Omega = \emptyset$, then we select an arbitrary $i \in \I$ as first index to eliminate;
otherwise, we choose the first index as
\begin{equation}
	i_1 \defeq \argmax_{i \in \I} \dist( x_i, \partial \Omega).
\end{equation}
Then, for the first $k$ indices of the ordering already chosen, we choose
\begin{equation}
	i_{k+1} \defeq \argmax_{i \in \I \setminus \{i_1, \dots, i_k \}} \dist( x_i, \{x_{i_1} , \dots, x_{i_k}\} \cup \partial \Omega).
\end{equation}
until we have ordered all the $N$ points (see \cref{fig-maximinOrder}).

\begin{figure}[t]
	\centering
	\includegraphics[width=0.24\textwidth]{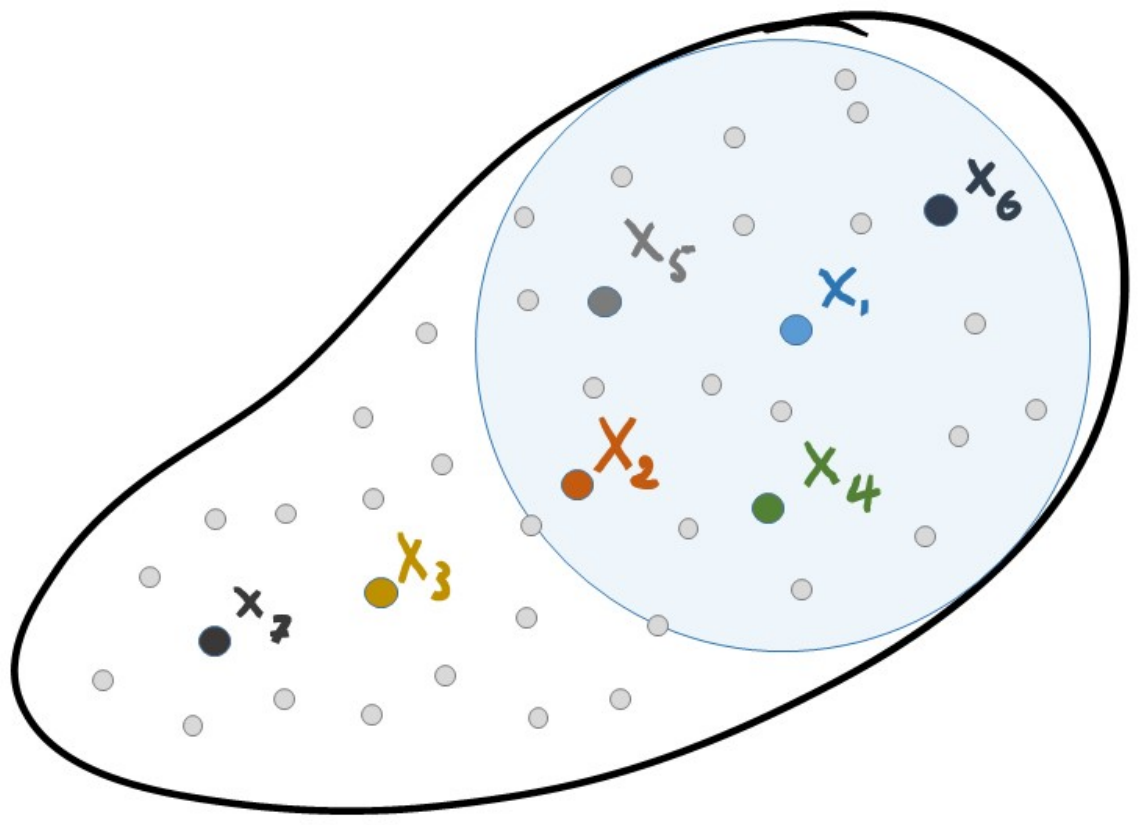}
	\includegraphics[width=0.24\textwidth]{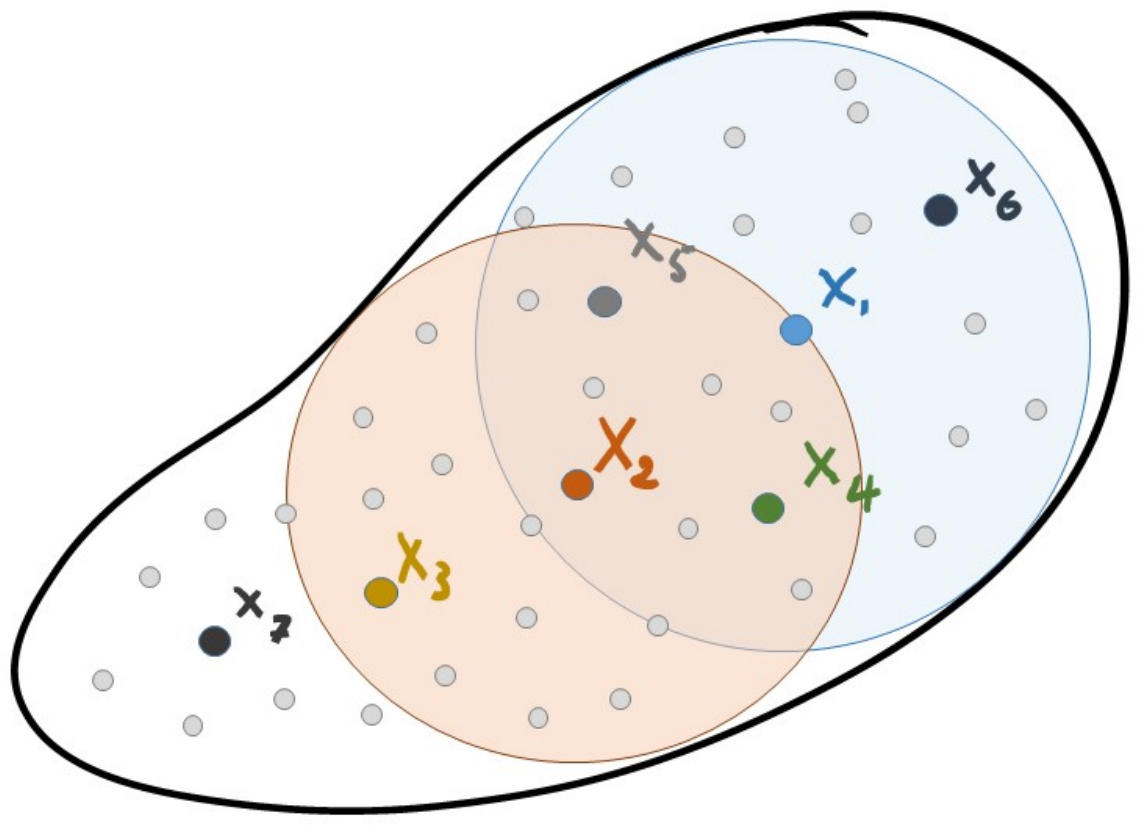}
	\includegraphics[width=0.24\textwidth]{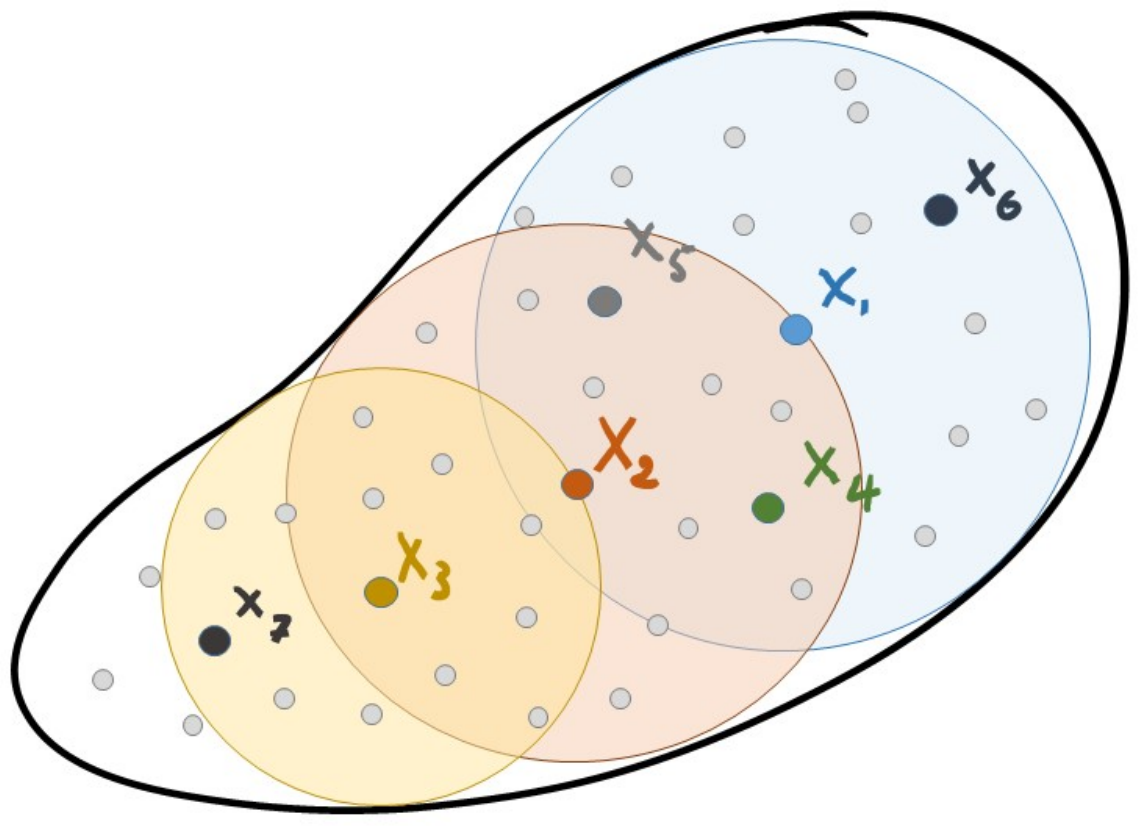}
	\includegraphics[width=0.24\textwidth]{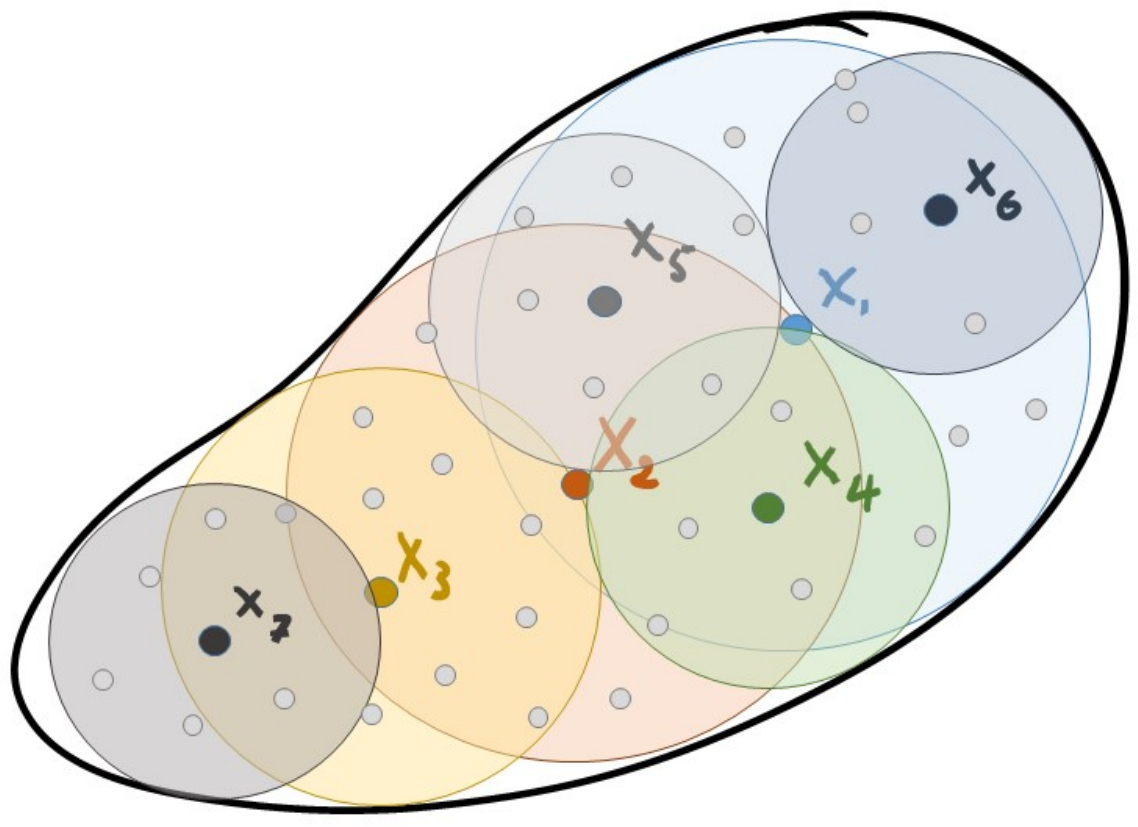}
	\includegraphics[width=0.24\textwidth]{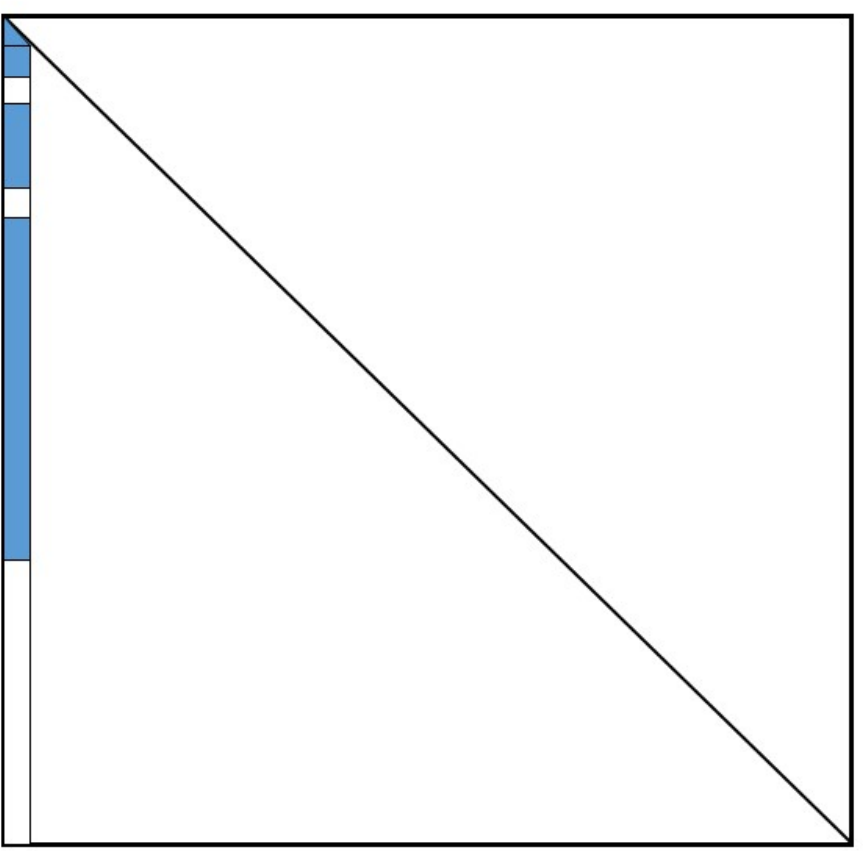}
	\includegraphics[width=0.24\textwidth]{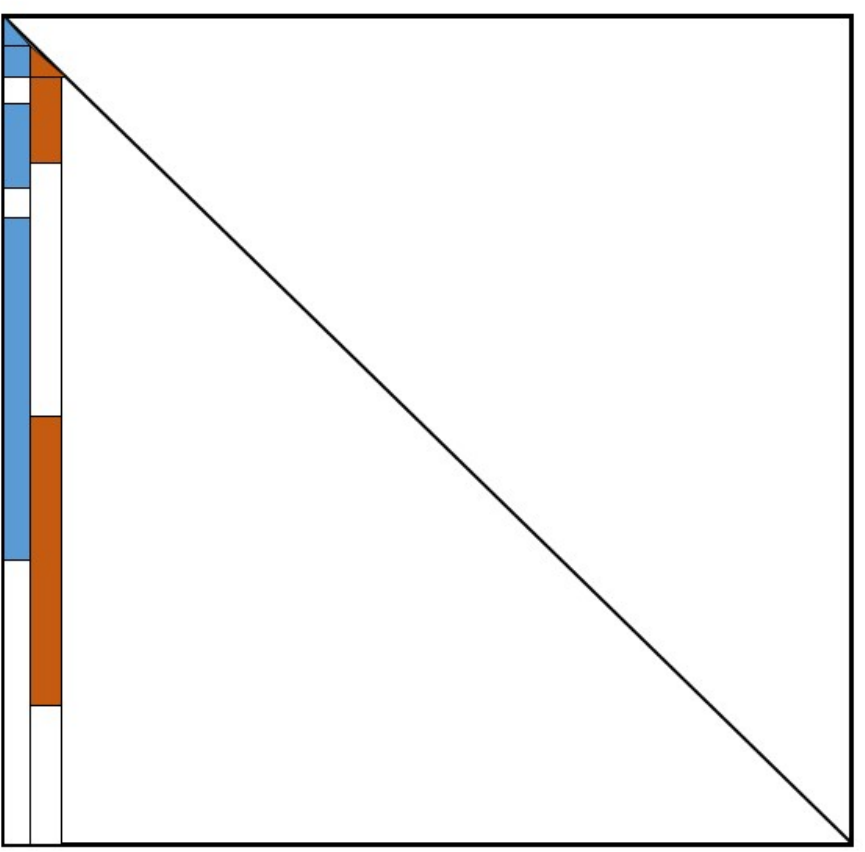}
	\includegraphics[width=0.24\textwidth]{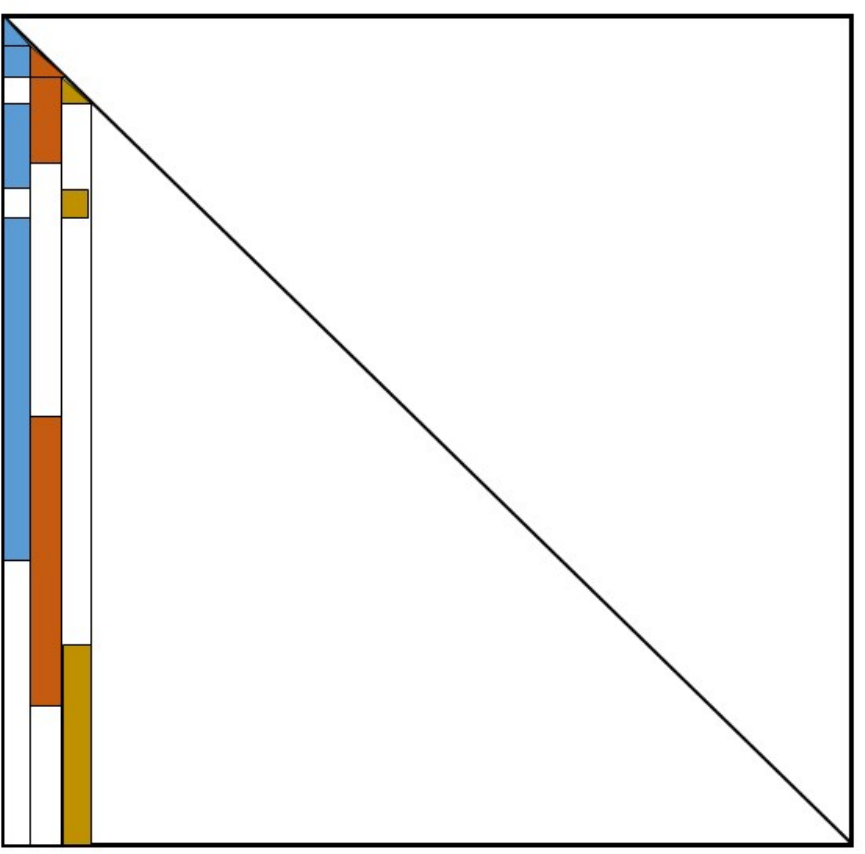}
	\includegraphics[width=0.24\textwidth]{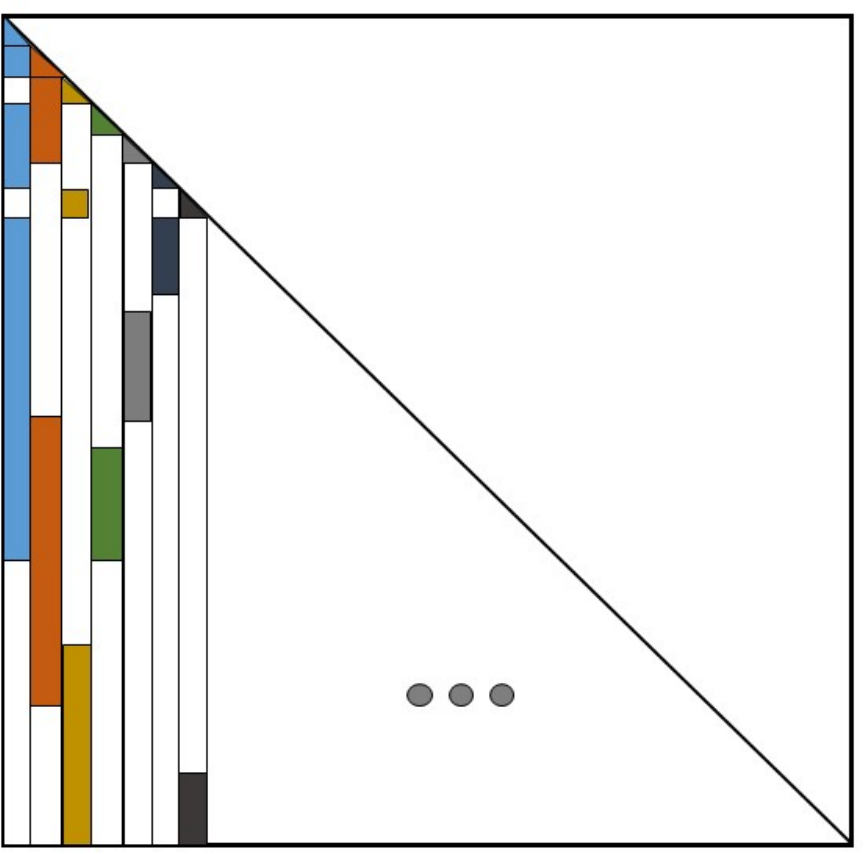}
	\caption{Upper row:  radii of interaction of the different degrees of freedom, for $\rho = 1$.
	Lower row: corresponding columns of the sparsity pattern.
	While the first columns are relatively dense, subsequent columns become more and more sparse.}
	\label{fig-maximinPattern}
\end{figure}

Let
\begin{equation}
	\label{eq-dist-point-to-boundary}
	l[i_{k}] \defeq \dist( x_{i_k}, \{x_{i_1} , \dots, x_{i_{k-1}}\} \cup \partial \Omega),
\end{equation}
be the distance between $x_{i_k}$ and $\partial \Omega$ and the earlier points in the ordering.
For $\rho > 0$, let $S_{\rho} \subset \I \times \I$ be the sparsity pattern defined by
\begin{equation}
	S_{\rho} \defeq \{ (i,j) \in \I \times \I \mid \dist(x_i,x_j) \leq \rho \max(l[i], l[j]) \}.
\end{equation}
%Note that an entry $( i, j )$ is added to $S_\rho$ if the distance between $x_i$ and $x_j$ is smaller than $\rho \max(l[i], l[j])$.
Here, $\rho$ parameterizes a trade-off between computational efficiency and accuracy.
For a given $\rho$, the sparsity pattern will have $C \rho^d N \log N$ entries and the Cholesky factorization will require $C \rho^{2d} N \log^2 N$ floating-point operations.
\Cref{fig-maximinPattern} shows the sparsity pattern for $\rho = 1$.
While a na{\"\i}ve implementation requires $\BigO(N^2)$ distance evaluations, \cref{thm-compSortSparse} shows that \cref{alg-sortSparse} delivers this sparsity pattern at computational complexity $C \rho^d N \log^2 N$.

\begin{figure}[t]
	\centering
	\includegraphics[width=0.4\textwidth]{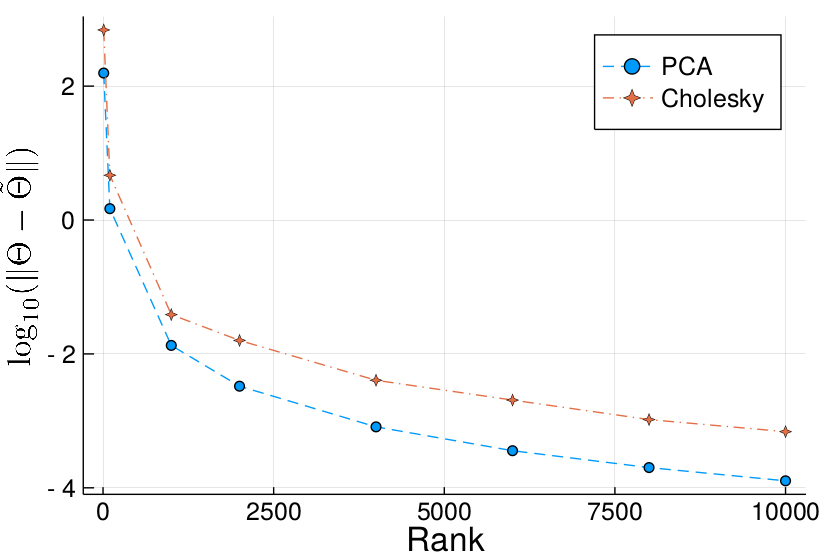}
	\includegraphics[width=0.4\textwidth]{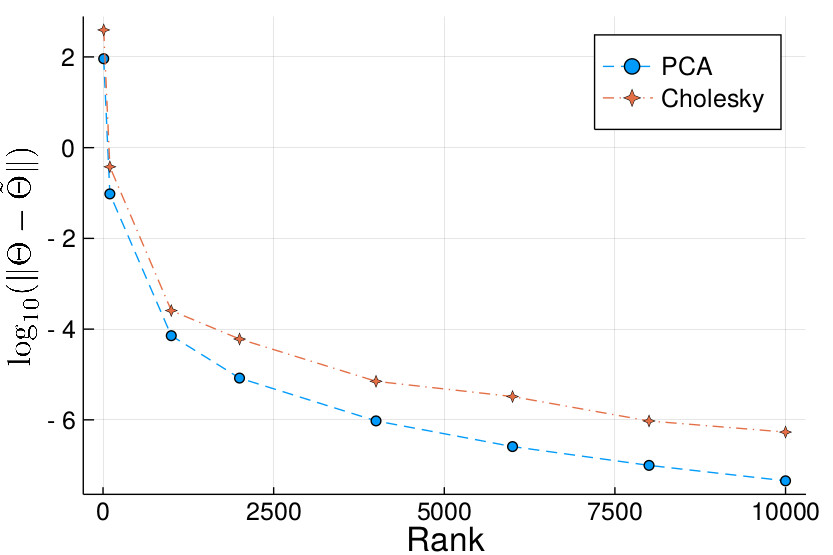}
	\caption{{Near-optimal sparse PCA:} Approximation errors comparisons between low-rank Cholesky ($\rho = \infty)$ and PCA for a Mat{\'e}rn kernel with smoothness parameters $\nu = 1$ (left) and $\nu = 2$ (right).}
	\label{fig-PCA}
\end{figure}

\subsection{Sparse approximate PCA}
\label{ssec-sparse-approximate-PCA}

The sparse Cholesky factorization described in \cref{ssec-simpleAlg} is also \emph{rank revealing} in the sense that the low-rank approximation obtained by using only the first $k$ columns of the Cholesky factorization achieves an accuracy within a constant factor of optimal rank-$k$ approximation (measured in operator norm).
This is illustrated by \cref{fig-PCA} and the following theorem:

\begin{theorem}
	\label{thm-rank-k-approximation}
	In the setting of \cref{thm-decayApproxIChol}, let $L^{(k)}$ be the rank-$k$ matrix defined by the first $k$ columns of the (dense) Cholesky factor $L$ of $\KM$.
	Then
	\begin{equation}
		\label{eq-rank-k-approximation}
		\bigl\| \KM - L^{(k)} L^{(k),\top} \bigr\| \leq C \| \KM \| k^{-\frac{2s}{d}}\,,
	\end{equation}
	where $\|\KM \|$ is the operator norm of $\KM $ and $C > 0$ depends only on $d$, $\Omega$, $s$, $\norm{ \IK }$, $\|\IK^{-1}\|$, and $\delta$.
\end{theorem}

The rank-$k$ approximation estimate \eqref{eq-rank-k-approximation} is a numerical homogenization accuracy estimate similar those obtained in \cite{malqvist2014localization,owhadi2014polyharmonic, owhadi2015multigrid,owhadi2017universal,hou2017sparse}.
Numerical homogenization basis functions can be identified by the \emph{last} $k$ rows of the lower triangular Cholesky factor of $\IKM \defeq \KM^{-1}$, obtained with the reverse elimination ordering described in \cref{sssec-spFacIKM}.

%% file: sec-whyitworks.tex
The method described in \cref{ssec-simpleAlg} combines two crude approximations.
First, it discards all but $\BigO(\rho^d N \log N)$ entries of the \emph{dense} $N \times N$ matrix $\KM$.
Second, it skips all but $\BigO(\rho^{2d} N \log^{2} N)$ operations of the Cholesky factorization of $\KM$ (which has complexity $\BigO(N^3)$).
The obvious question is:
why is the resulting approximation of $\KM$ accurate for $\rho \gtrsim \log N$?

\subsection{Sparse Cholesky factors of dense matrices}

The first part of the answer is that the Cholesky factors of
$\KM$ decay exponentially quickly away from the sparsity pattern $S_{\rho}$ when the maximin ordering is used as the elimination ordering.
This decay is illustrated in \cref{fig-sparsechol} and by the following \cref{thm-decayCholeskyIntro}.
Write $C$ for a constant depending only on $d$, $\Omega$, $s$, $\norm{ \IK }$, $\|\IK^{-1}\|$, and $\delta$.

\begin{theorem}
	\label{thm-decayCholeskyIntro}
	In the setting of \cref{thm-decayApproxIChol}, let $L$ be the full Cholesky factor of $\KM$ in the maximin ordering of \cref{ssec-simpleAlg}.
	Then, for $\rho \geq C \log(N/\epsilon)$, $S_{\rho}$ as defined in \cref{ssec-simpleAlg}, and
	\begin{equation}
		L^{S_\rho}_{ij} \defeq L_{ij} \one_{(i, j) \in S_{\rho}} =
		\begin{cases}
			L_{ij}, & \text{ for } (i,j) \in S_{\rho} ,\\
			0, & \text{ else,}
		\end{cases}
	\end{equation}
	the inequality $\left\|\KM - L^{S_\rho} L^{S_\rho, \top}\right\|_{\FRO} \leq \epsilon$ holds.
\end{theorem}

\Cref{alg-ICholesky} computes the exact Cholesky factorization under the assumption that the entries of $L$ lying outside $S_{\rho}$ are zero.
\Cref{thm-decayCholeskyIntro} shows that this assumption holds true up to an approximation error that decays exponentially in $\rho$, which supports the claim of accuracy of \cref{alg-ICholesky} for $\rho \gtrsim \log N$.
We will now explain the exponential decay of $L$ based on a probabilistic interpretation of Gaussian elimination.

\begin{figure}[t]
	\centering
	\includegraphics[width=0.24\textwidth]{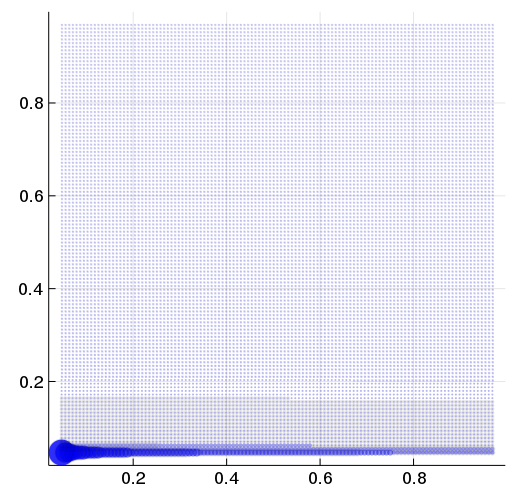}
	\includegraphics[width=0.24\textwidth]{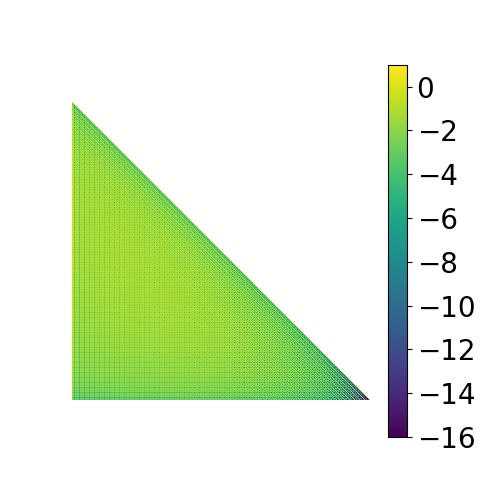}
	\includegraphics[width=0.24\textwidth]{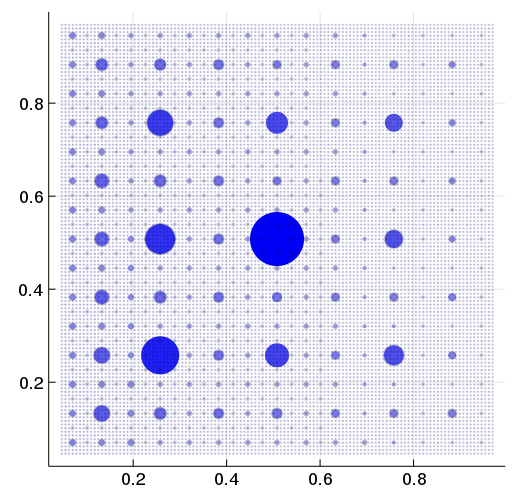}
	\includegraphics[width=0.24\textwidth]{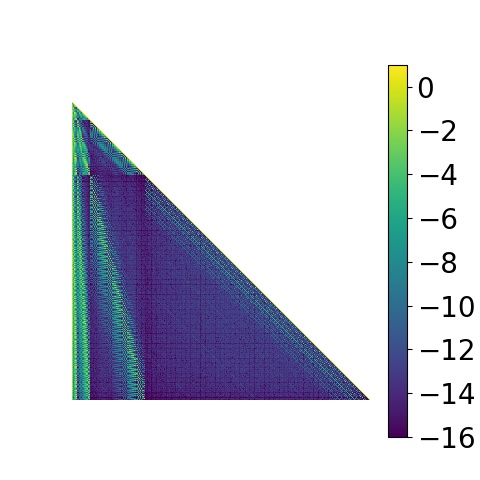}
	\caption{ The lexicographic--(left) and maximin (right) ordering of points in $\Omega \defeq (0, 1 )^2$ with larger and darker nodes corresponding to earlier elements of the ordering, together with the corresponding Cholesky factors of $\KM$ with entries plotted on a $\log_{10}$-scale.}
	\label{fig-sparsechol}
\end{figure}

\begin{figure}[t]
	\centering
	\includegraphics[scale=0.25]{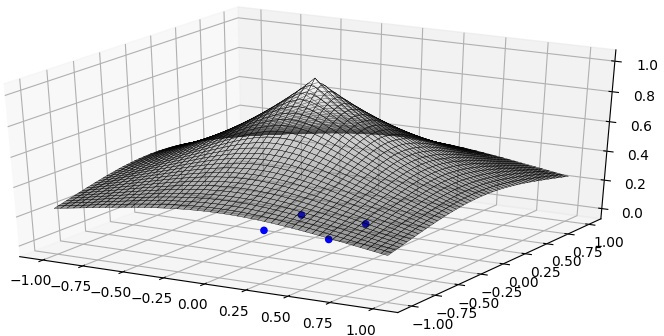}
	\includegraphics[scale=0.25]{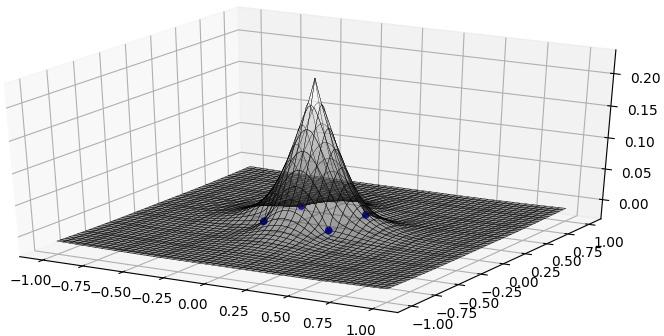}
	\caption{Left: covariance between a single site of a Mat{\'e}rn field with the	values at the remaining sites.
	Right: conditional covariance given the	values at the sites marked in blue.
	The conditional covariance decays significantly faster.}
	\label{fig-materndecorr}
\end{figure}

\subsection{Gaussian elimination, conditioning of Gaussian random variables,
and the screening effect}

The dense (block-)Cholesky factorization of a matrix $\KM$ can be seen
as the recursive application of the matrix identity
{\footnotesize
\begin{equation}
	\label{eqn-blockChol}
	\begin{pmatrix}
		\KM_{1,1} & \KM_{1,2} \\
		\KM_{2,1} & \KM_{2,2}
	\end{pmatrix}
	=
	\begin{pmatrix}
		\Id     & 0\\
		\KM_{2,1} ( \KM_{1,1} )^{-1} & \Id
	\end{pmatrix}
	\begin{pmatrix}
		\KM_{1,1} & 0 \\
		0         & \KM_{2,2} - \KM_{2,1} ( \KM_{1,1} )^{-1} \KM_{1,2}
	\end{pmatrix}
	\begin{pmatrix}
		\Id &  (\KM_{1,1} )^{-1} \KM_{1,2}\\
		0  & \Id
	\end{pmatrix},
\end{equation}}
where, at each step of the outermost loop, the above identity is applied to the Schur complement $\KM_{2,2} - \KM_{2,1} \left( \KM_{1,1} \right)^{-1} \KM_{1,2}$ obtained at the previous step.
If the Schur complements appearing during the factorization are sparse, then the final Cholesky factorization will also be sparse.

For $X = (X_1, X_2) \sim \N ( 0, \KM )$, the well-known identities
\begin{align}
	\Expect [ X_{2} \mid X_{1} = a ] &= \KM_{2,1} ( \KM_{1,1} )^{-1} a, \\
	\Cov [ X_{2} \mid X_{1} ] &= \KM_{2,2} - \KM_{2,1} ( \KM_{1,1} )^{-1} \KM_{1,2}\,,
\end{align}
imply that the sparsity of Cholesky factors of $\KM$ is equivalent to conditional independence of Gaussian vectors with covariance matrix $\KM$.
In the spatial statistics literature, it is well known that many smooth Gaussian processes are subject to the \emph{screening effect} \cite{stein2011when}.
This effect, illustrated in \cref{fig-materndecorr}, means that the value of the process at a given site, conditioned on the values at nearby sites, is only weakly dependent on the values at distant sites.

Consider now the $k$\textsuperscript{th} step of Cholesky factorization in the ordering described in \cref{ssec-simpleAlg}.
Any pair $x_i,x_j$ with $\dist\left( x_i, x_j \right) \gtrapprox l[k]$ will have points between them that have already been eliminated, as illustrated in \cref{fig-screening}.
Thus, the screening effect suggests that their correlation will be weak,
which supports choosing $\rho l[k]$ as a truncation radius.

\begin{figure}[t]
	\centering
	\includegraphics[scale=0.3]{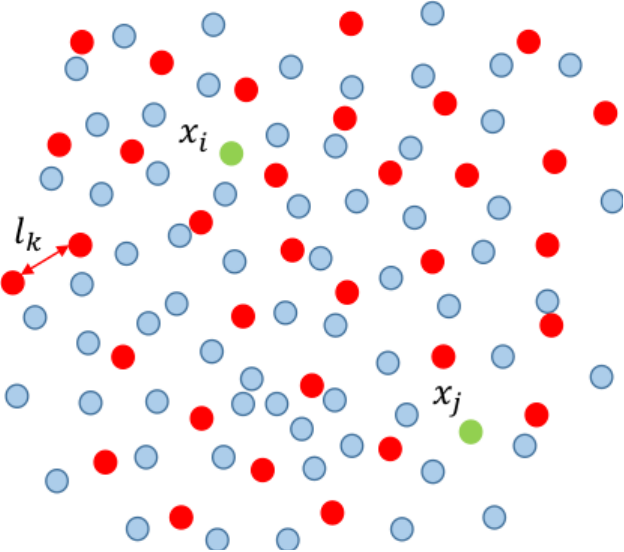}
	\caption{Step $k$ of the Cholesky factorization in the ordering described in \cref{ssec-simpleAlg}.
	The red points have already been eliminated and form a covering of radius $l[k]$.
	The separation of the green points $x_i$ and $x_j$ by points that have already been eliminated implies the weak correlation between $X_{x_i}$ and $X_{x_j}$ conditional on $\{ X_{x_i} \}_{i \preceq k}$.}
	\label{fig-screening}
\end{figure}

\subsection{Cholesky factorization and operator-adapted wavelets}
\label{sssec-choleskyGamblets}

Cholesky factorization in the maximin ordering is intimately related to computing operator-adapted wavelets.
In \cref{sec:analysis} we will use this connection to prove the accuracy of our approximation.

\paragraph{Operator-adapted wavelets.}
\cite{owhadi2015multigrid} and \cite{owhadi2017universal} introduced a novel class of operator-adapted wavelets called \emph{gamblets} (see also \cite{OwhScobook2018}).
For an operator $\IK$ defined as in \eqref{eq-op-between-sobolev}, gamblets can be identified as conditional expectations of the Gaussian process $\xi \sim \N\left(0,\IK^{-1}\right)$.
To construct the gamblets up to level $q \in \Naturals$ we start with a hierarchy of \emph{measurement functions} $\{\phi^{(k)}_i \}_{1 \leq k \leq q, i \in \I^{(k)}} \subset H^{-s}( \Omega )$;
heuristically, $k$ labels a scale, and $i$ a location at that scale.
These measurement functions are linearly nested in the sense that, for $k < l$,
\begin{equation}
	\label{eq-phi-i-k-nesting}
	\phi^{(k)}_{i} = \sum_{ j \in \I^{(l)}} \pi^{(k,l)}_{i,j} \phi^{(l)}_j.
\end{equation}
for some rank-$|\I^{(k)}|$ matrices $\pi^{(k,l)} \in \Reals^{\I^{(k)} \times \I^{(l)}}$.
Writing $\dualprod{ \quark }{ \quark }$ for the duality product between $H^{-s}( \Omega )$ and $H_0^s( \Omega )$, the conditional expectations
\begin{equation}
	\label{eq-phi-i-k-expectation}
	\psi^{(k)}_i \defeq \Expect\left[ \xi \,\middle|\, \dualprod{ \phi_j^{(k)} }{ \xi } = \delta_{ij} \text{ for all } j \in \I^{(k)} \right] \quad \text{for $i\in \I^{(k)}$}
\end{equation}
act as $\IK$-adapted pre-wavelets.
These pre-wavelets can be identified as optimal recovery splines in the sense of \cite{micchelli1977survey} through the representation formula
\begin{equation}
	\label{eq-pre-wavelet-representation}
	\psi^{(k)}_i = \sum_{j\in \I^{(k)}} \Theta^{(k),-1}_{i,j} \IK^{-1} \phi_j^{(k)} \quad \text{for $i\in \I^{(k)}$,}
\end{equation}
where $\Theta^{(k),-1}_{i,j}$ is the $(i,j)$\textsuperscript{th} entry of the inverse $\Theta^{(k),-1}$ of the matrix $\Theta^{(k)} \in \Reals^{\I^{(k)}\times \I^{(k)}}$ with entries $\Theta^{(k)}_{i,j} \defeq \int_{\Omega} \phi_i^{(k)} \IK^{-1}\phi_j^{(k)} \dx$.
The linear nesting of the $\phi_i^{(k)}$ across scales implies that the linear spaces $\V^{(k)} \defeq \spn\{\psi_i^{(k)}\mid i\in \I^{(k)}\}$ are nested (i.e.\ $\V^{(k-1)}\subset \V^{(k)}$).
The multi-resolution decomposition $\V^{(q)} \defeq \V^{(1)}\oplus \W^{(2)}\oplus \cdots \oplus \W^{(q)}$ is then obtained by defining $\W^{(k)}$ as the orthogonal complement $\W^{(k)}$ of $\V^{(k-1)}$ in $\V^{(k)}$ with respect to the energy scalar product $\innerprod{ u }{ v } \defeq \int_{\Omega} u \IK v \dx$.
Basis functions for $\W^{(k)}$ are identified (for $2\leq k\leq q$) by
\begin{equation}
	\label{eq-Wk-basis}
	\chi^{(k)}_i \defeq \sum_j W^{(k)}_{ij} \psi^{(k)}_j \quad \text{for $i\in \J^{(k)}$,}
\end{equation}
or, equivalently, by
\begin{equation}
	\chi^{(k)}_i \defeq \Expect\left[ \xi \,\middle|\, \left[\phi_j^{(k),W},\xi \right] = \delta_{ij} \delta_{kl} \text{ for all } 1\leq l \leq k, j \in \J^{(l)} \right] \quad \text{for $i\in \J^{(k)}$,}
\end{equation}
with $\phi^{(k),W}_i \defeq \sum_{j \in \I^{(k)}} W_{i,j}^{(k)} \phi^{(k)}_j$,
where $\J^{(k)} \cong \bigl( \I^{(k)} \setminus \I^{(k-1)} \bigr)$ and $W^{(k)}$ is a $\J^{(k)} \times \I^{(k)}$ matrix such that $\Image W^{(k),\top} = \Kernel \pi^{(k-1,k)}$ (writing $W^{(k),\top}$ for the transpose of $W^{(k)}$).
See \cref{fig-phipsichi} for an illustration.

\begin{figure}[t]
	\centering
	\includegraphics[width=0.24\textwidth]{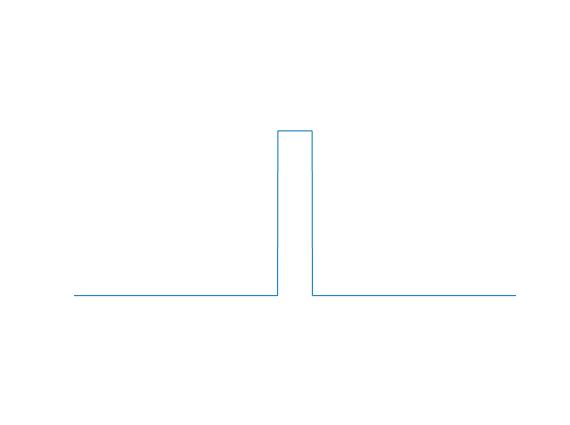}
	\includegraphics[width=0.24\textwidth]{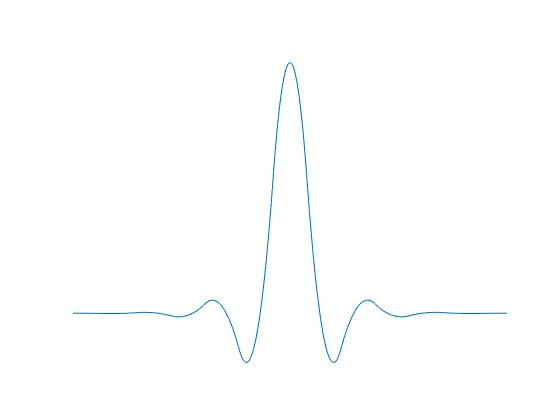}
	\includegraphics[width=0.24\textwidth]{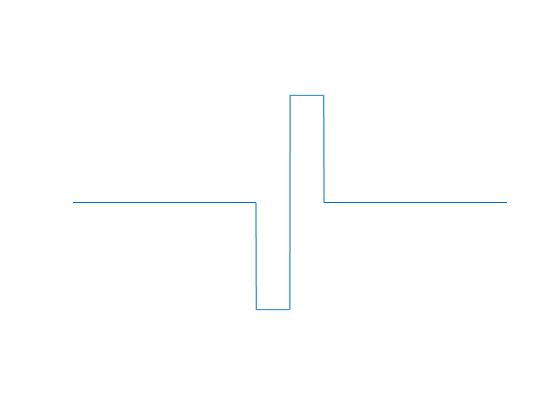}
	\includegraphics[width=0.24\textwidth]{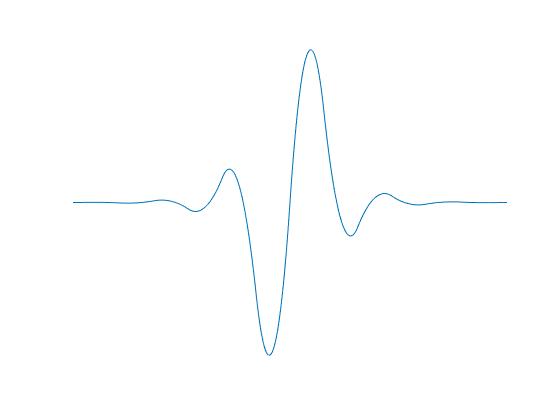}
	\caption{From left to right: An exemplary $\phi^{(k)}_i$, the corresponding
	$\psi^{(k)}_i$, a $\phi^{(k),W}_j$,	and the corresponding $\chi^{(k)}_j$, all in the setting of $d = 1$.}
	\label{fig-phipsichi}
\end{figure}

For simplicity we write $\J^{(1)} \defeq \I^{(1)}$ and $\chi_i^{(1)} \defeq \psi_i^{(1)}$.
Write $B^{(k)}$ for the $\J^{(k)}\times \J^{(k)}$ stiffness matrices $B^{(k)} \defeq \biginnerprod{ \chi_i^{(k)} }{ \chi_j^{(k)} }$.
The gamblets $\chi_i^{(k)}$ are $\IK$-adapted wavelets in the sense that, under sufficient conditions on the $\phi_i^{(k)}$, they satisfy the following three properties:
\begin{itemize}
	\item \textit{Scale orthogonality in the energy scalar product}, i.e.
	\begin{equation}
		\label{eq-scale-orthogonality}
		\biginnerprod{ \chi_i^{(k)} }{ \chi_j^{(l)} } = 0 \text{ for } l\not=k \text{ and }(i,j)\in \J^{(k)}\times \J^{(l)}\,.
	\end{equation}
	This leads to the block-diagonalization of the operator (with the $B^{(k)}$ as diagonal blocks).
	\item \textit{Uniform Riesz stability} in the energy norm:
	the condition numbers of the blocks $B^{(k)}$ are uniformly	bounded in $k$.
	\item \textit{Exponential decay}, which leads to sparse blocks $B^{(k)}$: the gamblets $\chi^{(k)}_i$ exhibit exponential decay on the scale associated with $k$.
\end{itemize}

Although the scale-orthogonality property \eqref{eq-scale-orthogonality} is always satisfied, the two others (exponential decay and uniform Riesz stability) depend on the properties of $\IK$ and the $\phi_i^{(k)}$.
In the setting of the localization of numerical homogenization basis functions (where $\IK$ is an elliptic PDE and the measurements $\phi_i^{(k)}$ are local and possibly not explicitly introduced), rigorous exponential decay estimates were pioneered in \cite{malqvist2014localization} and generalized in \cite{kornhuber2016analysis, owhadi2015multigrid, hou2017sparse, owhadi2017universal};
see \cref{sssec-decayHomog} for detailed comparisons.
For $\phi_i^{(k)}$ spanning the space of local polynomials of order up to $s-1$, bounded condition numbers are shown by \cite{owhadi2015multigrid,owhadi2017universal}.
The homogenization results obtained in the special case $q=2$ \cite{malqvist2014localization, owhadi2014polyharmonic, hou2017sparse} are closely related to the lower bound on the spectrum of $B^{(2)}$ (see \cref{sssec-boundCond}).

%Hierarchies of sub-sampled Diracs (Example \cref{examp-subsamp}) and Haar pre-wavelets (Example \cref{examp-average}) are two prototypical examples of $\phi_i^{(k)}$ satisfying all three properties.  This is shown in \citep{owhadi2015multigrid}  for $s=1$ and the Haar pre-wavelets \cref{examp-average}. For $s\in \mathbb{N}$ \citep{owhadi2017universal}  shows that all three properties are also satisfied if the $\phi_i^{(k)}$ form a local basis for polynomials of degree at most $s-1$ in each subset of a nested partition of $\Omega$ (i.e.\ in each $\tau_i^{(k)}$ of Example \cref{examp-average}, these local polynomials are also introduced in  \cite{hou2017sparse} without the hierarchy for numerical homogenization purposes).
%Although \citep{owhadi2017universal} also shows that the exponential decay property is  satisfied   for examples \cref{examp-subsamp} and \cref{examp-average}  (for $s\in \mathbb{N}$),  its proof of uniform Riesz stability
% employs a vanishing moment condition that is, in general, not satisfied for these two examples.
% This paper will relax this vanishing moment condition and thereby prove that the uniform Riesz stability holds for
% Examples \cref{examp-subsamp} and \cref{examp-average}.

\begin{figure}[t]
	\centering
	\includegraphics[scale=0.20]{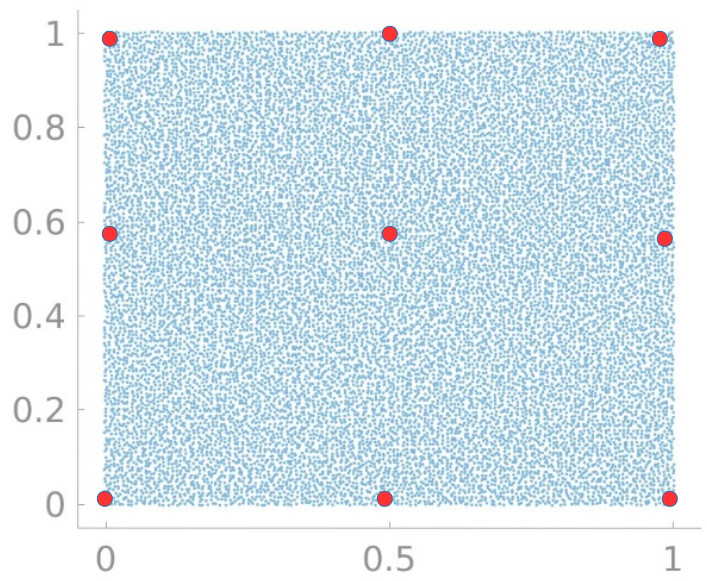}
	\includegraphics[scale=0.20]{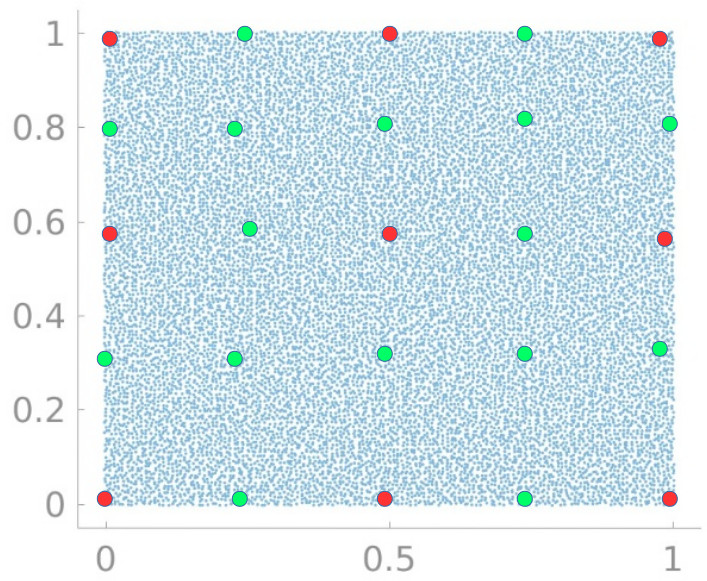}
	\includegraphics[scale=0.20]{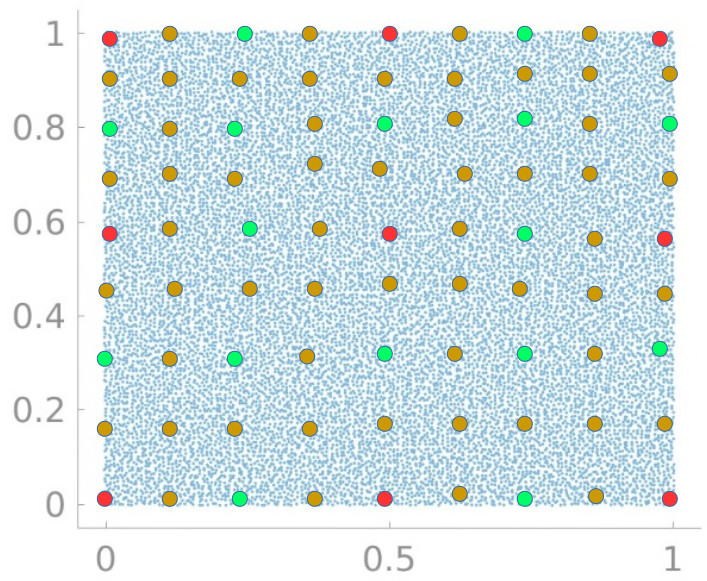}
	\caption{{The implicit hierarchy of the maximin ordering}.
	The maximin ordering has a hidden hierarchical structure, which can be discovered by picking a scale factor $h \in (0,1)$ and defining $\J^{(k)} \defeq \left\{ j \in \I \,\middle|\, h^{k} \leq l[i]/l[1] < h^{k - 1} \right\}$ for $1 \leq k \leq q \defeq \left \lceil \log_{h}\left( l[k]/l[1] \right) \right \rceil$.
	In the figure, we see $\J^{(1)}$ in red, $\J^{( 2 )}$ in green, and $\J^{( 3 )}$ in brown, for $h = 1/2$.}
	\label{fig-ordering}
\end{figure}

\paragraph{Relation to Cholesky factorization.}
To explain the connection between gamblets and Cholesky factorization, let $\J \defeq \J^{(1)}\cup \cdots \cup \J^{(q)}$, let $W^{(1)}$ be the $\I^{(1)}\times \I^{(1)}$ identity matrix, let $\pi^{(k,k)}$ be the $\I^{(k)}\times \I^{(k)}$ identity matrix, and let $\bar{\Theta}$ be the $\J\times \J$ symmetric matrix with $\J^{(k)}\times \J^{(l)}$ block defined for $k\leq l$ by
\begin{equation}
	\bar{\Theta}_{k,l} \defeq W^{(k)}\Theta^{(k)} \pi^{(k,l)}W^{(l),T}\,,
\end{equation}
or equivalently by
\begin{equation}
	\bigl(\bar{\Theta}_{k,l}\bigr)_{ij} \defeq \bigdualprod{ \phi^{(k),W}_i }{ \IK^{-1} \phi^{(l),W}_j }.
\end{equation}
Then, the block-Cholesky factorization of $\bar{\Theta}$ satisfies the identity
\begin{equation}
	\label{eqn-gambletFact}
	\bar{\Theta} = \bar{L} D \bar{L}^{\top},
\end{equation}
where $D$ is a block-diagonal matrix with the $\J^{(k)}\times \J^{(k)}$ diagonal block equal to $B^{(k),-1}$ and
\begin{equation}
	\bar{L}_{i,j}
	\defeq
	\begin{cases}
		\delta_{i,j} , & \text{if $i, j\in \J^{(k)}$,} \\
		0 , & \text{if $i\in \J^{(k)}, j\in \J^{(k')}$, and $k' > k$,} \\
		[\phi_i^{(k)},\chi_j^{(k')}] , & \text{if $i\in \J^{(k)}, j\in \J^{(k')}$, and $k' < k$.}
	\end{cases}
\end{equation}
Therefore, computing gamblets associated to the operator $\IK$ and measurement functions $\phi_i$ is equivalent to computing a block-Cholesky factorization of $\KM$ in the multiresolution basis given by the $\phi_i^{(k),W}$.

The Cholesky decomposition of $\KM$ \eqref{eq-kernel-matrix} belongs to this setting.
Indeed, although the maximin ordering of \cref{ssec-simpleAlg} has no explicit multiscale structure, this structure can be introduced, as described in \cref{fig-ordering}, by decomposing $x_1,\ldots,x_N$ into a nested hierarchy $\{x_i\}_{i\in \I^{(1)}}\subset \{x_i\}_{i\in \I^{(2)}}\subset \cdots \subset \{x_i\}_{i\in \I^{(q)}}$, and choosing $\phi_i^{(k)}=\boldsymbol{\delta}(\quark-x_i)$ for $i\in \I^{(k)}$ and $k\in \{1,\ldots,q\}$, where $\boldsymbol{\delta}$ denotes the unit (unscaled) Dirac delta function.
Under this choice, $\pi^{(k,k+1)}_{i,j}=1$ for $j \in \I^{(k)}$ and $\pi^{(k,k+1)}_{i,j}=0$ for $j\not\in \I^{(k)}$.
Letting $\J^{(k)}$ label the indices in $\I^{(k)}/\I^{(k-1)}$ and choosing $W^{(k)}_{i,j}=1$ for $j\in \I^{(k)}/\I^{(k-1)}$ and $W^{(k)}_{i,j}=0$ for $j\in \I^{(k-1)}$ implies $\KM=\bar{\Theta}$.
The exponential decay of $\bar{L}$ and $D^{-1}$ follows from known results \cite{owhadi2017universal} on exponential decay of the $\chi_j^{(k)}$.
The uniform bound on the condition number of the $\IKMC^{(k)}$ is proved in \cref{sssec-boundCond}.
The exponential decay and uniform bound on the condition numbers of the blocks $B^{(k)}$ imply the exponential decay of the Cholesky factors $\hat{L}$ of $D$ and hence of $L = \bar{L} \hat{L}$.
The approximation error estimate \eqref{eq-decayApproxIChol} is then obtained by matching the sparsity set $S$ with the near-sparse structure of $L$.

%% file: sec-numerics.tex
\begin{algorithm}[!ht]\footnotesize
	% \LinesNumbered
	\textbf{Input:} Real $\rho \geq 2$ and Oracles $\disttt( \quark, \quark ), \disttt_{\partial \Omega}( \quark )$ such that
	  $\disttt ( i, j ) = \dist\left(x_i, x_j\right)$ and
	$\disttt_{\partial \Omega} \left(i\right) =
	\dist\left( x_i, \partial \Omega \right)$ \\
	\textbf{Output:} An array $l[:]$ of distances, an array $P$ encoding the multiresolution
	  ordering, and an array of index pairs $S$ containing the sparsity pattern. \\
	\begin{algorithmic}[1]
	\setcounter{ALC@unique}{0}
		\STATE $P = \emptyset$
		\FOR{$i \in \{1,  \dots , N\}$}
		  \STATE $l[i] \leftarrow \disttt_{\partial \Omega}(i)$
		  \STATE $p[i] \leftarrow \emptyset$
		  \STATE $c[i] \leftarrow \emptyset$
		\ENDFOR
		\STATE \COMMENT{Creates a mutable binary heap, containing pairs of indices and
		distances as elements:}
		\STATE $H \leftarrow \mathtt{MutableMaximalBinaryHeap}\left( \{ (i, l[i] ) \}_{i \in \{1, \dots ,N\}} \right)$\;
		\STATE \COMMENT{Instates the Heap property, with a pair with maximal distance occupying the
		root of the heap:}
		\STATE $\mathtt{heapSort}!( H )$

		\STATE \COMMENT{Processing the first index:}
		\STATE \COMMENT{Get the root of the heap, remove it, and restore the heap property:}
		\STATE $(i, l) = \mathtt{pop}(H)$
		\STATE \COMMENT{Add the index as the next element of the ordering}
		$\mathtt{push}\left(P, i\right)$
		\FOR{$j \in \{1,  \dots , N\}$ }
		  \STATE $\mathtt{push}( c[i], j )$
		  \STATE $\mathtt{push}( p[j], i )$
		  \STATE $\mathtt{sort!}\left( c[i], \disttt( \quark, i ) \right)$
		  \STATE $\mathtt{decrease!}\left(H, j, \disttt( i, j ) \right)$
		\ENDFOR
		\STATE \COMMENT{Processing remaining indices:}
		\WHILE{$H \neq \emptyset$}
		  \label{line-whileSortSparse}
		  \STATE \COMMENT{Get the root of the heap, remove it, and restore the heap property:}
		  $(i, l) = \mathtt{pop}(H)$\;
		  $l[i] \leftarrow l$\;
		  	\STATE \COMMENT{Select the parent node that has all possible children of $i$ amongst its
		    children, and is closest to $i$:}
		    \STATE $k = \argmin_{j \in p[i]: \disttt( i, j ) + \rho l[i] \leq \rho l[j]}
		  	\label{line-kArgminSortSparse}
			\disttt\left( i, j \right)$\;
		  	\STATE \COMMENT{Loop through those children of $k$ that are close enough to $k$ to possibly
		  	be children of $i$:}
		  	\FOR{$j \in c[k]: \disttt( j, k ) \leq \disttt( i, k ) +
		  	  \rho l[i] $}
		  	  \label{line-forSortSparse}
		  	  \STATE $\mathtt{decrease!}\left(H, j, \disttt( i, j ) \right)$
		  	  \IF{ $\disttt( i, j ) \leq \rho l[i]  $ }
		  	    \STATE $\mathtt{push}( c[i], j )$
		  	    \STATE $\mathtt{push}( p[j], i )$
		  	  \ENDIF
		  	\ENDFOR
		  	\STATE \COMMENT{Add the index as the next element of the ordering}
		  
		 	\STATE  $\mathtt{push}\left(P, i\right)$
			\STATE \COMMENT{Sort the children according to distance to the parent node, so that the closest
			children can be found more easily}
			$\mathtt{sort!}\left( c[i], \disttt( \quark, i ) \right)$\;
			\label{line-sort!SortSparse}
		\ENDWHILE
		\STATE \COMMENT{Aggregating the lists of children into the sparsity pattern:}
		\FOR{$i \in \{1, \dots ,N\}$}
		  \FOR{ $j \in c[i]$}
		    \STATE $\mathtt{push!}\left(S, ( i, j ) \right) $
		    \STATE $\mathtt{push!}\left(S, ( j, j ) \right) $
		  \ENDFOR
		\ENDFOR
	\end{algorithmic}
	\caption{Ordering and sparsity pattern algorithm.}
	\label{alg-sortSparse}
\end{algorithm}

\begin{figure}
	\centering
	\includegraphics[scale=0.20]{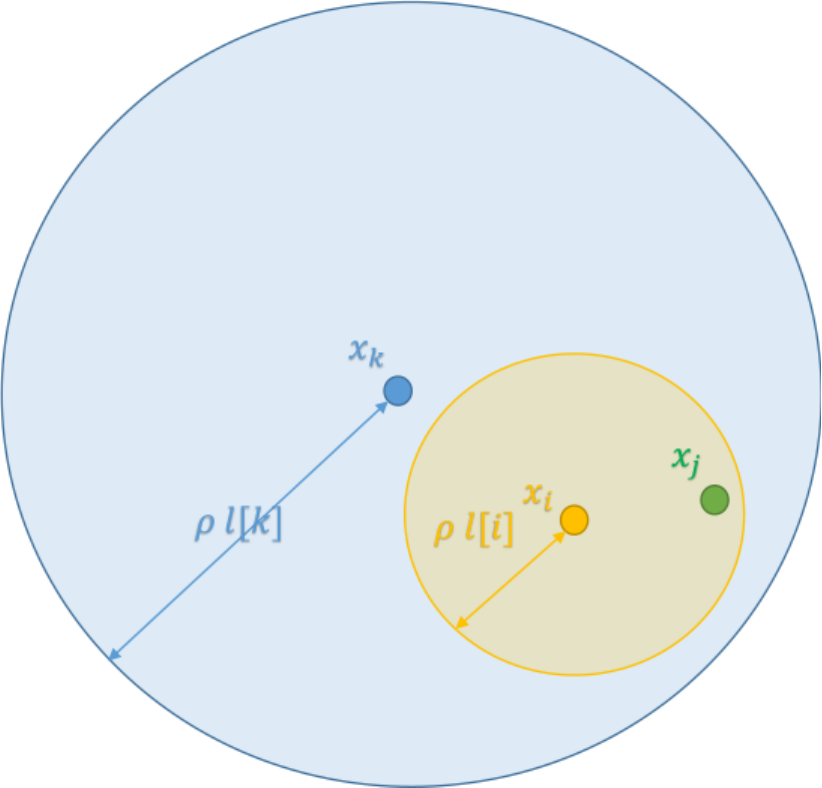}
	\caption{Localization of computation in \cref{alg-sortSparse} based on hierarchy.
	When adding $i$ to the ordering, only consider indices $j$ such that $\dist(x_i, x_j) \leq \rho l[i]$.
	Those indices are a subset of the \emph{children} of the coarse-level index $k$ if $\rho l[k] \geq \dist(x_i, x_j) + \rho l[i]]$.
	Thus, the search for candidates $j$ can be restricted to those children of $k$.}
	\label{fig-sortSparse}
\end{figure}

\subsection{Selection of the sparsity pattern and ordering}

This section introduces an $\BigO( \rho^d N \log^2 N)$-complexity algorithm (\cref{alg-sortSparse}) for selecting the sparsity pattern and ordering used as inputs in \cref{alg-ICholesky}.
This algorithm does not explicitly query the position of the $\{ x_i \}_{i \in \I}$ and only uses pairwise distances by processing points one by one by updating a mutable binary heap, keeping track of the point to be processed at each step.
With this approach, our proposed algorithm is oblivious to the dimension $d$ of the ambient space and, in particular, can automatically exploit low-dimensional structure in the point cloud $\{ x_i \}_{i \in \I}$.
In order to avoid computing all $\BigO(N^2)$ pairwise distances, as illustrated in \cref{fig-sortSparse}, \cref{alg-sortSparse} uses the sparsity pattern obtained on the coarser scales to restrict computation at the finer scales to local neighborhoods.

\begin{theorem}
	\label{thm-compSortSparse}
	The output of \cref{alg-sortSparse} is the ordering and sparsity pattern described in \cref{ssec-simpleAlg}.
	Furthermore, in the setting of \cref{thm-decayCholeskyIntro}, if the oracles $\disttt(\quark, \quark)$ and $\disttt_{\partial \Omega}(\quark)$ can be queried in complexity $\BigO(1)$, then the complexity of \cref{alg-sortSparse} is bounded by $C \rho^d N \log^2 N$, where $C$ is a constant depending only on $d$, $\Omega$ and $\delta$.
\end{theorem}

\Cref{thm-compSortSparse} is proved in \cref{apsec-compSortSparse}.
As discussed therein, in the case $\Omega = \Reals^d$, \cref{alg-sortSparse} has the advantage that its computational complexity depends only on the intrinsic dimension of the dataset, which can be much smaller than $d$.

\subsection{The case of the whole space \texorpdfstring{($\Omega = \Reals^d$)}{}}
\label{ssec-noBoundary}

Many applications in Gaussian process statistics and machine learning are in the  $\Omega = \Reals^d$ setting.
In that setting, the Mat{\'e}rn family of kernels \eqref{eqn-matern} is a popular choice that is equivalent to using the whole-space Green's function of an elliptic PDE as covariance function \cite{whittle1954stationary, whittle1963stochastic}.
Let $\bar{\Omega}$ be a bounded domain containing the $\{ x_i \}_{i \in \I}$.
The case $\Omega=\Reals^d$ is not covered in \cref{thm-decayCholeskyIntro} because in this case the screening effect is weakened near the boundary of $\bar{\Omega}$ by the absence of measurements points outside of $\bar{\Omega}$.
Therefore, distant points close to the boundary of $\bar{\Omega}$ will have stronger conditional correlations than similarly distant points in the interior of $\bar{\Omega}$ (see \cref{fig-boundaryEffects}).
As observed by \cite{roininen2014whittle} and \cite{daon2016mitigating},
Markov random field (MRF) approaches that use a discretization of the underlying PDE face similar challenges at the boundary.
While the weakening of the exponential decay at the boundary worsens the accuracy of our method, the numerical results of \cref{ssec-numres} (which are all obtained without imposing boundary conditions) suggest that its overall impact is limited.
In particular, as shown in \cref{fig-boundaryEffects}, it does not cause significant artifacts in the quality of the approximation near the boundary.
This differs from the significant boundary artifacts of MRF methods, which have to be mitigated against by a careful calibration of boundary conditions \cite{roininen2014whittle,daon2016mitigating}.
Although the numerical results presented in this section are mostly obtained with $x_i \sim \unif ([0,1]^d )$,
in many practical applications, the density of measurement points will slowly (rather than abruptly) decrease towards zero near the boundary of the sampled domain, which drastically decreases the boundary errors shown above.
Accuracy can also be enhanced by adding \emph{artificial} points $\{ x_i \}_{i \in \tI}$ at the boundary.
By applying the Cholesky factorization to $\{ x_i \}_{i \in \I \cup \tI}$, and then restricting the resulting matrix to $\I \times \I$, we can obtain a very accurate approximate matrix-vector multiplication.
Although not in the form of a Cholesky factorization, this approximation can be efficiently inverted using iterative methods such as conjugate gradient \cite{shewchuk1994introduction} preconditioned with the Cholesky factorization obtained from the original set of points.

\begin{figure}
	\centering
	\includegraphics[width=0.32\textwidth]{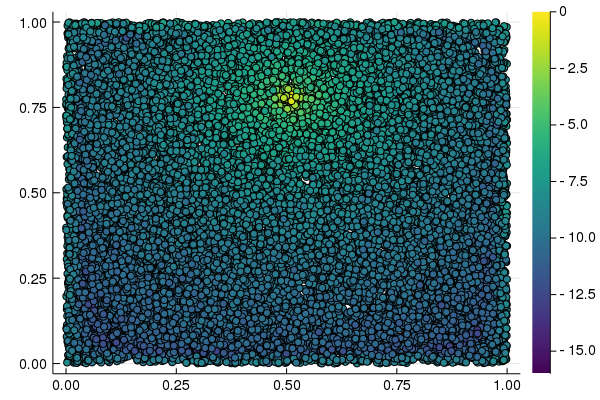}
	\includegraphics[width=0.32\textwidth]{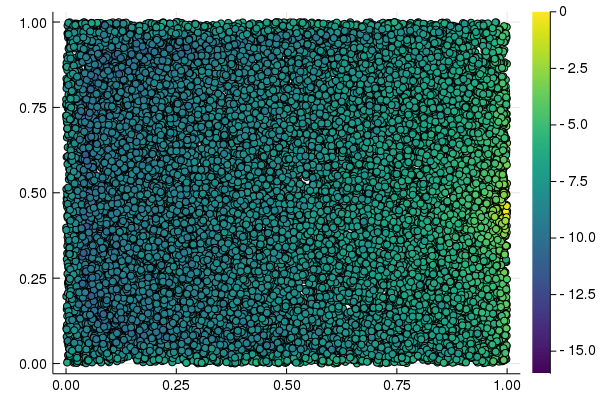}
	\includegraphics[width=0.32\textwidth]{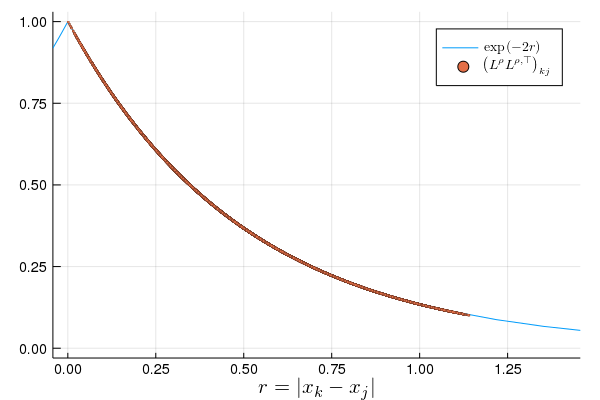}
	\caption{{Weaker screening between boundary points.}
	Left and center: $i$\textsuperscript{th} (left) and $j$\textsuperscript{th} (center) column of the Cholesky factor $L$ (normalized to unit diagonal) of $\KM$ in maximin ordering, where $x_i$ is an interior point and $x_j$ is near the boundary.
	Although $l[i]$ is of the order of $l[j]$, the exponential decay of $L_{:,j}$ near the boundary is significantly weakened by the absence of Dirichlet boundary conditions. Right: approximate correlations $\bigl\{ (L^{\rho} L^{\rho, \top} )_{kj} \bigr\}_{k \in \I}$ (with $\rho = 3.0$) and true covariance function $\exp(-2r)$ with $r = |x_k - x_j|$.
	Correlations between $x_j$ and  remaining points are captured accurately, despite the weakened exponential decay near the boundary.}
	\label{fig-boundaryEffects}
\end{figure}

\subsection{Nuggets and measurement errors}
\label{ssec-nugget}
In the Gaussian process regression setting it is common to to model measurement error by adding a \emph{nugget} $\sigma^2 \Id$ to the covariance matrix:
\begin{equation}
	\tilde{\KM} = \KM + \sigma^2 \Id .
\end{equation}
The addition of a diagonal matrix diminishes the screening effect and thus the accuracy of \cref{alg-ICholesky}.
This problem can be avoided by rewriting the modified covariance matrix $\tilde{\KM}$ as
\begin{equation}
	\tilde{\KM} = \KM ( \sigma^2 \IKM + \Id ),
\end{equation}
where $\IKM \defeq \KM^{-1}$.
As noted in \cref{sssec-spFacIKM}, $\IKM$ can be interpreted as a discretized partial differential operator and has near-sparse Cholesky factors in the reverse elimination ordering.
Adding a multiple of the identity to $\IKM$ amounts to adding a zeroth-order term to the underlying PDE and thus preserves the sparsity of the Cholesky factors.
This leads to the sparse decomposition
\begin{equation}
	\tilde{\KM} = L L^{\top} \rP \tilde{L} \tilde{L}^{\top} \rP,
\end{equation}
where $\rP$ is the order-reversing permutation and $\tilde{L}$ is the
Cholesky factor of $\rP ( \sigma^2 \IKM + \Id ) \rP$.
\cref{fig-cholVarySigma} shows that the exponential decay of these Cholesky factors is robust with respect $\sigma$.

\begin{figure}
	\centering
	\includegraphics[width=0.24\textwidth]{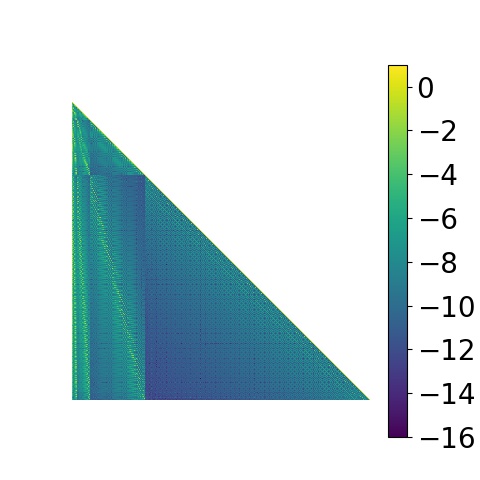}
	\includegraphics[width=0.24\textwidth]{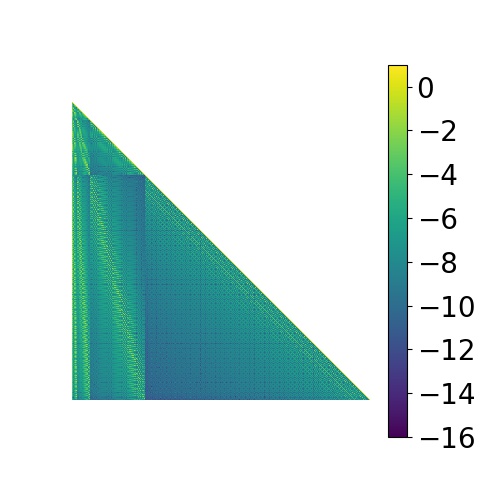}
	\includegraphics[width=0.24\textwidth]{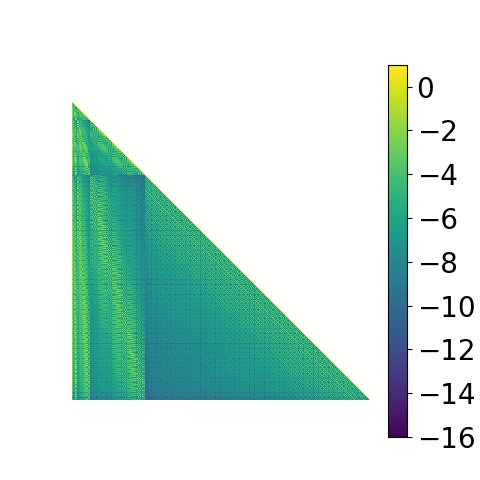}
	\includegraphics[width=0.24\textwidth]{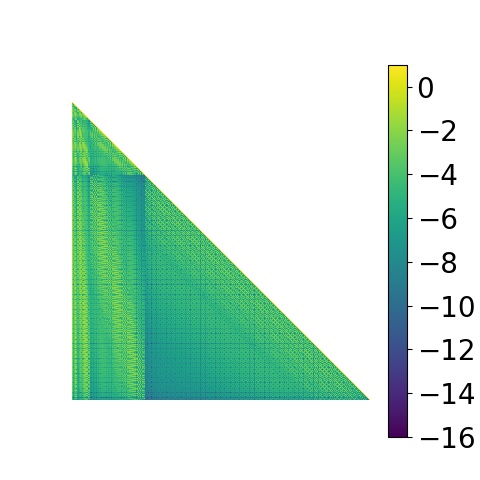}
	\includegraphics[width=0.24\textwidth]{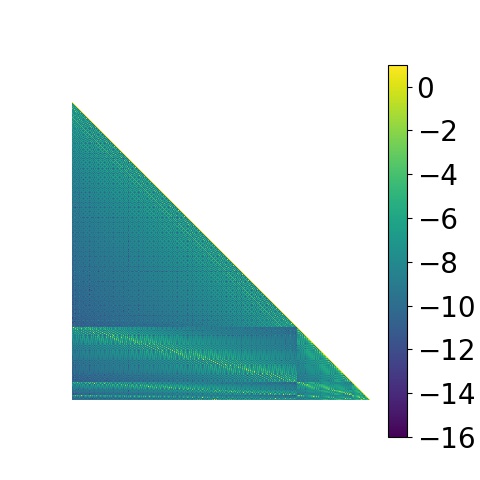}
	\includegraphics[width=0.24\textwidth]{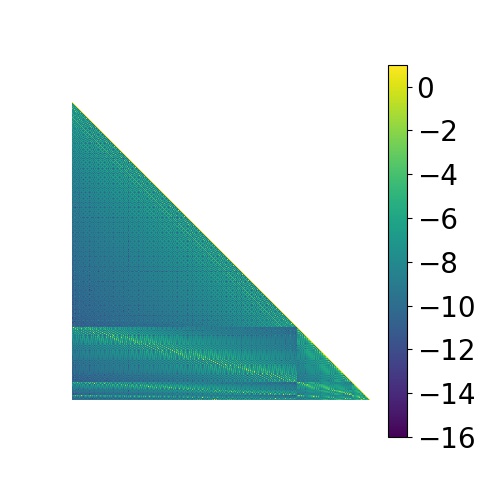}
	\includegraphics[width=0.24\textwidth]{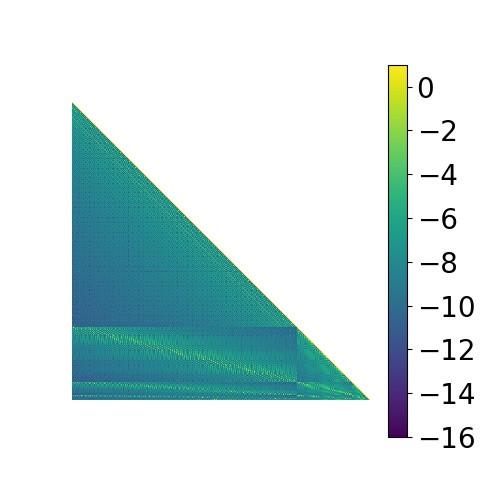}
	\includegraphics[width=0.24\textwidth]{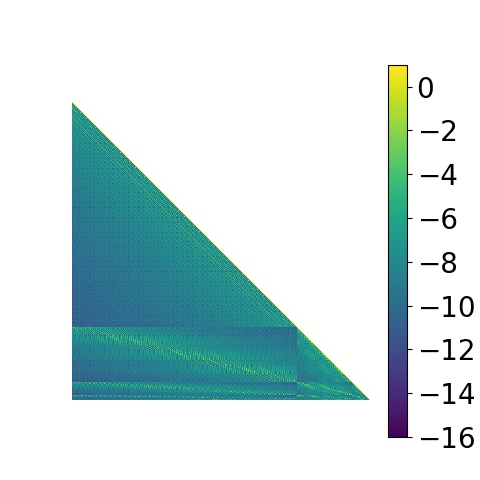}
	\caption{(Lack of) robustness to varying size of the nugget:
	We plot the $\log_{10}$ of the magnitude of the Cholesky factors of $\KM + \sigma^2 \Id$ in maximin ordering (first row) and of $\IKM + \sigma^2$ in reverse maximin ordering (second row).
	As we increase $\sigma^2 \in [0.0, 0.1, 1.0, 10.0]$ from left to right the decay of the Cholesky factors of $\KM + \sigma^2 \Id$ deteriorates, and that of the factors of $\IKM + \sigma^2 \Id $ is preserved.}
	\label{fig-cholVarySigma}
\end{figure}

This idea can be turned into an algorithm by first approximately computing $L$ using \cref{alg-ICholesky};
then using $L$ to approximate $\IKM$, which can be done in near-linear complexity by exploiting sparsity;
and then approximating $\tilde{L}$, again using \cref{alg-ICholesky}.
While this algorithm is asymptotically efficient, our preliminary results suggest that the additional inversion step significantly increases the constants featured in the approximation accuracy.
Therefore, when low accuracy is sufficient, we instead recommend simply applying \cref{alg-ICholesky} to the matrix $\KM$.
This preserves the original approximation accuracy and the matrix inversion
can then efficiently be performed using iterative methods such as conjugate gradient (CG) \cite{shewchuk1994introduction} by taking advantage of the fast matrix-vector multiplication obtained from the sparse factorization.
For small values of $\sigma$ (which would lead to slow convergence of CG) we can directly apply \cref{alg-ICholesky} to $\tilde{\KM}$.
For large values of $\sigma$, $\tilde{\KM}$ will be well conditioned and the convergence of $CG$ is fast.
For intermediate values of $\sigma$, we can apply \cref{alg-ICholesky} to $\tilde{\KM}$ and use the resulting factors as a preconditioner for CG.
Sampling from $\N(0,\tilde{\KM})$ can be done by adding independent samples from $\N(0,\KM)$ and  $\N(0, \sigma^2 \Id)$.
Approximations of the log-determinant could be obtained either by applying \cref{alg-ICholesky} directly to $\tilde{\KM}$ (with some loss of accuracy) or by combining iterative methods \cite{saibaba2017randomized,fitzsimons2017entropic} with the fast matrix-vector multiplication obtained from the sparse factorization of $\KM$.
Just like CG, these methods benefit from the fact that we can work with well-conditioned matrices for small and large $\sigma$.
A detailed investigation of the efficiency of the above mentioned strategies for computing with nuggets is beyond the scope of this work.

\subsection{Numerical results}
\label{ssec-numres}

\begin{figure}
	\centering
	\includegraphics[scale=0.30]{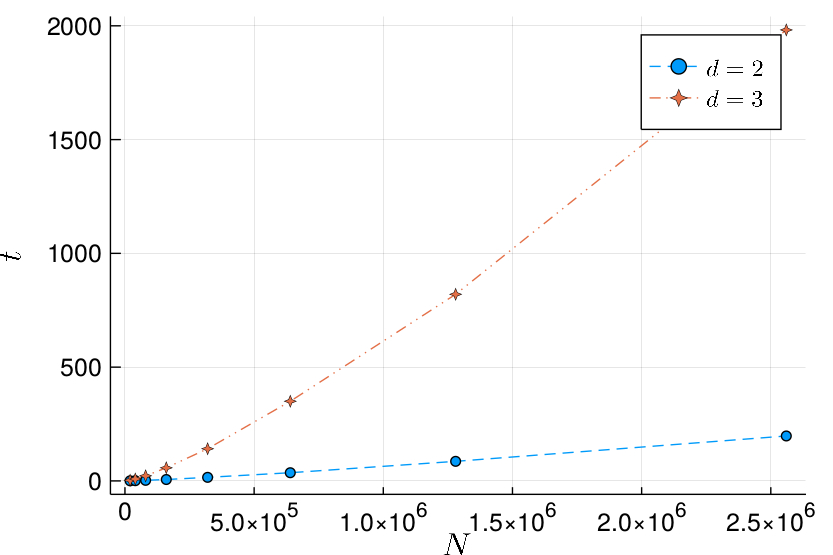}
	\includegraphics[scale=0.30]{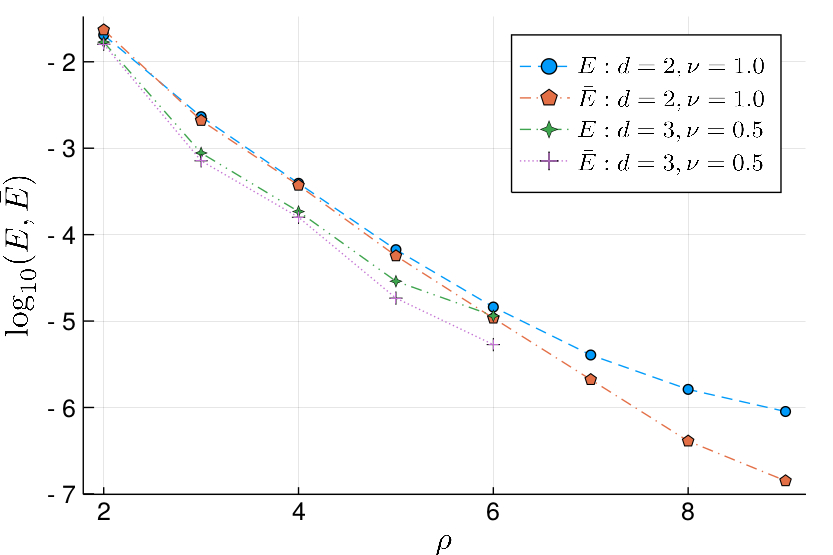}
	\includegraphics[scale=0.30]{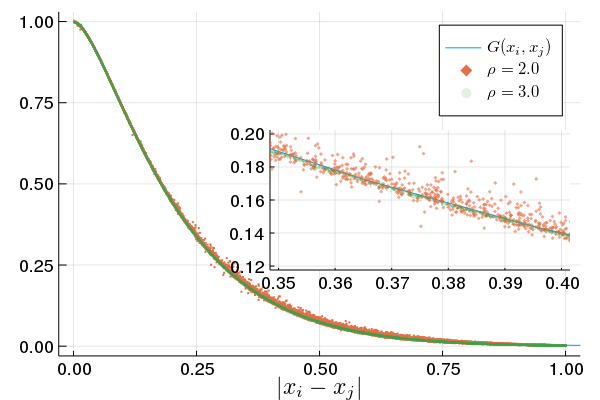}
	\includegraphics[scale=0.30]{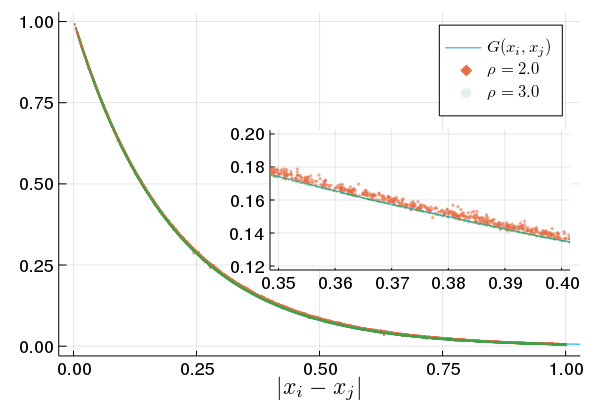}
	\caption{First panel:
	the increase in computational time taken by the Cholesky factorization, as $N$ increases (for $\rho = 3.0$).
	Second panel:
	the exponential decay of the relative error in Frobenius norm, as $\rho$ is increased.
	In the third ($d=2$) and fourth panel ($d=3$), we see the comparison of the approximate and true covariance for $\rho = 2.0$ and $\rho = 3.0$.}
	\label{fig-errRho}
\end{figure}

We will now present numerical evidence in support of our results.
All experiments reported below were run on a workstation using an Intel\textsuperscript{\textregistered}Core\texttrademark i7-6400 CPU with 4.00GHz and 64\,GB of RAM.
The time-critical parts of the code are run on a single thread, leaving the exploration of parallelism to future work.
The Julia scripts implementing the experiments can be found online under \url{https://github.com/f-t-s/nearLinKernel}.
In the following, $\nnz(L)$ denotes the number of nonzero entries of the lower-triangular factor $L$;
$\tsos$ denotes the time taken by \cref{alg-sortSparse} to compute the maximin ordering $\prec$ and sparsity pattern $S_{\rho}$;
$\tent$ denotes the time taken to compute the entries of $\KM$ on $S_{\rho}$;
and $\tich$ denotes the time taken to perform \cref{alg-ICholesky} (\ICH), all measured in seconds.
The relative error in Frobenius norm is approximated by
\begin{equation}
	E
	\defeq \frac{\norm{ LL^{\top} - \KM }_{\FRO}}{\norm{ \KM }_{\FRO}}
	\approx \frac{\sqrt{ \sum_{k=1}^m \bignorm{ \bigl( L L^{\top} - \KM \bigr)_{i_k j_k} }^2}} {\sqrt{\sum_{k=1}^m \norm{ \KM_{i_k j_k} }^2}} ,
\end{equation}
where the $m = 500000$ pairs of indices $i_k, j_k \sim \unif(\I)$ are independently and uniformly distributed in $\I$.
This experiment is repeated 50 times and the resulting mean and standard deviation (in brackets) are reported.
For measurements in $[0,1]^d$, in order to isolate the boundary effects, we also consider the quantity $\bar{E}$ which is defined as $E$, but with only those sample $i_k,j_k$ for which $x_{i_k}, x_{j_k} \in [0.05, 0.95]^d$.
Most of our experiments will use the Mat{\'e}rn class of covariance functions \cite{matern1960spatial}, defined by
\begin{equation}\label{eqn-matern}
  	\GMat_{l, \nu} (x, y) \defeq \frac{2^{1-\nu}}{\Gamma(\nu)} \left( \frac{\sqrt{2\nu} |x - y |}{l} \right)^{\nu} K_{\nu} \left( \frac{\sqrt{2\nu} | x - y |} {l} \right),
\end{equation}
where $K_{\nu}$ is the modified Bessel function of second kind \cite[Section 9.6]{abramowitz1964handbook} and $\nu$, $l$ are parameters describing the degree of smoothness, and the length-scale of interactions, respectively \cite{rasmussen2006gaussian}.
In \cref{fig-MaternComp}, the Mat{\'e}rn kernel is plotted for different degrees of smoothness.
The Mat{\'e}rn covariance function is used in many branches of statistics and machine learning to model random fields with finite order of smoothness \cite{guttorp2006studies, rasmussen2006gaussian}.

\begin{figure}[t]
	\centering
	\includegraphics[width=0.4\textwidth]{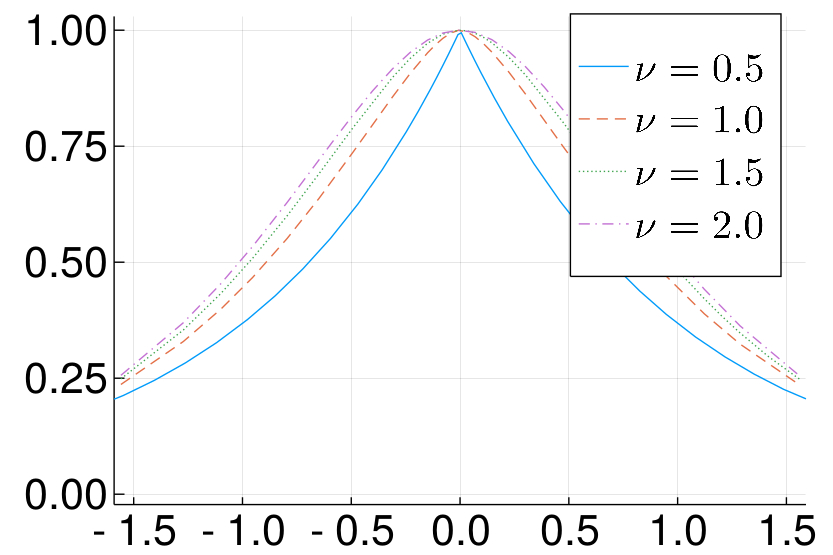}
	\includegraphics[width=0.4\textwidth]{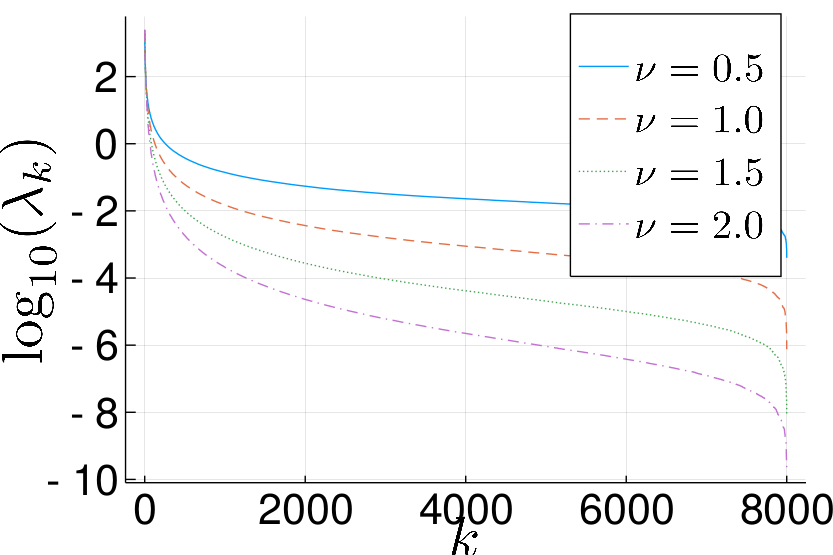}
	\caption{Mat{\'e}rn kernels for different values of $\nu$ (left), and the spectrum of $\KM$, for $2000$ points $x_i \in [0,1]^2$ (right).
	Smaller values of $\nu$ correspond to stronger singularities at zero and hence lower degrees of smoothness of the associated Gaussian process.}
	\label{fig-MaternComp}
\end{figure}

As observed by \cite{whittle1954stationary, whittle1963stochastic}, the Mat{\'e}rn kernel is the Green's function of an elliptic PDE of possibly fractional order $2 ( \nu + d/2 )$ in the whole space.
Therefore, for $2 ( \nu + d/2 ) \in \Naturals$, the Mat{\'e}rn kernel falls into the framework of our theoretical results, up to the behavior at the boundary discussed in \cref{ssec-noBoundary}.
Since the locations of our points will be chosen at random, some of the points will be very close to each other, resulting in an almost singular matrix $\KM$ that can become nonpositive under the approximation introduced by {\ICH}.
If \cref{alg-ICholesky} encounters a nonpositive pivot $A_{ii}$, then we set the corresponding column of $L$ to zero, resulting in a low-rank approximation of the original covariance matrix.
We report the rank of $L$ in our experiments and note that we obtain a full-rank approximation for moderate values of $\rho$.

We begin by investigating the scaling of our algorithm as $N$ increases.
To this end, we consider $\nu = 0.5$ (the exponential kernel), $l = 0.2$ and choose $N$ randomly distributed points in $[0,1]^d$ for $d \in \{2,3\}$.
The results are summarized in \cref{tab-maternN2d} and \cref{tab-matern3d}, and in \cref{fig-errRho}, and confirm the near-linear computational complexity of our algorithm.

\begin{table}
	\caption{$\GMat_{\nu,l}$, with $\nu = 0.5$, $l = 0.2$, $\rho = 3.0$, and $d = 2$.}
	\label{tab-maternN2d}
	\tiny
	\centering
	\input{./tables/vary_N_matern_l0.2_nu0.5_rho3.0_d2.tex}
\end{table}

\begin{table}
	\caption{$\GMat_{\nu,l}$, with $\nu = 0.5$, $l = 0.2$, $\rho = 3.0$, and $d = 3$.}
	\label{tab-maternN3d}
	\tiny
	\centering
	\input{./tables/vary_N_matern_l0.2_nu0.5_rho3.0_d3.tex}
\end{table}

\begin{table}
	\caption{$\GMat_{\nu,l}$, with $\nu = 1.0$, $l = 0.2$, $N = 10^6$, and $d = 2$.}
	\label{tab-matern2d}
	\tiny
	\centering
	\input{./tables/vary_rho_matern_l0.2_nu1.0_N1000000_d2.tex}
\end{table}

\begin{table}
	\caption{$\GMat_{\nu,l}$, with $\nu = 0.5$, $l = 0.2$, $N = 10^6$, and $d = 3$.}
	\label{tab-matern3d}
	\tiny
	\centering
	\input{./tables/vary_rho_matern_l0.2_nu0.5_N1000000_d3.tex}
\end{table}

\begin{table}
	\caption{We tabulate the approximation rank and error for $\rho = 5.0$ and $N = 10^6$ points uniformly distributed in $[0,1]^3$.
	The covariance function is $\GMat_{\nu,0.2}$ for $\nu$ ranging around $\nu = 0.5$ and $\nu = 1.5$.
	Even though the intermediate values of $\nu$ correspond to a fractional order elliptic PDE, the behavior of the approximation stays the same.}
	\label{tab-maternFrac3d}
	\tiny
	\centering
	\input{./tables/vary_nu_matern_l0.2_rho5.0_N1000000_d3.tex}
\end{table}

\begin{table}
	\caption{$\GCauchy_{l,\alpha,\beta}$ for $(l,\alpha,\beta) = (0.4,0.5,0.025)$ (first table) and $(l,\alpha,\beta) = (0.2,1.0,0.20)$ (second table), for $N = 10^6$ and $d = 2$.}
	\label{tab-cauchy}
	\tiny
	\centering
	\input{./tables/vary_rho_cauchy_l0.4_alpha0.5_beta0.025_N1000000_d2.tex} \\
	\input{./tables/vary_rho_cauchy_l0.2_alpha1.0_beta0.2_N1000000_d2.tex}
\end{table}

\begin{table}
	\caption{$\GMat_{\nu,l}$ for $\nu = 0.5$, $l = 0.2$, and $\rho = 3.0$ with $N = 10^6$ points chosen as in \cref{fig-dzpoints}.}
	\label{tab-materndz}
	\tiny
	\centering
	\input{./tables/vary_dz_matern_l0.2_rho3.0_N1000000.tex}
\end{table}

Next, we investigate the trade-off between computational efficiency and accuracy of the approximation.
To this end, we choose $d = 2$, $\nu = 1.0$ and $d = 3$, $\nu = 0.5$, corresponding to fourth-order equations in two and three dimensions.
We choose $N = 10^6$ data points $x_i \sim \unif([0,1]^d)$ and apply our method with different values of $\rho$.
The results of these experiments are tabulated in \cref{tab-matern2d,tab-matern3d} and the impact of $\rho$ on the approximation error is visualized in \cref{fig-errRho}.

While our theoretical results only cover integer-order elliptic PDEs, we observe no practical difference between the numerical results for Mat{\'e}rn kernels corresponding to integer- and fractional-order smoothness.
As an illustration, for the case $d = 3$, we provide approximation results for $\nu$ ranging around $\nu =0.5$ (corresponding to a fourth-order elliptic PDE) and $\nu = 1.5$ (corresponding to a sixth-order elliptic PDE).
As seen in \cref{tab-maternFrac3d}, the results vary continuously as $\nu$ changes, with no qualitative differences between the behavior for integer- and fractional-order PDEs.
To further illustrate the robustness of our method, we consider the Cauchy class of covariance functions introduced in \cite{gneiting2004stochastic}
\begin{equation}
  	\GCauchy_{l,\alpha,\beta}(x,y) \defeq \left( 1 + \left( \frac{\left|x-y\right|}{l}\right)^{\alpha} \right)^{-\frac{\beta}{\alpha}}.
\end{equation}
As far as we are aware, the Cauchy class has not been associated to an elliptic PDE.
Furthermore, it does not have exponential decay in the limit $|x-y| \to \infty$, which allows us to emphasize the point that the exponential decay of the error is \emph{not} due to the exponential decay of the covariance function itself.
\Cref{tab-cauchy} gives the results for $(l,\alpha,\beta) = (0.4, 0.5, 0.025)$ and $(l,\alpha,\beta) = (0.2, 1.0, 0.2)$.

In Gaussian process regression, the ambient dimension $d$ is typically too large to ensure computational efficiency of our algorithm.
However, since our algorithm only requires access to pairwise distances between points, it can take advantage of low intrinsic dimension
of the dataset.
We might be concerned that in this case, interaction through the higher dimensional ambient space will disable the screening effect.
As a first demonstration that this is not the case, we will draw $N = 10^6$ points in $[0,1]^2$ and equip them with a third component according to $x_{i}^{(3)} \defeq - \dz \sin ( 6 x_i^{(1)} ) \cos ( 2  ( 1 - x_i^{(2)} ) ) + \xi_i 10^{-3}$, for $\xi_i$ i.i.d.\ standard Gaussian.
\Cref{fig-dzpoints} shows the resulting point sets for different values of $\dz$, and \cref{tab-materndz} shows that the approximation is robust to increasing values of $\dz$.

\begin{figure}
	\centering
	\includegraphics[width=0.32\textwidth]{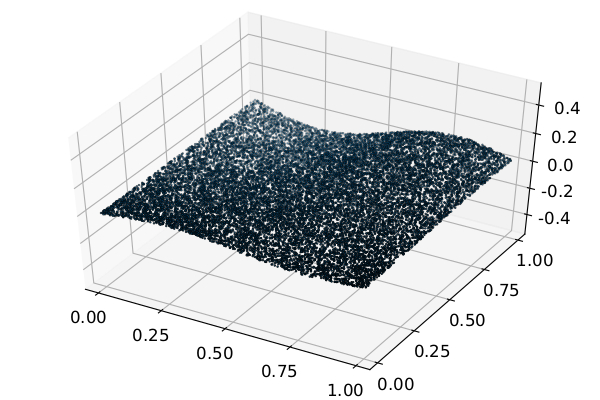}
	\includegraphics[width=0.32\textwidth]{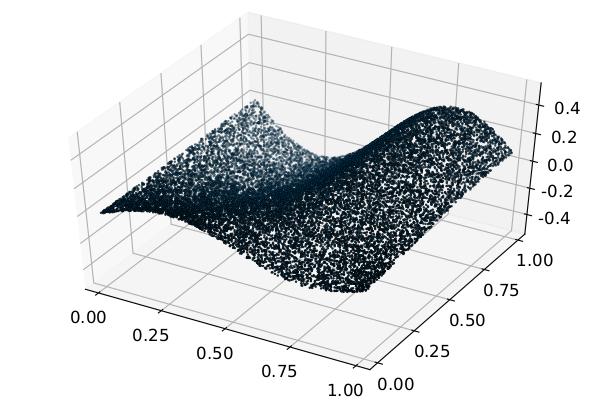}
	\includegraphics[width=0.32\textwidth]{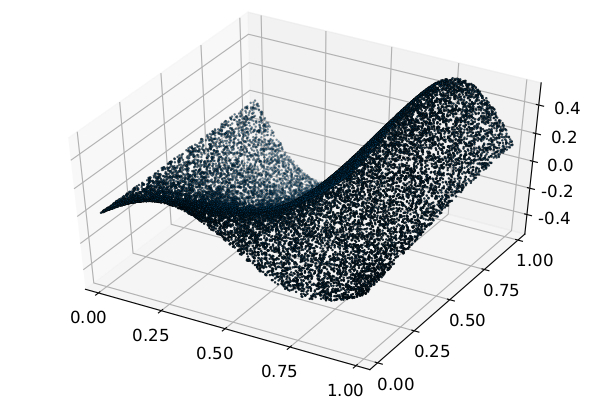}
	\caption{A two-dimensional point cloud deformed into a two-dimensional submanifold of $\Reals^3$, with $\dz \in \{ 0.1, 0.3, 0.5\}$.}
	\label{fig-dzpoints}
\end{figure}

An appealing feature of our method is that it can be formulated in terms of the pairwise distances alone.
This means that the algorithm will automatically exploit any low-dimensional structure in the dataset.
In order to illustrate this feature, we artificially construct a dataset with low-dimensional structure by randomly rotating four low-dimensional structures into a $20$-dimensional ambient space (see \cref{fig-spiralEllipse}).
\Cref{tab-expHighDim} shows that the resulting approximation is even better than the one obtained in dimension $3$, illustrating that our algorithm did indeed exploit the low intrinsic dimension of the dataset.

\begin{figure}
	\centering
	\includegraphics[width=0.24\textwidth]{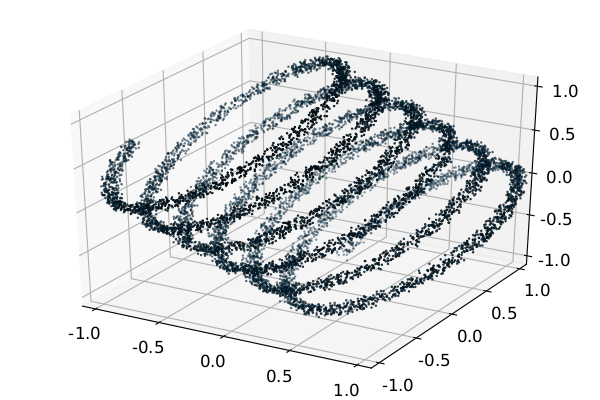}
	\includegraphics[width=0.24\textwidth]{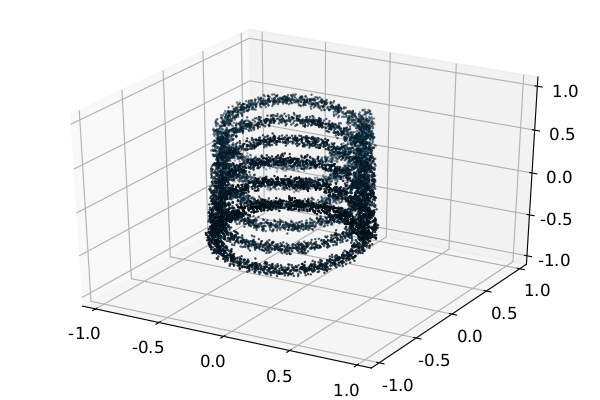}
	\includegraphics[width=0.24\textwidth]{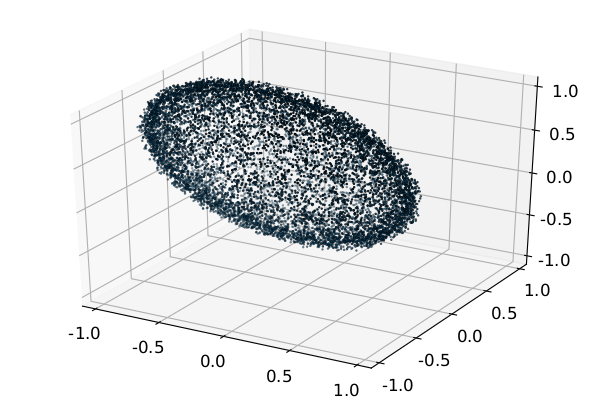}
	\includegraphics[width=0.24\textwidth]{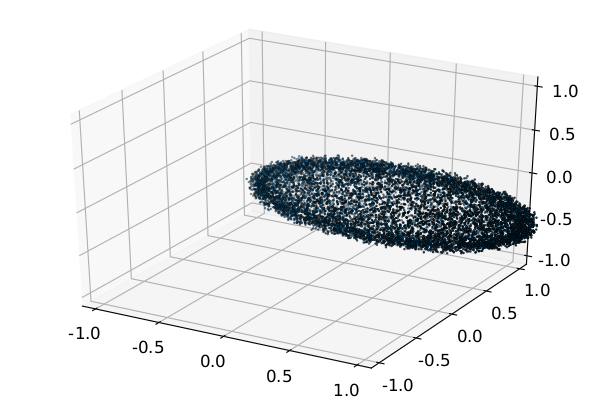}
	\caption{We construct a high-dimensional dataset with low-dimensional structure by rotating the above structures at random into a $20$-dimensional ambient space.}
	\label{fig-spiralEllipse}
\end{figure}

\begin{table}
	\caption{$\GMat_{\nu,l}$ for $\nu = 0.5$, $l = 0.5$, and $N = 10^6$ points as in \cref{fig-spiralEllipse}.}
	\label{tab-expHighDim}
	\scriptsize
	\centering
	\input{./tables/vary_rho_highDim_l_0.5_nu_0.5_N1000000.tex}
\end{table}

%% file: tables/vary_N_matern_l0.2_nu0.5_rho3.0_d2.tex
\begin{tabular}{r|ccccccc}
\toprule
     $N$       & $\nnz(L)/N^2$ & $\rank(L)$ & $\tsos$ & $\tent$ & $\tich$ & $ E $               & $ \bar{E} $         \\ \hline
  $20000$ &     5.26e-03  &     20000  &   0.71  &   0.81  &   0.42  & 1.25e-03 (3.68e-06) & 1.11e-03 (3.01e-06) \\
  $20000$ &     5.26e-03  &     20000  &   0.71  &   0.81  &   0.42  & 1.25e-03 (3.68e-06) & 1.11e-03 (3.01e-06) \\
  $40000$ &     2.94e-03  &     40000  &   1.21  &   1.19  &   1.00  & 1.27e-03 (3.32e-06) & 1.12e-03 (3.56e-06) \\
  $80000$ &     1.62e-03  &     80000  &   2.72  &   2.82  &   2.55  & 1.30e-03 (3.20e-06) & 1.21e-03 (3.29e-06) \\
 $160000$ &     8.91e-04  &    160000  &   6.86  &   6.03  &   6.11  & 1.28e-03 (3.57e-06) & 1.16e-03 (3.32e-06) \\
 $320000$ &     4.84e-04  &    320000  &  17.22  &  13.79  &  15.66  & 1.23e-03 (3.19e-06) & 1.11e-03 (2.40e-06) \\
 $640000$ &     2.63e-04  &    640000  &  41.40  &  31.02  &  36.02  & 1.24e-03 (2.58e-06) & 1.09e-03 (3.02e-06) \\
$1280000$ &     1.41e-04  &   1280000  &  98.34  &  65.96  &  85.99  & 1.23e-03 (3.72e-06) & 1.10e-03 (3.74e-06) \\
$2560000$ &     7.55e-05  &   2560000  & 233.92  & 148.43  & 197.52  & 1.16e-03 (2.82e-06) & 1.04e-03 (3.36e-06) \\
\bottomrule
\end{tabular}

%% file: tables/vary_N_matern_l0.2_nu0.5_rho3.0_d3.tex
\begin{tabular}{r|ccccccc}
\toprule
       $N$       & $\nnz(L)/N^2$ & $\rank(L)$ & $\tsos$ & $\tent$ & $\tich$  & $ E $               & $ \bar{E} $         \\ \hline
  $20000$ &     1.30e-02  &     20000  &   1.61  &   1.44  &    2.94  & 1.49e-03 (5.00e-06) & 1.20e-03 (5.09e-06) \\
  $40000$ &     7.60e-03  &     40000  &   3.26  &   3.32  &    8.33  & 1.21e-03 (4.29e-06) & 9.91e-04 (3.72e-06) \\
  $80000$ &     4.35e-03  &     80000  &   7.46  &   7.64  &   22.46  & 1.06e-03 (3.74e-06) & 8.51e-04 (2.93e-06) \\
 $160000$ &     2.45e-03  &    160000  &  20.95  &  18.42  &   57.64  & 9.81e-04 (2.33e-06) & 7.88e-04 (3.23e-06) \\
 $320000$ &     1.37e-03  &    320000  &  53.58  &  40.72  &  141.46  & 9.27e-04 (2.26e-06) & 7.53e-04 (2.72e-06) \\
 $640000$ &     7.61e-04  &    640000  & 133.55  &  96.67  &  350.10  & 8.98e-04 (3.25e-06) & 7.25e-04 (3.02e-06) \\
$1280000$ &     4.19e-04  &   1280000  & 312.43  & 212.57  &  820.07  & 8.59e-04 (2.79e-06) & 7.00e-04 (2.87e-06) \\
$2560000$ &     2.29e-04  &   2560000  & 795.68  & 480.17  & 1981.92  & 8.96e-04 (2.76e-06) & 7.73e-04 (4.28e-06) \\
\bottomrule
\end{tabular}

%% file: tables/vary_rho_matern_l0.2_nu1.0_N1000000_d2.tex
\begin{tabular}{r|ccccccc}
\toprule
            & $\nnz(L)/N^2$ & $\rank(L)$ & $\tsos$ & $\tent$ & $\tich$  & $ E $               & $ \bar{E} $         \\ \hline
$\rho= 2.0$ &     8.78e-05  &    254666  &  38.06  &  33.72  &   17.54  & 2.04e-02 (1.73e-02) & 2.34e-02 (2.75e-02) \\
$\rho= 3.0$ &     1.76e-04  &    964858  &  71.07  &  67.85  &   61.35  & 2.32e-03 (6.02e-06) & 2.09e-03 (7.50e-06) \\
$\rho= 4.0$ &     2.90e-04  &    999810  & 115.07  & 112.56  &  152.93  & 3.92e-04 (1.44e-06) & 3.72e-04 (2.32e-06) \\
$\rho= 5.0$ &     4.26e-04  &    999999  & 165.91  & 166.60  &  312.19  & 6.70e-05 (2.98e-07) & 5.68e-05 (2.55e-07) \\
$\rho= 6.0$ &     5.83e-04  &   1000000  & 227.62  & 229.76  &  566.94  & 1.45e-05 (6.69e-08) & 1.08e-05 (5.01e-08) \\
$\rho= 7.0$ &     7.59e-04  &   1000000  & 292.52  & 300.65  &  944.33  & 4.05e-06 (4.96e-08) & 2.10e-06 (1.69e-08) \\
$\rho= 8.0$ &     9.53e-04  &   1000000  & 363.90  & 380.07  & 1476.71  & 1.62e-06 (2.30e-08) & 4.08e-07 (9.47e-09) \\
$\rho= 9.0$ &     1.16e-03  &   1000000  & 447.47  & 467.07  & 2200.32  & 8.98e-07 (1.44e-08) & 1.42e-07 (5.14e-09) \\
\bottomrule
\end{tabular}

%% file: tables/vary_rho_matern_l0.2_nu0.5_N1000000_d3.tex
\begin{tabular}{r|ccccccc}
\toprule
            & $\nnz(L)/N^2$ & $\rank(L)$ & $\tsos$  & $\tent$ & $\tich$   & $ E $               & $ \bar{E} $         \\ \hline
$\rho= 2.0$ &     1.87e-04  &    998046  &   87.83  &  56.44  &    85.20  & 1.69e-02 (6.89e-04) & 1.60e-02 (3.36e-04) \\
$\rho= 3.0$ &     5.17e-04  &   1000000  &  226.84  & 158.42  &   599.86  & 8.81e-04 (3.21e-06) & 7.15e-04 (2.99e-06) \\
$\rho= 4.0$ &     1.05e-03  &   1000000  &  446.52  & 326.27  &  2434.52  & 1.85e-04 (5.37e-07) & 1.59e-04 (5.30e-07) \\
$\rho= 5.0$ &     1.82e-03  &   1000000  &  747.65  & 567.06  &  7227.45  & 2.89e-05 (1.94e-07) & 1.84e-05 (1.15e-07) \\
$\rho= 6.0$ &     2.82e-03  &   1000000  & 1344.59  & 928.27  & 17640.58  & 1.15e-05 (1.06e-07) & 5.34e-06 (5.34e-08) \\
\bottomrule
\end{tabular}

%% file: tables/vary_nu_matern_l0.2_rho5.0_N1000000_d3.tex
\begin{tabular}{r|cccccccc}
\toprule
            & $\nu = 0.3$ & $\nu = 0.5$ & $\nu = 0.7$ & $\nu = 0.9$ & $\nu = 1.1$ & $\nu = 1.3$ & $\nu = 1.5$ & $\nu = 1.7$ \\ \hline
 $\rank(L)$ &     1000000 &     1000000 &     1000000 &     1000000 &     1000000 &     1000000 &     1000000 &      999893 \\
      $ E $ &    7.04e-05 &    2.89e-05 &    2.49e-05 &    3.58e-05 &    6.03e-05 &    8.77e-05 &    1.18e-04 &    1.46e-04 \\
            &  (3.98e-07) &  (1.79e-07) &  (1.11e-07) &  (1.19e-07) &  (2.37e-07) &  (3.06e-07) &  (4.52e-07) &  (5.39e-07) \\
$ \bar{E} $ &    5.19e-05 &    1.85e-05 &    1.77e-05 &    2.82e-05 &    4.88e-05 &    6.87e-05 &    9.06e-05 &    1.13e-04 \\
            &  (2.26e-07) &  (1.18e-07) &  (8.11e-08) &  (1.30e-07) &  (2.37e-07) &  (3.50e-07) &  (5.14e-07) &  (5.45e-07) \\
\bottomrule
\end{tabular}

%% file: tables/vary_rho_cauchy_l0.4_alpha0.5_beta0.025_N1000000_d2.tex
\begin{tabular}{r|cccccccc}
\toprule
            & $\rho = 2.0$ & $\rho = 3.0$ & $\rho = 4.0$ & $\rho = 5.0$ & $\rho = 6.0$ & $\rho = 7.0$ & $\rho = 8.0$ & $\rho = 9.0$ \\ \hline
 $\rank(L)$ &       999923 &      1000000 &      1000000 &      1000000 &      1000000 &      1000000 &      1000000 &      1000000 \\
      $ E $ &     4.65e-04 &     5.98e-05 &     2.36e-05 &     1.19e-05 &     4.84e-06 &     4.17e-06 &     2.25e-06 &     1.42e-06 \\
            &   (4.23e-07) &   (1.56e-07) &   (9.53e-08) &   (6.32e-08) &   (4.14e-08) &   (4.99e-08) &   (1.86e-08) &   (1.64e-08) \\
$ \bar{E} $ &     3.81e-04 &     3.49e-05 &     9.83e-06 &     4.65e-06 &     1.47e-06 &     8.49e-07 &     4.25e-07 &     2.12e-07 \\
            &   (4.98e-07) &   (1.59e-07) &   (5.56e-08) &   (2.63e-08) &   (7.73e-09) &   (1.04e-08) &   (4.81e-09) &   (3.24e-09) \\
\bottomrule
\end{tabular}

%% file: tables/vary_rho_cauchy_l0.2_alpha1.0_beta0.2_N1000000_d2.tex
\begin{tabular}{r|cccccccc}
\toprule
            & $\rho = 2.0$ & $\rho = 3.0$ & $\rho = 4.0$ & $\rho = 5.0$ & $\rho = 6.0$ & $\rho = 7.0$ & $\rho = 8.0$ & $\rho = 9.0$ \\ \hline
 $\rank(L)$ &       999547 &      1000000 &      1000000 &      1000000 &      1000000 &      1000000 &      1000000 &      1000000 \\
      $ E $ &     1.08e-03 &     1.36e-04 &     2.89e-05 &     2.35e-05 &     5.33e-06 &     3.25e-06 &     2.53e-06 &     1.68e-06 \\
            &   (5.02e-06) &   (6.27e-07) &   (2.63e-07) &   (3.01e-07) &   (6.15e-08) &   (5.74e-08) &   (4.84e-08) &   (4.25e-08) \\
$ \bar{E} $ &     7.23e-04 &     8.96e-05 &     1.17e-05 &     5.65e-06 &     1.09e-06 &     5.84e-07 &     4.03e-07 &     2.40e-07 \\
            &   (4.07e-06) &   (2.63e-07) &   (7.10e-08) &   (1.47e-07) &   (7.71e-09) &   (5.48e-09) &   (3.44e-09) &   (2.23e-09) \\
\bottomrule
\end{tabular}

%% file: tables/vary_dz_matern_l0.2_rho3.0_N1000000.tex
\begin{tabular}{r|cccccccc}
\toprule
              & $\delta_{z} = 0.0$ & $\delta_{z} = 0.1$ & $\delta_{z} = 0.2$ & $\delta_{z} = 0.3$ & $\delta_{z} = 0.4$ & $\delta_{z} = 0.5$ & $\delta_{z} = 0.6$ \\%& $\delta_{z} = 0.7$ \\ \hline
$\frac{\nnz(L)}{N^2}$ &           1.76e-04 &           1.77e-04 &           1.78e-04 &           1.80e-04 &           1.82e-04 &           1.84e-04 &           1.85e-04 \\%&            1.87e-04 \\
    $ \tich $ &              61.92 &              62.15 &              62.81 &              64.27 &              64.87 &              65.50 &              66.12 \\%&              67.59 \\
   $\rank(L)$ &            1000000 &            1000000 &            1000000 &            1000000 &            1000000 &            1000000 &            1000000 \\%&            1000000 \\
        $ E $ &           1.17e-03 &           1.11e-03 &           1.28e-03 &           1.60e-03 &           1.72e-03 &           1.89e-03 &           2.11e-03 \\%&           2.17e-03 \\
              &         (2.74e-06) &         (3.00e-06) &         (2.73e-06) &         (4.28e-06) &         (3.95e-06) &         (5.11e-06) &         (5.07e-06) \\%&         (8.00e-06) \\
\bottomrule
\end{tabular}

%% file: tables/vary_rho_highDim_l_0.5_nu_0.5_N1000000.tex
\begin{tabular}{r|cccccc}
\toprule
            & $\nnz(L)/N^2$ & $\rank(L)$ & $\tsos$  & $\tent$ & $\tich$   & $ E $               \\ \hline
$\rho= 2.0$ &     1.62e-04  &    997635  &   80.60  &  57.11  &    52.49  & 1.57e-02 (1.13e-03) \\
$\rho= 3.0$ &     3.76e-04  &   1000000  &  173.86  & 135.61  &   248.78  & 2.88e-03 (1.14e-05) \\
$\rho= 4.0$ &     6.76e-04  &   1000000  &  302.98  & 247.74  &   748.62  & 8.80e-04 (4.97e-06) \\
$\rho= 5.0$ &     1.05e-03  &   1000000  &  462.98  & 397.42  &  1802.44  & 3.44e-04 (2.54e-06) \\
$\rho= 6.0$ &     1.49e-03  &   1000000  &  645.56  & 556.72  &  3696.31  & 1.44e-04 (8.76e-07) \\
$\rho= 7.0$ &     2.02e-03  &   1000000  &  891.08  & 758.88  &  6855.23  & 7.61e-05 (5.66e-07) \\
$\rho= 8.0$ &     2.62e-03  &   1000000  & 1248.90  & 990.86  & 11598.66  & 4.57e-05 (4.36e-07) \\
\bottomrule
\end{tabular}

%% file: sec-analysis.tex
\subsection{General Setting}
\label{secgenset}

We will start the analysis in a more general setting than that of Section \cref{sssec-classElliptic}.
Let $\B$ be a separable Banach space with dual space $\B^{\ast}$, and write $\dualprod{ \quark }{ \quark }$ for the duality product between $\B^{\ast}$ and $\B$.
Let $\IK \colon \B \to \B^{\ast}$ be a linear bijection and let $\K \defeq \IK^{-1}$.
Assume $\IK$ to be symmetric and positive (i.e.\ $\dualprod{ \IK u }{ v } = \dualprod{ \IK v }{ u }$ and $\dualprod{ \IK u }{ u }\geq 0$ for $u,v\in \B$).
Let $\norm{ \quark }$ be the quadratic (energy) norm defined by $\norm{ u }^2 \defeq \dualprod{ \IK u }{ u }$ for $u\in \B$ and let $\norm{ \quark }_{\ast}$ be its dual norm defined by
\begin{equation}
 	\label{eq-dual-norm}
	\norm{ \phi }_{\ast} \defeq \sup_{0 \neq u\in \B} \frac{\dualprod{ \phi }{ u }}{\norm{ u }}=[\phi,\K \phi]\text{ for }\phi\in \B^{\ast}.
\end{equation}
Let $\{ \phi_i \}_{i \in \I}$ be linearly independent elements of $\B^{\ast}$ (known as \emph{measurement functions}) and let $\KM \in \Reals^{\I \times \I}$ be the symmetric positive-definite matrix defined by
\begin{equation}
	\KM_{ij} \defeq \dualprod{ \phi_i }{ \K \phi_j } \quad \text{for $i,j \in \I$.}
\end{equation}
We assume that we are given $q\in \mathbb{N}$ and a partition $\I = \bigcup_{1 \leq k \leq q} \J^{(k)}$ of $\I$.
We represent $I\times I$ matrices as $q \times q$ block matrices according to this partition.
Given an $I\times I$ matrix $M$ we write $M_{k,l}$ for the $(k,l)$\textsuperscript{th} block of $M$ and $M_{k_1:k_2,l_1:l_2}$ for the sub-matrix of $M$ defined by blocks ranging from $k_1$ to $k_2$ and
$l_1$ to $l_2$.
Unless specified otherwise we write $L$ for the lower-triangular Cholesky factor of $\KM$ and define
\begin{align}
	\label{eqjehdhgdjh}
	\KM^{(k)} & \defeq \KM_{1:k,1:k}, &
	\IKM^{(k)} & \defeq \KM^{(k),-1}, &
	\IKMC^{(k)} & \defeq \IKM^{(k)}_{k,k} &
	\text{for $1 \leq k \leq q$.}
\end{align}
We interpret the $\{ \J^{(k)} \}_{1 \leq k \leq q}$ as labelling a hierarchy of scales with $\J^{(1)}$ representing the coarsest and $\J^{(q)}$ the finest.
We write $\I^{(k)}$ for $\bigcup_{1 \leq k' \leq k} \J^{(k')}$.

Throughout this section we assume that the ordering of the set $I$ of indices is compatible with the partition $I=\bigcup_{k=1^q} \J^{(k)}$, i.e.\ $k<l$, $i\in \J^{(k)}$ and $j\in \J^{(l)}$ together imply $i\prec j$.
We will write $L$ or $\chol(\KM)$ for the Cholesky factor of $\KM$ in that ordering.

\subsection{Main examples}

We will prove the main results of this section in the setting where $\IK$ is defined as in \cref{sssec-classElliptic} and the $\phi_{i}$ are chosen as in \cref{examp-subsamp,examp-average}.
We will assume (without loss of generality after rescaling) that $\diam(\Omega)\leq 1$.
As described in \cref{fig-ordering}, successive points of the maximin ordering can be gathered into levels, so that, after appropriate rescaling of the measurements, the Cholesky factorization in the maximin ordering falls in the setting of \cref{examp-subsamp}.

\begin{example}
	\label{examp-subsamp}
	Let $s > d/2$.
	For $h, \delta \in (0,1)$ let $\{x_i\}_{i\in \I^{(1)}}\subset \{x_i\}_{i\in \I^{(2)}}\subset \cdots \subset \{x_i\}_{i\in \I^{(q)}}$ be a nested hierarchy of points in
	$\Omega$ that are homogeneously distributed at each scale in the sense of the following three inequalities:
	\begin{enumerate}[label=(\arabic*)]
		\item $\sup_{x\in \Omega} \min_{i\in \I^{(k)}} |x-x_i|\leq h^k$,
		\item $\min_{i\in \I^{(k)}}\inf_{x \in \partial \Omega} |x-x_i|\geq \delta h^k$, and
		\item $ \min_{i,j\in \I^{(k)}: i \neq j}|x_i-x_j|\geq \delta h^k$.
	\end{enumerate}
	Let $\J^{(1)} \defeq \I^{(1)}$ and $\J^{(k)} \defeq \I^{(k)}\setminus \I^{(k-1)}$ for $k\in \{2,\ldots,q\}$.
	Let $\boldsymbol{\delta}$ denote the unit Dirac delta function and choose
	\begin{equation}
		\phi_i \defeq h^{\frac{kd}{2}} \boldsymbol{\delta}(x-x_i) \text{ for }i\in \J^{(k)}\text{ and }k\in \{1,\ldots,q\} .
	\end{equation}
\end{example}

Given subsets $\tI, \tJ \subset \I$ we extend a matrix $M \in \Reals^{\tI \times \tJ}$ to an element of $\Reals^{\I \times \J}$ by padding it with zeros.

\begin{example}
	\label{examp-average}
	(See \cref{fig-averageExample}.)
	For $h, \delta \in (0,1)$, let $(\tau_i^{(k)})_{i\in \I^{(k)}}$ be uniformly Lipschitz convex sets forming a regular nested partition of $\Omega$ in the following sense.
	For $k \in \{1,\ldots, q\}$, $\Omega=\bigcup_{i\in \I^{(k)}}\tau_i^{(k)}$ is a disjoint union except for the boundaries.
	$\I^{(k)}$ is a nested set of indices, i.e.\ $\I^{(k)}\subset \I^{(k+1)}$ for $k\in \{1,\ldots,q-1\}$.
	For $k \in \{2,\ldots, q\}$ and $i\in \I^{(k-1)}$, there exists a subset $ c_i \subset \I^{(k)}$ such that $i\in c_i$ and $\tau_i^{(k-1)} = \bigcup_{j \in c_{i}} \tau_j^{(k)}$.
	Assume that each $\tau_i^{(k)}$ contains a ball $B_{\delta h^{k}}(x_{i}^{(k)})$ of center $x_i^{(k)}$ and radius $\delta h^k$, and is contained in the ball $B_{h^{k}}(x_{i}^{(k)})$.
	For $k \in \{2,\ldots, q\}$ and $i \in \I^{(k-1)}$, let the submatrices $\mathfrak{w}^{(k),i} \in \Reals^{( c_i \setminus \{i\} ) \times c_{i}}$ satisfy $\sum_{j \in c_i} \mathfrak{w}^{(k),i}_{m,j} \mathfrak{w}^{(k),i}_{n,j} |\tau_j^{(k)}| = \delta_{mn} $ and $\sum_{j \in c_i} \mathfrak{w}^{(k),i}_{l,j} |\tau_j^{(k)}| = 0$ for each $l \in c_i\setminus\{i\}$,	where $|\tau_i^{(k)}|$ denotes the volume of $\tau_i^{(k)}$.
	Let $\J^{(1)} \defeq \I^{(1)}$ and $\J^{(k)} \defeq \I^{(k)} \setminus \I^{(k-1)}$ for $k\in \{2,\ldots,q\}$.
	Let $W^{(1)}$ be the $\J^{(1)}\times \I^{(1)}$ matrix defined by $W^{(1)}_{ij} \defeq \delta_{ij}$.
	Let $W^{(k)}$ be the $\J^{(k)}\times \I^{(k)}$ matrix defined by $W^{(k)} \defeq \sum_{i \in \I^{(k-1)}} \mathfrak{w}^{(k),i}$ for $k>2$, we set
	\begin{equation}
		\phi_{i} \defeq h^{-kd/2} \sum_{j \in \I^{(k)}} W^{(k)}_{i,j} \one_{\tau_j^{(k)}} \quad \text{for each $i \in \J^{(k)}$}
	\end{equation}
	and define $\dualprod{ \phi_i }{ u } \defeq \int_{\Omega} \phi_i u \dx$.
	In order to keep track of the distance between the different $\phi_i$ of
	\cref{examp-average}, we choose an arbitrary set of points
	$\{ x_i \}_{i \in \I} \subset \Omega$ with the property that
	$x_i \in \supp ( \phi_i )$ for each $i \in \I$.
\end{example}

\begin{figure}
	\centering
	\includegraphics[scale = 0.45]{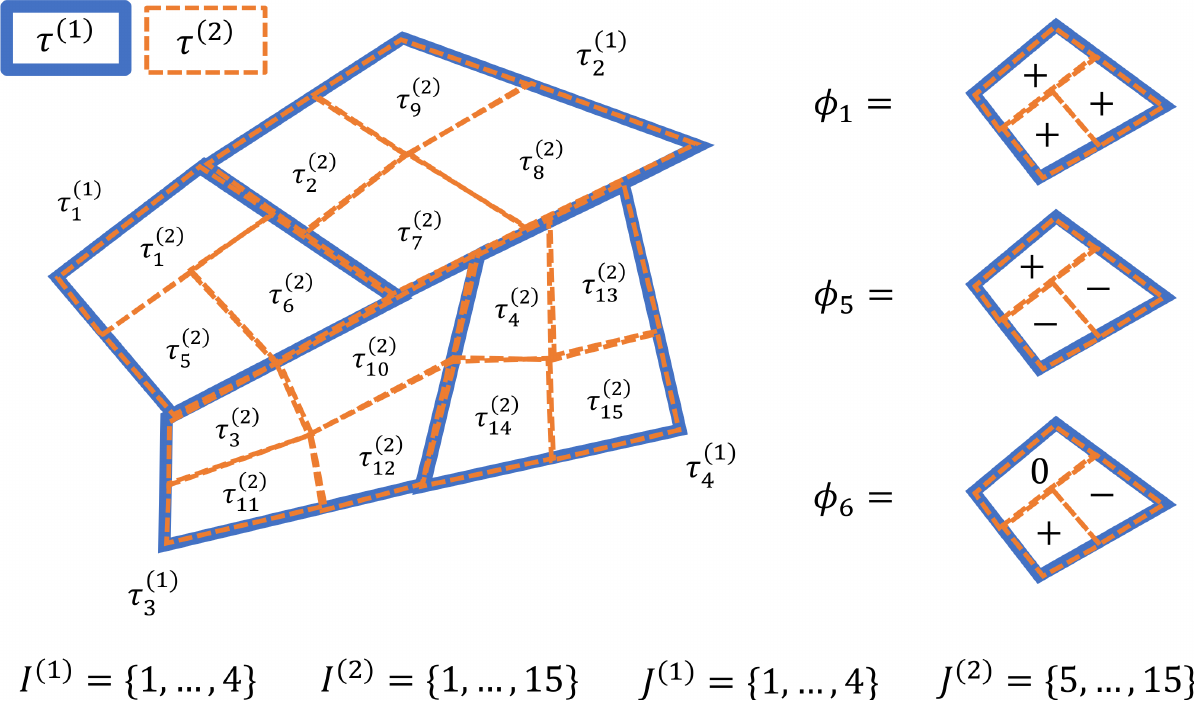}
	\caption{We illustrate the construction described in \cref{examp-average} in the case $q = 2$.
	On the left we see the nested partition of the domain, and on the right we see (the signs of) a possible choice for $\phi_1$, $\phi_5$, and $\phi_6$.}
	\label{fig-averageExample}
\end{figure}

\subsection{Exponential decay of Cholesky factors}

Our bound on the {\ICH} approximation error will be based on the following exponential decay estimate on the entries of the Cholesky factor $L$ of $\KM$:
\begin{equation}
	\label{eq-exp-decay-estimate-L}
	\absval{ L_{ij} } \leq \poly(N) \exp (-\gamma d( i, j ) ),
\end{equation}
for a constant $\gamma > 0$ and a suitable distance measure $d(\quark, \quark) \colon \I \times \I \to \Reals$.

\subsubsection{Algebraic Identities and roadmap}

The following block-Cholesky decomposition of $\Theta$ will be used to obtain the estimate \eqref{eq-exp-decay-estimate-L}.

\begin{lemma}
	\label{lem-BlockCholesky}
	We have $\KM = \bar{L} D \bar{L}^{T}$, with $\bar{L}$ and $D$ defined by
	\begin{footnotesize}
		\begin{equation}
			\label{eqjhgjhgyug}
			D \defeq
			\begin{pmatrix}
				\IKMC^{(1),-1} & 0 & \dots  & 0\\
				0 & \IKMC^{(2),-1} & \ddots & \vdots\\
				\vdots & \ddots & \ddots &  \vdots \\
				%\vdots & 0 & \ddots & \IKMC^{(q-1),-1} & 0 \\
				0 & 0 & \dots &  \IKMC^{(q),-1}
			\end{pmatrix},
			\bar{L} \defeq
			\begin{pmatrix}
				\Id &  \dots & \dots & 0\\
				\IKMC^{( 2 ),-1} \IKM^{( 2 )}_{2,1} &  \ddots & 0 &  \vdots\\
				\vdots &  \ddots & \ddots& \vdots \\
				% \vdots & \vdots & \ddots & \Id & 0 \\
				\IKMC^{( q ),-1} \IKM^{(q)}_{q,1} &  \dots &\IKMC^{( q ),-1} \IKM^{( q )}_{q,q-1}&\Id
			\end{pmatrix}^{-1} .
		\end{equation}
	\end{footnotesize}
	In particular, if $\tilde{L}$ is the lower-triangular Cholesky factor of $D$, then the lower-triangular Cholesky factor $L$ of $\KM$ is given by $L = \bar{L}\tilde{L}$.
\end{lemma}

\begin{proof}
	To obtain \cref{lem-BlockCholesky} we successively apply \cref{lem-blockChol2Scale} to $\KM$ (see \cref{apsec-abstractDecay} for details).
	\cref{lem-blockChol2Scale} summarizes classical identities satisfied by Schur complements.
\end{proof}

\begin{lemma}[{\cite[Chapter 1.1]{zhang2005schur}}]
	\label{lem-blockChol2Scale}
	Let $\KM =
	\left(
	\begin{smallmatrix}
		\KM_{1,1} & \KM_{1,2} \\
		\KM_{2,1} & \KM_{2,2}
	\end{smallmatrix}
	\right)$ be symmetric positive definite and $\IKM =
	\left(
	\begin{smallmatrix}
		\IKM_{1,1} & \IKM_{1,2} \\
		\IKM_{2,1} & \IKM_{2,2}
	\end{smallmatrix}
	\right)$ its inverse.
	Then
	\begin{align}
		\label{eqn-blockCholKM}
		\KM &=
		\begin{pmatrix}
			\Id     & 0\\
			L_{2,1} & \Id
		\end{pmatrix}
		\begin{pmatrix}
			D_{1,1} & 0 \\
			0       & D_{2,2}
		\end{pmatrix}
		\begin{pmatrix}
			\Id & L_{2,1}^\top\\
			0  & \Id
		\end{pmatrix} \\
		\label{eqn-blockCholIKM}
		\IKM & =
		\begin{pmatrix}
			\Id & -L_{2,1}^\top\\
			0  & \Id
		\end{pmatrix}
		\begin{pmatrix}
			D_{1,1}^{-1} & 0 \\
			0       & D_{2,2}^{-1}
		\end{pmatrix}
		\begin{pmatrix}
			\Id     & 0\\
			-L_{2,1} & \Id
		\end{pmatrix}
	\end{align}
	where
	\begin{align}
		L_{2,1} &= \KM_{2,1} \KM_{1,1}^{-1} = - \IKM_{2,2}^{-1} \IKM_{2,1} \\
		D_{1,1} &= \KM_{1,1} = \left( \IKM_{1,1} -
		\IKM_{1,2} \IKM_{2,2}^{-1} \IKM_{2,1} \right)^{-1} \\
		D_{2,2} &= \KM_{2,2} - \KM_{2,1} \KM_{1,1}^{-1} \KM_{1,2} = \IKM_{2,2}^{-1} .
	\end{align}
\end{lemma}

Based on \cref{lem-BlockCholesky}, \eqref{eq-exp-decay-estimate-L} can be established by ensuring that:
\begin{enumerate}[label=(\arabic*)]
	\item the matrices $\IKM^{(k)}$ (and hence also $\IKMC^{(k)}$) decay exponentially according to $d(\quark, \quark)$;
	\item the matrices $\IKMC^{(k)}$ have uniformly bounded condition numbers;
	\item the products of exponentially decaying matrices decay exponentially;
	\item the inverses of well-conditioned exponentially decaying matrices decay exponentially;
	\item the Cholesky factors of the inverses of well-conditioned exponentially decaying matrices decay exponentially;
	and
	\item if a $q \times q$ block lower-triangular matrix $\bar{L}$ with unit block-diagonal decays exponentially, then so does its inverse.
\end{enumerate}
We will carry out this program in the setting of \cref{examp-subsamp,examp-average} and prove that \eqref{eq-exp-decay-estimate-L} holds with
\begin{equation}
	\label{eq-exp-decay-estimate-L2}
	d( i, j ) \defeq h^{-\min(k,l) } \dist( x_i, x_j ), \quad \text{for each $i \in \J^{(k)}$, $j \in \J^{(l)}$.}
\end{equation}
To prove (1), the matrices $\KM^{(k)}$, $\IKM^{(k)}$ (interpreted as coarse-grained versions of $\K$ and $\IK$), and $\IKMC^{(1)}$ will be identified as stiffness matrices of the $\IK$-adapted wavelets described in \cref{sssec-choleskyGamblets}.
This identification is established on the general identities $\Theta_{i,j}^{(k)}=[\phi_i, \K \phi_j]$ for $i,j\in \I^{(k)}$, $\IKM^{(k)}=(\KM^{(k)})^{-1}$, $\IKM^{(k)}_{i,j}=[\IK \psi_i^{(k)},\psi_j^{(k)}]$ and $\IKMC^{(k)}_{i,j}=[\IK \chi_i^{(k)},\chi_j^{(k)}]$ where the $\psi_i^{(k)}$ and $\chi_i^{(k)}$ are defined as in \eqref{eq-pre-wavelet-representation} and \eqref{eq-Wk-basis}.

\subsubsection{Exponential decay of \texorpdfstring{$\IKM^{(k)}$}.}
\label{sssec-decayHomog}

Our proof of the exponential decay of $L$ will be based on that of $\IKM^{(k)}$ as expressed in the following condition:

\begin{condition}
	\label{cond-spatloc}
	Let $\gamma, C_\gamma \in \Reals_+$ be constants such that for $1 \leq k \leq q$ and $i,j\in \I^{(k)}$,
	\begin{equation}
		\bigabsval{ \IKM^{(k)}_{ij} } \leq C_\gamma \sqrt{\IKM^{(k)}_{ii} \IKM^{(k)}_{jj}}\exp (-\gamma d( i, j ) ) .
	\end{equation}
\end{condition}

The matrices $\IKM^{(k)}$ are coarse-grained versions of the local operator $\IK$ and as such inherit some of its locality in the form of exponential decay.
Such exponential localization results were first obtained by \cite{malqvist2014localization} for the coarse-grained operators obtained from local orthogonal decomposition (LOD) applied to second-order elliptic PDEs with rough coefficients.
\cite{owhadi2015multigrid} gives similar results for measurement functions chosen as in \cref{examp-average}.
\cite{hou2017sparse} extend the results on exponential decay to higher-order operators satisfying a strong ellipticity condition.
These results were obtained using similar \emph{mass chasing} techniques that are difficult to extend to general higher-order operators.
\cite{kornhuber2016analysis} present a simpler proof of the exponential decay of the LOD basis functions of \cite{malqvist2014localization} based on the exponential convergence of subspace iteration methods.
\cite{owhadi2017universal} extend this technique (by presenting necessary and
sufficient conditions expressed as frame inequalities in dual spaces) to elliptic PDEs of arbitrary
(integer) order and new classes of (possibly non-conforming) measurements, including those of \cref{examp-subsamp} and \cref{examp-average}.
More recently, \cite{brown2018numerical} show localization results for the fractional partial differential operators by using the Caffarelli--Silvestre extension.
The results of \cite{owhadi2017universal} are sufficient to show that \cref{cond-spatloc} holds true in the setting of \cref{examp-subsamp} and \cref{examp-average}.

\begin{theorem}[\cite{owhadi2017universal}]
	\label{thm-spatloc}
	In \cref{examp-subsamp}, the matrices
	$\IKM^{(k)}$ satisfy
	\begin{equation}
		\left| \IKM^{(k)}_{ij} \right| \leq C_\gamma \sqrt{\IKM^{(k)}_{ii}\IKM^{(k)}_{jj}} \exp\left(- \frac{\gamma}{h^k} \dist ( x_i, x_j ) \right) \leq C_\gamma \sqrt{\IKM^{(k)}_{ii}\IKM^{(k)}_{jj}} \exp (- \gamma d(i,j) )
	\end{equation}
	and in \cref{examp-average} they satisfy
	\begin{equation}
		\bigabsval{ \IKM^{(k)}_{ij} } \leq C_\gamma \exp\left(\frac{\gamma}{h}\right) \sqrt{\IKM^{(k)}_{ii}\IKM^{(k)}_{jj}} \exp (- \gamma d( i, j ) ) ,
	\end{equation}
	with the constants $C_\gamma$ and $\gamma$ depending only on $\|\IK\|$, $\norm{ \IK^{-1} }$,	$s$, $d$, $\Omega$, and $\delta$.
	In particular, they satisfy \cref{cond-spatloc} with the constants described above.
\end{theorem}

\begin{proof}
	Our \cref{examp-subsamp} is equivalent to Example 2.29 of \cite{owhadi2017universal}.
	In \cite[Theorem~2.25 and Theorem~2.26]{owhadi2017universal} it is shown that in the gamblets	$\{ \psi_i^{(k)} \}_{i \in \I^{(k)}}$ computed in this setting decay exponentially on the length-scale $h^k$, with respect to the energy norm.
	By \cite[Theorem 3.8]{owhadi2017universal} we have $\IKM^{(k)}_{ij} = \dualprod{ \psi_i^{(k)} }{ \IK \psi_j^{(k)} }$ and, therefore, the exponential decay of gamblets implies the exponential decay of the $\IKM^{(k)}$.
	
	We further note that \cref{examp-average} is equivalent to Example 2.27 in \cite{owhadi2017universal}.
	Therefore, by the same theorems, as above, the results of \cite{owhadi2017universal} imply exponential decay of the $\IKM^{(k)}$ in this setting\footnote{We point out that the block $\IKM^{(k)}_{m,l}$ in our notation is	$W^{(m)} \pi^{(m,k)} \IKM^{(k)} \pi^{(k,l)} W^{(l),\top}$ in the notation of \cite{owhadi2017universal}.}.

See also \cite[Theorem~15.45]{OwhScobook2018} for a detailed proof and \cite[Theorem~15.43]{OwhScobook2018} for required sufficient lower bounds on $\IKM^{(k)}_{ii}$.
\end{proof}

\subsubsection{Bounded condition numbers}
\label{sssec-boundCond}

In this section, we will bound the condition numbers of $\IKMC^{(k)}$ based on the following condition, which we will show to be satisfied for \cref{examp-subsamp,examp-average}.

\begin{condition}
	\label{cond-specloc}
	Let $H \in (0,1), C_{\Phi} \geq 1$ be constants such that for $1 \leq k < l \leq q$,
	\begin{align}
		\label{eqn-specloclow}
		\lambda_{\min} \bigl( \KM^{(k)} \bigr) & \geq \frac{1}{C_{\Phi}} H^{2k} , \\
		\label{eqn-speclocup}
		\lambda_{\max} \bigl( \KM^{(q)}_{l, l} -
		\KM^{(q)}_{l,1:k} \KM^{(q),-1}_{1:k,1:k} \KM^{(q)}_{1:k,l} \bigr) & \leq C_{\Phi} H^{2k} \,.
	\end{align}
\end{condition}

\begin{theorem}
	\label{thm-condphi}
	\Cref{cond-specloc} implies that, for all $1 \leq k \leq q$,
	\begin{equation}
		\label{eqjheddhhe}
		C_{\Phi}^{-1}H^{-2\left(k-1\right)} \Id \prec \IKMC^{(k)} \prec C_{\Phi} H^{-2k} \Id,
		\end{equation}
	and, for $\kappa \defeq H^{-2} C_{\Phi}^2$,
	\begin{equation}
		\cond \bigl( \IKMC^{(k)} \bigr) \leq \kappa \,.
	\end{equation}
\end{theorem}

\begin{proof}
	The lower bound in \eqref{eqjheddhhe} follows from \eqref{eqn-speclocup} and
	\begin{equation}
		\IKMC^{(k)} = \bigl( \KM_{k, k}^{(q)} - \KM^{(q)}_{k,1:(k-1)} \KM^{(q),-1}_{1:k,1:k} \KM^{(q)}_{1:(k-1),k} \bigr)^{-1}\,.
	\end{equation}
	The upper bound in \eqref{eqjheddhhe} follows from \eqref{eqn-specloclow} and $\IKMC^{(k)} = \bigl( \bigl( \KM^{(k)} \bigr)^{-1} \bigr)_{k,k}$.
\end{proof}

The following theorem shows that \eqref{eqn-speclocup} is a Poincar\'{e} inequality closely related to the accuracy of numerical homogenization basis functions \cite{malqvist2014localization, owhadi2014polyharmonic, hou2017sparse} and \eqref{eqn-specloclow} is an inverse Sobolev inequality related to the regularity of the discretization of $\IK$:

\begin{theorem}
	\label{thm-condNew}
	\Cref{cond-specloc} holds true if the constants $C_\Phi \geq 1$ and $H \in (0,1)$ satisfy
	\begin{enumerate}[label=(\arabic*)]
		\item $\frac{1}{C_\Phi}H^{2k} \leq \frac{\norm{ \phi }_{\ast}^2}{\absval{ \x }^2}$, for $\x \in \Reals^{\I^{(k)}}$
		and $\phi = \sum_{i\in \I^{(k)}} \x_i \phi_i$; and
		\item $\min_{\varphi \in \spn (\phi_i )_{i \in \I^{(k-1)}} }\frac{\| \phi - \varphi \|_{\ast}^2}{|\x|^2} \leq C_{\Phi} H^{2(k-1)}$, for $\x \in \Reals^{\J^{(l)}}$,
		$k<l \leq q$, and $\phi = \sum_{i\in \J^{(l)}} \x_i \phi_i$.
	\end{enumerate}
\end{theorem}

\begin{proof}
	Inequality \eqref{eqn-specloclow} is a direct consequence of the first assumption of the theorem, whereas \eqref{eqn-speclocup} follows from the variational property 	\cite[Theorem 5.1]{zhang2005schur} of the Schur complement:
	\begin{align}
		\label{eqkwhdhkjhd}
		\x^{\top} \left(\KM_{l, l} - \KM^{(q)}_{l,1:k} \KM^{(q),-1}_{1:k,1:k} \KM^{(q)}_{1:k,l} \right)\x
		& = \inf_{\y \in \Reals^{\I^{(k)}}} \left(\x - \y \right)^{\top} \KM^{(q)} ( \x - \y )\\
		& = \min_{\varphi \in \spn \{ \phi_i \mid i \in \I^{(k)} \} }
		\norm{ \phi - \varphi }_{\ast}^2 \leq C_{\Phi}H^{2k} |\x|^2.
		% \qedhere
	\end{align}
\end{proof}

We will now show that \cref{examp-subsamp,examp-average} satisfy the conditions of \cref{thm-condNew}.
For simplicity, for $\tilde{\Omega} \subset \Omega$ and $\phi \in H^{-s}( \Omega )$ we still write $\phi$ for the unique element $\tilde{\phi} \in H^{-s} ( \tilde{\Omega} ) $ such that $\dualprod{ \tilde{\phi} }{ u } = \dualprod{ \phi }{ u }$ for $u\in H_0^{s} ( \tilde{\Omega} )$.
The following Fenchel conjugate identity \cite[Ex.~3.27, p.~93]{boyd2004convex} will be useful throughout this section.
\begin{equation}
	\label{eqn-duality}
	\norm{ \phi }_{H^{-s}( \Omega )}^2 = \sup_{v \in H_0^{s}( \Omega )} 2 \dualprod{ \phi }{ v } - \norm{ v }_{v \in H_0^{s}( \Omega )}^2.
\end{equation}
The first condition can be verified similarly as is done in \cite{owhadi2017universal}.

\begin{lemma}
	\label{lem-lowerBound}
	Let $\KM$ be given as in \cref{examp-subsamp,examp-average}.
	Then there exists a constant $C$ depending only on $\delta$, $s$, and $d$, such that
	\begin{equation}
		\frac{1}{C_\Phi} h^{2sk} \leq \frac{\norm{ \phi }_{\ast}^2}{\absval{ \x }^2},
	\end{equation}
	for $C_{\Phi} = \|\IK\| C$, $\x \in \Reals^{\I^{(k)}}$, and $\phi = \sum_i \x_i \phi_i$.
\end{lemma}

\begin{proof}
	The proof can be found in \cref{apsec-abstractDecay}.
\end{proof}

In order to verify the second condition in \cref{thm-condNew}, we will construct a $\varphi$ such that $\phi - \varphi$ integrates to zero against polynomials of order at most $s-1$ on domains of size $h^k$.
Then an application of the Bramble--Hilbert lemma \cite{dekel2004bramble} will yield the desired factor $h^{ks}$.
To avoid scaling issues we define, for $1 \leq k \leq q$ and $i \in \I^{(k)}$,
\begin{equation}
	\phi^{(k)}_i \defeq
	\begin{cases}
		\boldsymbol{\delta}_{x_i} , & \text{in \cref{examp-subsamp},} \\
		\one_{\tau^{(k)}_i}/|\tau_i^{(k)}| , & \text{in \cref{examp-average},}
	\end{cases}
\end{equation}
noting that $\spn \{ \phi^{(k)}_i \mid i \in \I^{(k)} \} = \spn \{ \phi_i \mid i \in \I^{(k)} \}$.
To obtain estimates independent of the regularity of $\Omega$, for the simplicity of the proof and without loss of generality, we will partially work in the extended space $\Reals^d$ (rather than on $\Omega$).
We write $v$ for the zero extension of $v\in H_0^s( \Omega )$ to $H^s (\Reals^d )$ and $\phi_i^{(k)}$ for the extension of $\phi_i^{(k)}\in H^{-s}(\Omega)$ to an element of the dual space of $H_{\loc}^s (\Reals^d)$.
We introduce new measurement functions in the complement of $\Omega$ as follows.
For $1 \leq k \leq q$ we consider countably infinite index sets $\tilde{\I}^{(k)} \supset\I^{(k)}$.
We choose points $( x_i )_{i \in \tilde{\I}^{(q)} \setminus \I^{(q)}}$ satisfying
\begin{align}
	\sup_{x \in \Reals^d \setminus \Omega} \min_{i \in \tilde{\I}^{(k)}}
	\dist\left(x_i, x\right) & \leq \delta^{-1} h^k , &
	\min_{i \neq j \in \tilde{\I}^{(k)} \setminus \I^{(k)}}
	\dist( x_i, x_j \cup \partial \Omega ) & \geq \delta h^k .
\end{align}
We then define, for $1 \leq k \leq q$ and $i \in \tilde{\I}^{(k)}$,
$\phi_i^{(k)} \defeq \delta_{x_i}$ for \cref{examp-subsamp},
and $\phi_i^{(k)} \defeq \frac{ \one_{B_{\delta h^k}( x_i )} }{|B_{\delta h^k}( x_i )|}$ for \cref{examp-average}.
Let $\mathcal{P}^{s-1}$ denote the linear space of polynomials of degree at most $s-1$ (on $\Reals^d$).

\begin{lemma}
	\label{lem-condForW}
	Let $\KM$ be as in \cref{examp-subsamp} or \cref{examp-average}.
	Given $\rho \in (2,\infty)$ and
	$1 \leq k < l \leq q$ let $w \in \Reals^{\J^{(l)} \times \tilde{\I}^{(k)}}$
	be such that
	\begin{equation}
		\label{eqjkgfkhkjhgh}
		\int_{B_{\rho h^k}( x_i )} \left( \phi_i - \sum_{j\in \tilde{ \I}^{(k)}} w_{ij} \phi_j^{(k)} \right)( x ) p( x ) \dx = 0, \quad \text{for all $p \in \mathcal{P}^{s-1}$ and $i \in \J^{(l)}$}
	\end{equation}
	and $w_{ij} \neq 0 \Rightarrow \supp\left( \phi_j^{(k)}\right) \subset B_{\rho h^k}(x_i)$.
	Then, for $\x\in \Reals^{\J^{(l)}}$, $\phi \defeq \sum_{i \in \J^{(l)}} \x_i \phi_i$ and $\varphi \defeq \sum_{ i \in \J^{(l)}, j \in \I^{(k)}} \x_i w_{ij} \phi_{j}^{(k)}$ satisfy
	\begin{equation}
		\norm{ \phi - \varphi }_{\ast}^2 \leq \norm{ \IK^{-1} } C(d,s) \frac{\rho^{d+2s}}{\delta^d} \left( 1 + h^{- l d} \omega_{l,k}^2 \right) h^{2sk} \absval{ \x }^2 ,
	\end{equation}
	with $\omega_{l,k} \defeq \sup_{i \in \J^{(l)}} \sum_{j \in \tilde{\I}^{(k)}} \absval{ w_{ij} }$ and $\norm{ \phi }_{\ast} \defeq \sup_{u\in H^s_0(\Omega)}\dualprod{ \phi }{ u }/[\IK u,u]^\frac{1}{2}$ as in \eqref{eq-dual-norm}.
\end{lemma}

We proceed by proving \cref{lem-condForW} in the setting of \cref{examp-subsamp}.
The proof in the setting of \cref{examp-average} can be found in \cref{apsec-abstractDecay}.
For $u\in H^s(\Omega)$ write $\rD^0 u \defeq u$ and for $1 \leq k \leq s$ write $\rD^k u$ for the vector of partial derivatives of $u$ of order $k$, i.e.\ $\rD^k u \defeq \Bigl(\frac{\partial^k u}{\partial_{i_1}\cdots \partial_{i_k} }\Bigr)_{i_1,\ldots,i_k=1,\ldots, d}$.
The proof of \cref{lem-condForW} will use the following version of the Bramble--Hilbert lemma:

\begin{lemma}[\cite{dekel2004bramble}]
	\label{lem-brambleHilbert}
	Let $\Omega \subset \Reals^d$ be convex and let $\phi$ be a sublinear functional on $H^s( \Omega )$ for $s \in \mathbb{N}$ such that
	\begin{enumerate}[label=(\arabic*)]
		\item there exists a constant $\tilde{C}$ such that, for all
		$u \in H^s( \Omega )$,
		\begin{equation}
			\absval{ \phi(u) } \leq \tilde{C} \sum_{k =0}^s \diam( \Omega )^k \norm{ \rD^k u }_{L^2( \Omega )} ;
		\end{equation}
		\item and $\phi (p) = 0$ for all $p \in \mathcal{P}^{s-1}$.
	\end{enumerate}
	Then, for all $u \in H^s( \Omega )$,
	\begin{equation}
		\absval{ \phi (u) } \leq \tilde{C} C ( d, s ) \diam( \Omega )^s \norm{ \rD^s u }_{L^2( \Omega )} .
	\end{equation}
\end{lemma}

The following lemma is obtained from \cref{lem-brambleHilbert}:

\begin{lemma}
	\label{lem-brambleHilbertCaps}
	For $1 \leq k < l \leq q$ and $i \in \J^{(l)}$, let $\phi_i, w_{ij}$ be as in \cref{lem-condForW} and \cref{examp-average} and define $\varphi_i \defeq \sum_{j \in \I^{(k)}} w_{ij} \phi_j^{(k)}$.
	Then there exists a constant $C( d, s )$ such that, for all $v \in H_0^s( \Omega )$,
	{\scriptsize
	\begin{equation}
		\left| \int_{B_{\rho h^k}( x_i )} ( \phi_i - \varphi_i )( x ) v( x ) \dx \right| \leq C ( d, s ) \rho^{s-d/2} h^{(s-d/2)k} \left( h^{ld/2} + \sum_{j \in \tilde{\I}^{(k)}} | w_{ij} | \right) \| \rD^s v \|_{L^2 \left( B_{\rho h^k}( x_i )\right)} .
	\end{equation}}
\end{lemma}

\begin{proof}
	We apply \cref{lem-brambleHilbert} to the linear functional $u \mapsto \int_{B_{\rho h^k}} (\phi_i - \varphi_i)u$.
	Since the second requirement of \cref{lem-brambleHilbert} is fulfilled by definition, it remains to bound $\tilde{C}$.
	We only execute the proof for \cref{examp-subsamp};
	the proof for \cref{examp-average} is analogous.
	We first note that while the sum in the definition of $\varphi_i$ only ranges over 	$j \in \I^{(k)}$, we can increase it to run over all of $j \in \tI^{(k)}$,	since for $j \in \tilde{\I}^{(k)} \setminus \I^{(k)}$, the support of $\phi^{(k)}_j$ is disjoint from that of $v\in H^s_0(\Omega)$.
	Let $u\in H^s(\Omega)$.
	Writing $C( d, s )$ for the continuity constant of the embedding of $H^s (B_1(0) )$ into $C_b (B_1(0))$, the inequalities
	{\scriptsize
	\[
		\max_{B_{\rho h^k}( x_i )} | u( \quark ) | = \max_{x \in B_{1}(0)} \left|u\left( \rho h^k \left(x - x_i\right) \right) \right| \leq C ( d, s ) \sum_{m=0}^s (\rho h^k )^{m} \left\| [ \rD^m u ] \bigl( \rho h^k ( \quark - x_i ) \bigr) \right\|_{L^2 (B_1(0) )}
	\]}
	and
	\[
		\left\| [ \rD^m u ] \bigl( \rho h^k ( \quark - x_i ) \bigr)	\right\|_{L^2\left(B_1(0)\right)} = (\rho h^k )^{-d/2} \norm{ \rD^m u }_{L^2 (B_{\rho h^k}( x_i ) )}
	\]
	imply that
	{\small
	\begin{align}
		\left| \phi_i\left( u \right) - \varphi_i \left( u \right) \right|
		&\leq \left( h^{ld/2} + \sum_{j \in \tilde{\I}^{(k)}} \absval{ w_{ij} } \right)	\max_{x \in B_{\rho h^k}( x_i )} \left| u( x ) \right|\\
		&\leq C ( d, s ) \rho^{-d/2} h^{-kd/2} \left( h^{ld/2} + \sum_{j \in \tilde{\I}^{(k)}} \absval{ w_{ij} } \right) \sum_{m=0}^s ( \rho h^k )^{m} \norm{ \rD^m u }_{L^2 (B_{\rho h^k}( x_i ) )}
	\end{align}}
	Therefore the first condition of \cref{lem-brambleHilbert}
	holds with
	\begin{equation}
		\tilde{C} = C ( d, s ) \rho^{-d/2} h^{-kd/2} \left( h^{ld/2} + \sum_{j \in \tilde{\I}^{(k)}} \absval{ w_{ij} } \right)\,,
	\end{equation}
	and we conclude the proof by writing $C(d,s)$ for any constant depending only on $d$ and $s$.
\end{proof}

We can now conclude the proof of \cref{lem-condForW}.

\begin{proof}[Proof of \cref{lem-condForW}]
	Write $\varphi \defeq \sum_{i\in \J^{(l)}} \x_i \varphi_i$ and $\varphi_i \defeq \sum_{j\in \I^{(k)}} w_{ij} \phi_j^{(k)}$.
	Equation~\eqref{eqn-duality} implies that
	\begin{equation}
		\norm{ \phi - \varphi }_{H^{-s}(\Omega)}^2
		= \sup_{v \in H_0^s( \Omega )} \left( \sum_{ i \in \J^{(l)}} 2 \x_i \int_{B_{\rho h^k}( x_i )} \left(\phi_i - \varphi_i\right)( x ) v(x) \dx \right) - \|v\|_{H_0^s( \Omega )}^2 \,.
	\end{equation}
	The packing inequality $\sum_{i \in \J^{(l)}} \norm{ \rD^s v }_{L^2\left( B_{\rho h^k}( x_i )\right)}^2 \leq C(d) \left( h^{k-l}\rho / \delta \right)^d \norm{ v }_{H_0^s( \Omega )}^2$ together with \cref{lem-brambleHilbertCaps} yields
	{\scriptsize
	\begin{align}
		\norm{ \phi - \varphi }_{H^{-s}(\Omega)}^2
		\leq \sup_{v \in H_0^s( \Omega )}
		\sum_{i \in \J^{(l)}} \Bigg[ & 2|\alpha_i| C( d, s )\rho^{s - \frac{d}{2}} h^{(s - \frac{d}{2}) k} \left( h^{\frac{l d}{2}} + \sum_{j \in \I^{(k)}} \absval{ w_{ij} } \right)	\norm{ \rD^s v }_{L^2 ( B_{\rho h^k}( x_i ) )} \\
		& - (C(d))^{-1} \left( h^{k-l}\rho / \delta \right)^{-d} \norm{ \rD^s v }_{L^2\left( B_{\rho h^k}( x_i )\right)}^2 \Bigg]\,.
	\end{align}}
	Applying the inequality $2ax-bx^2 \leq a^2/b$ to each summand yields
	{\scriptsize
	\begin{align}
		\norm{ \phi - \varphi }_{H^{-s}(\Omega)}^2
		&\leq C(d) \left(h^{k-l}\rho / \delta\right)^d \sum_{i \in \J^{(l)}} \left(\x_j C( d, s )\rho^{s-\frac{d}{2}} h^{(s-\frac{d}{2})k} \left( h^{\frac{ld}{2}} + \sum_{j \in \J^{(k)}} \absval{ w_{ij} } \right) \right)^2\\
		&\leq C(d,s) \frac{\rho^{2s}}{\delta^d}
		\left( 1 + h^{- l d} \omega_{l,k}^2 \right)	h^{2sk}	\absval{ \x }^2\,.
	\end{align}}
	Since, for all $f\in H^{-s}(\Omega)$,
	\begin{equation}
		\norm{ f }_{\ast}^2=[f,\IK^{-1} f] \leq \norm{ f }_{H^{-s}(\Omega)} \|\IK^{-1} f\|_{H^s_0(\Omega)}\leq \norm{ \IK^{-1} }\norm{ f }_{H^{-s}(\Omega)}^2 ,
	\end{equation}
	we have $\|\phi - \varphi\|_{\ast}\leq \sqrt{\norm{ \IK^{-1} }}\|\phi - \varphi\|_{H^{-s}(\Omega)}$, and this completes the proof.
\end{proof}

The following geometric lemma shows that the assumption \eqref{eqjkgfkhkjhgh} of \cref{lem-condForW} can be satisfied with a uniform bound on the value of $\rho$ and the norm of weights $w_{i,j}$.

\begin{lemma}
	\label{lem-geomLemma}
	There exists constants $\rho(d,s)$ and $C(d,s,\delta)$ such that for all $1 \leq k < l \leq q$ there exists weights $w \in \Reals^{\J^{(l)} \times \tilde{\I}^{(k)}}$ satisfying \eqref{eqjkgfkhkjhgh} and (with $\omega_{l,k}$ defined as in \cref{lem-condForW})
	\begin{equation}
		\omega_{l,k}^2 \leq h^{ld} C(d,s,\delta)\,.
	\end{equation}
\end{lemma}

\begin{proof}
	For \cref{examp-subsamp}, \eqref{eqjkgfkhkjhgh} is equivalent to
	\begin{equation}
		\label{eqn-polyEqual}
		h^{ld/2} p(x_i) = \sum_{j \in \tilde{\I}^{(k)}_\rho} w_{ij} p(x_j), \forall p \in \mathcal{P}^{s-1},
	\end{equation}
	where $\tilde{\I}^{(k)}_\rho \defeq \{j\in \tilde{\I}^{(k)}\mid x_j \in B(x_i,\rho h^k)\}$.

	Fix $i\in \J^{(l)}$, let $\lambda>0$ and write $x_j^\lambda \defeq \frac{x_j-x_i}{\lambda}$.
	Write $\mathbf{0} \defeq (0,\ldots,0)\in \Reals^d$.
	Since the function $p(\quark)\mapsto p(\frac{\quark-x_i}{\lambda})$ is surjective on
	$\mathcal{P}^{s-1}$, \eqref{eqn-polyEqual} is satisfied if
	\begin{equation}
		\label{eqn-polyEqual2b}
		h^{ld/2} p(\mathbf{0}) = \sum_{j \in \tilde{\I}^{(k)}_\rho} w_{ij} p(x_j^\lambda), \forall p \in \mathcal{P}^{s-1}.
	\end{equation}
	For a multiindex $n=(n_1,\ldots,n_d) \in \Naturals^d$ and a point $z=(z_1,\ldots,z_d) \in \Reals^d$, write 	$z^n \defeq \prod_{m = 1}^d z_m^{n_m}$.
	Use the convention $\mathbf{0}^n=0$ if $n\not=\mathbf{0}$ and $\mathbf{0}^\mathbf{0}=1$.
	To satisfy \eqref{eqn-polyEqual2b} it is sufficient to identify a subset $\sigma$ of $\tilde{\I}^{(k)}_\rho$ and $w_{i,\quark}\in \Reals^{ \tilde{\I}^{(k)}}$ such that $\# \sigma = s^d$, $w_{i,j}=0$ for $j\not \in \sigma$, and
	\begin{equation}
		\label{eqn-polyEqualSigma}
		h^{ld/2}\mathbf{0}^n = \sum_{j \in \sigma} w_{ij} (x_j^\lambda)^n, \forall n \in \{0,\ldots,s-1\}^d\,.
	\end{equation}
	Let $\mathbb{V}^\lambda \in \Reals^{\{0,1,\dots ,s-1\}^d \times \sigma }$ be the $s^d\times s^d$ matrix defined by
	\begin{equation}
		\mathbb{V}_{n,j}^\lambda \defeq \left(x_j^\lambda\right)^n\,.
	\end{equation}
	for a multiindex $n \in \Naturals^d$ and a point $x \in \Reals^d$
	$x^n \defeq \prod_{m = 1}^d x^{n_m}$.
	Let $\mathbf{w}\in \Reals^{\sigma }$ be defined by $\mathbf{w}_j \defeq w_{i,j}$ for $j\in \sigma$.
	Equation~\eqref{eqn-polyEqualSigma} is then equivalent to
	\begin{equation}
		\label{eqjedgjhd22ggdj}
		h^{ld/2} \mathbf{e} = \mathbb{V}^\lambda \mathbf{w},
	\end{equation}
	where $\mathbf{e} \in \Reals^{\{0,1,\dots ,s-1\}^d}$ is defined by $\mathbf{e}_n \defeq \mathbf{0}^n$ for $n\in \{0,1,\dots ,s-1\}^d$.
	We will now identify $\mathbf{w}$ by inverting \eqref{eqjedgjhd22ggdj}.
	To achieve this while keeping the norm of $\mathbf{w}$ under control we will seek to identify the subset $\sigma$ and $\lambda>0$ such that $\sigma_{\min}(\mathbb{V}^\lambda)$ (the minimal
	singular value of $\mathbb{V}^\lambda$) is bounded from below by a constant depending only on $s$ and $d$.

	For $\alpha\geq 0$ let $(\epsilon_j)_{j \in \{0 , 1 , \dots , s-1 \}^d}$ be elements of $\Reals^d$ satisfying $|\epsilon_j|\leq \alpha$ for all $j \in \{0 , 1 , \dots , s-1 \}^d$.
	Let $\mathbf{1} \defeq (1,\ldots,1)\in \Reals^d$ and, for $j \in \{0 , 1 , \dots , s-1 \}^d$, let $z_j \defeq \mathbf{1} + j+\epsilon_j$.
	Observe that for $\alpha=0$ the points $z_j$ are on a regular grid.
	Let $\bar{\mathbb{V}}^\alpha \in \Reals^{\{0,1,\dots ,s-1\}^d \times \{0,1,\dots ,s-1\}^d }$ be the $s^d\times s^d$ matrix defined by
	$ \bar{\mathbb{V}}^\alpha_{n,j} \defeq \left(z_j\right)^n$.
	Let $V$ be the $s\times s$ Vandermonde matrix defined by
	$V_{i,j}=i^j$.
	Writing $\sigma_{\min}(V)$ for the minimal singular value of $V$ we have, for $\alpha = 0$, by \cite[Theorem 4.2.12]{horn1994topics},
	\begin{equation}
		\sigma_{\min}\left( \bar{\mathbb{V}}^0 \right) = \left( \sigma_{\min}(V)\right)^d.
	\end{equation}
	Since univariate polynomial interpolation on $s$ points with polynomials of degree $s-1$ 	is uniquely solvable, we have $\sigma_{\min}\left( V \right) > 0$ and $\sigma_{\min}(\bar{\mathbb{V}}^0)> C(d,s)>0 $.
	Therefore, the continuity of the minimal singular value
	with respect the entries of $\bar{\mathbb{V}}^\alpha$ implies that there exists $\alpha^{\ast}, \sigma^{\ast}>0$ depending only on $s,d$ such that $\alpha\leq \alpha^{\ast}$ implies $\sigma_{\min}(\bar{\mathbb{V}}^\alpha)> \sigma^{\ast}$.
	Since (by construction) the $( x_i )_{i \in \tilde{\I}^{(k)}}$ form a covering of $\Reals^d$ of radius $h^k$, the $(x_i^\lambda)_{i \in \tilde{\I}^{(k)}}$ form a covering of $\Reals^d$ of radius $h^k/\lambda$ and for each 	$n \in \{0 , 1 , \dots , s-1 \}^d$ there exists an $x_{j_n}^\lambda$ that is at distance at most $h^k/\lambda$ from $n$.
	Let $\sigma \defeq \{j_n \mid n\in \{0 , 1 , \dots , s-1 \}^d\}\subset \tilde{\I}^{(k)}$ be the collection of corresponding labels.
	It follows from $|x_{j_n}^\lambda|\leq \sqrt{ d} s+ h^k/\lambda$ that $|x_{j_n}-x_i|\leq \lambda \sqrt{ d} s+ h^k$, and $\sigma \subset \tilde{\I}^{(k)}_{\rho}$ for $\rho > 1+ \lambda \sqrt{ d} s/h^k$.
	Selecting $\lambda =h^k/\alpha^{\ast}$ implies that $\sigma_{\min}(\mathbb{V}^\lambda)> \sigma^{\ast}$ and $\sigma \subset \tilde{\I}^{(k)}_{\rho}$ for $\rho > 1+ \sqrt{ d} s/\alpha^{\ast}$.
	Defining
	\begin{equation}
		w_{ij} \defeq
		\begin{cases}
		\left( (\mathbb{V}^\lambda)^{-1} h^{ld/2} \mathbf{e} \right)_n , & \text{if $j = j_n \in \sigma$,} \\
		0, & \text{otherwise,}
		\end{cases}
	\end{equation}
	the weights $w_{ij}$ satisfy $\omega_{kl} \leq C(s,d) h^{ld/2}$ and \eqref{eqjkgfkhkjhgh} with a $\rho$ depending only on $s$ and $d$.
	This concludes the proof for \cref{examp-subsamp}.
	The proof is similar for \cref{examp-average} with minor changes (the bound on $\omega$ also depends on $\delta$).
\end{proof}

The following lemma concerns the satisfaction of the second condition of \cref{thm-condNew}:

\begin{lemma}
	\label{lem-upperBound}
	In the setting of \cref{examp-subsamp,examp-average}, there exists some constant $ C(d,s,\delta)>0$ such that, for $2\leq k<l \leq q$, $\x \in \Reals^{\J^{(l)}}$ and $\phi = \sum_i \x_i \phi_i$,
	\begin{equation}
		\min_{\varphi \in \spn\left(\phi_i \right)_{i \in \I^{(k-1)}} } \frac{\| \phi - \varphi \|_{\ast}^2}{|\x|^2} \leq C(d,s,\delta) \norm{ \IK^{-1} } h^{2s(k-1)}\,.
	\end{equation}
\end{lemma}

\begin{proof}
	Apply \cref{lem-condForW} with the bounds on $\rho$ and $\omega$ obtained in \cref{lem-geomLemma}.
\end{proof}

The following theorem is a direct consequence of \cref{thm-condNew}, \cref{lem-lowerBound} and \cref{lem-upperBound}.

\begin{theorem}
	\label{thm-specloc}
	In the setting of \cref{examp-subsamp,examp-average} there exists a constant $C(d,s,\delta)$ such that \cref{cond-specloc} is fulfilled with $C_{\Phi} \defeq \max ( \|\IK\|, \norm{ \IK^{-1} } ) C(d,s,\delta)$ and	$H\defeq h^s$.
\end{theorem}

\subsubsection{Propagation of exponential decay}
\label{sssec-abstractDecay}

We will now derive the exponential decay of the Cholesky factors $L$ by combining the algebraic identities of \cref{lem-BlockCholesky} with the bounds on the condition numbers of the $\IKMC^{(k)}$ (implied by \cref{cond-specloc}) and the exponential decay of the $\IKM^{(k)}$ (specified in \cref{cond-spatloc}).
The core of our proof is based on a combination/extension of the results of
\cite{demko1984decay,jaffard1990proprietes,benzi2015decay,benzi2016localisation,
krishtal2015localization,benzi2000orderings} on decay algebras.
The pseudodistance $d(\quark,\quark)$ appearing in \eqref{eq-exp-decay-estimate-L} is not a pseudometric because it does not satisfy
the triangle inequality.
However, to prove \eqref{eq-exp-decay-estimate-L} we we will only need the following weaker version of the triangle inequality:

\begin{definition}
	\label{def-hierarchical-pseudometric}
	A function $d \colon \I \times \I \longrightarrow \Reals_+$ is called a \emph{hierarchical pseudometric} if
	\begin{enumerate}[label=(\arabic*)]
		\item $d(i,i) = 0, \text{ for all } i \in \I$;
		\item $d(i,j) = d(j,i), \text{ for all } i,j \in \I$;
		\item for all $1 \leq k \leq q$, $d(\quark,\quark)$ restricted to
		$\J^{(k)} \times \J^{(k)}$ is a pseudometric;
		\item for all $1 \leq k \leq l \leq m \leq q$ and $i \in \J^{(k)}, s \in \J^{(l)},j \in \J^{(m)}$, we have $d(i,j) \leq d(i,s) + d(s,j)$.
	\end{enumerate}
\end{definition}

Note that the $d(\quark, \quark)$ specified in \eqref{eq-exp-decay-estimate-L2} for \cref{examp-subsamp,examp-average} is a hierarchical pseudometric.
For a hierarchical pseudometric $d(\quark, \quark)$ and $\gamma \in \Reals_+$, let
\begin{equation}
	c_d(\gamma) \defeq \sup_{1 \leq k \leq l \leq q} \sup_{j \in \J^{(l)}} \sum_{i \in \J^{(k)}} \exp ( - \gamma d( i, j ) ) .
\end{equation}

The following theorem states the main result of this section:

\begin{theorem}[Exponential decay of the Cholesky factors]
	\label{thm-decayAbstractCholesky}
	Assume that $\KM$ fulfils \cref{cond-spatloc,cond-specloc} with the	constants $\gamma, C_\gamma, H,C_{\Phi}$ and the
	hierarchical pseudometric $d(\quark,\quark)$.
	Then
	\begin{equation}
		\left| \left( \chol( \KM ) \right)_{ij} \right| \leq \frac{2 C_{\Phi} c_d\left(\tilde{\gamma}/8\right)^2}{(1-r)^2} \left(4 c_d\left(\tilde{\gamma}/4\right)\frac{C_{\Phi} C_{\gamma} \left(c_d\left(\tilde{\gamma}/2\right)\right)^2} {(1-r)^2} \right)^q \exp\left( -\frac{\tilde{\gamma}}{8}d( i, j ) \right),
	\end{equation}
	where $C_R \defeq \max \left\{ 1,\frac{2C_\gamma C_{\Phi}}{1 + \kappa}\right\}$, $r \defeq \frac{1- \kappa^{-1}}{1 + \kappa^{-1}}$, $\tilde{\gamma} \defeq \frac{-\log(r)} {1 + \log ( c_d ( \gamma/2 ) ) + \log (C_R) - \log ( r ) } \frac{\gamma}{2}$, and $\kappa = H^{-2} C_{\Phi}^2$ is defined as in \cref{thm-condphi}.
\end{theorem}

The remaining part of this section will present the proof of \cref{thm-decayAbstractCholesky}.
We will use the following lemma on the stability of exponential decay under matrix multiplication, the proof of which is a minor modification of that of \cite{jaffard1990proprietes}.

\begin{lemma}
	\label{lem-decayProduct}
	Let $\I$ be an index set that is partitioned as $\I = \J^{(1)} \cup \cdots \J^{(q)}$ and let $d \colon \I \times \I \to \Reals_{\geq 0}$ satisfy
	\begin{equation*}
		d( i_{1}, i_{n+1} ) \leq \sum_{k=1}^{n} d ( i_{k}, i_{k+1}) \quad \text{for all $1\leq n \leq q-1$ and $i_{k} \in \J^{(k)}$.}
	\end{equation*}
	Let $M^{(k)} \in \Reals^{J^{(k)} \times J^{(k+1)}}$ be such that $|M^{(k)}_{i,j}| \leq C \exp ( -\gamma d( i , j ) )$ for $1 \leq k \leq q-1$ and	let
	\begin{equation}
		c_d(\gamma/2) \defeq \sup_{1 \leq k \leq q-1} \sup_{j \in \J^{(k+1)}} \sum_{i \in \J^{(k)}} \exp\left( - \frac{\gamma}{2} d( i, j )\right)\text{ for $\gamma \in \Reals_+$}.
	\end{equation}
	Then, for $1\leq n \leq q-1$,
	\begin{equation*}
		\Absval{ \left( \prod_{k=1}^{n} M^{(k)}\right)_{i,j} } \leq \left( c_d\left( \gamma/2 \right) C \right)^n \exp\left(-\frac{\gamma}{2} d( i, j ) \right) .
	\end{equation*}
\end{lemma}

\begin{proof}
	Set $i_{1} \defeq i$, $i_{n+1}\defeq j$.
	Then
	\begin{align*}
		\Absval{ \left( \prod_{k=1}^{n} M^{(k)}\right)_{i,j} }
		&\leq C^n \sum_{i_{2} , \dots, i_{n} \in \J^{(2)},\ldots,\J^{(n)} } \exp \left( - \gamma \sum_{k=1}^{n} d\left( i_{k}, i_{k+1} \right) \right) \\
		&\leq C^n \exp\left(-\frac{\gamma}{2} d\left( i_{1} , i_{n+1} \right) \right) \sum_{i_{2} , \dots i_{n} \in I } \exp \left( - \frac{\gamma}{2} \sum_{k=1}^{n} d\left( i_{k}, i_{k+1} \right) \right) \\
		&\leq \left( c_d\left( \gamma/2 \right) C \right)^n \exp\left(-\frac{\gamma}{2} d( i, j )\right) .
		% \qedhere
	\end{align*}
\end{proof}

The proof of the following lemma (on the stability of exponential decay under matrix inversion for well conditioned matrices) is nearly identical to that of \cite{jaffard1990proprietes} (we only keep track of constants; see also \cite{demko1984decay} for a related result on the inverse of sparse matrices).

\begin{lemma}
	\label{lem-decayInverse}
	Let $A \in \Reals^{I\times I}$ be symmetric and positive definite with $| A_{i,j} | \leq C \exp ( -\gamma d( i , j ) )$ for some $C,\gamma >0$ and a metric $d(\quark, \quark)$ on $\I$.
	It holds true that
	{\scriptsize
	\begin{align}
		\left|( A^{-1} )_{i,j} \right|
		\leq \frac{4}{\left(\|A\| + \|A^{-1}\|^{-1}\right)( 1-r )^{2}} \exp \left( -\frac{\log( \frac{1}{r}) }{\left(1 + \log\left( c_{d}\left( \gamma/2 \right) \right) + \log( C_{R} ) \right) + \log( \frac{1}{r} )}\frac{\gamma}{2} d( i , j )\right)
	\end{align}}
	where $c_d(\gamma/2) \defeq \sup_{j \in \I} \sum_{i \in I} \exp\left( - \frac{\gamma}{2} d( i, j )\right)$, $C_{R} \defeq \max\left\{1 , \frac{2 C}{\|A\| + \| A^{-1}\|^{-1}} \right\} = \max\left\{1 , \frac{2 C \|A^{-1}\|}{ 1 + \kappa } \right\}$, $r\defeq \frac{1 - \frac{1}{\|A\|\|A^{-1}\|}}{1 + \frac{1}{\|A\|\|A^{-1}\|}} = \frac{1 - \kappa^{-1}}{1 + \kappa^{-1}}$, and $\kappa \defeq \norm{A} \norm{A^{-1}}$ is the condition number of $A$.
\end{lemma}

\begin{proof}
	On a compact set not containing $0$, the function $x \mapsto x^{-1}$ can be accurately approximated by low-order polynomials in $x$.
	Then, the spread of the exponential decay can be controlled by \cref{lem-decayProduct}.
	See \cref{apsec-abstractDecay} for details.
\end{proof}

By representing Schur complements as matrix inverses, \cref{lem-decayInverse} can also be used to show that the Cholesky factors of well-conditioned exponentially-decaying matrices are exponentially decaying.
The following lemma appears in a similar form in \cite{benzi2000orderings} for banded matrices and in \cite{krishtal2015localization} without explicit constants.

\begin{lemma}
	\label{lem-decayCholesky}
	Let $B \in \Reals^{I\times I} \simeq \Reals^{N\times N}$ be symmetric and positive definite with condition number $\kappa$ and such that $\left|B_{i,j}\right| \leq C \exp( -\gamma d( i , j ) )$ for some constant $C>0$ and some metric $d$ on $I$.
	Let $L$ be the Cholesky factor (in an arbitrary order) of $B^{-1}$ ($B^{-1} = LL^{T}$).
	Then
	{\small
	\begin{align}
		\left|L_{i,j}\right|
		\leq \frac{ 4 \sqrt{\|B\|}}{\left(\|B\| + \|B^{-1}\|^{-1}\right)( 1-r )^{2}} \exp\left( \frac{\log( r ) }{1 + \log\left( c_{d}\left( \gamma/2 \right) \right) + \log( C_{R} ) - \log( r )}\frac{\gamma}{2} d( i , j )\right).
	\end{align}}
	where $c_d(\gamma/2) \defeq \sup_{j \in \I} \sum_{i \in I} \exp\left( - \frac{\gamma}{2} d( i, j )\right)$, $C_{R} \defeq \max\left\{1 , \frac{2 C\|B^{-1}\|}{1 + \kappa} \right\}$, and $r \defeq \frac{1 -\kappa^{-1}}{1 + \kappa^{-1}}$.
\end{lemma}

\begin{proof}
	\Cref{lem-blockChol2Scale} implies that the Schur complements of $B^{-1}$ can be expressed as inverses of sub-matrices of $B$.
	The result then follows from \cref{lem-decayInverse} (see \cref{prv-decayCholesky} for details).
\end{proof}

The last ingredient needed to prove the exponential decay of the Cholesky factors of $\KM$ is the following lemma showing the stability of exponential decay under inversion for block-lower-triangular matrices (this operation appears in
the definition of $\bar{L}$ in \eqref{eqjhgjhgyug}):

\begin{lemma}
	\label{lem-decayTriang}
	Let $\I$ be an index set that is partitioned as $\I = \J^{(1)} \cup \cdots \J^{(q)}$ and assume that the matrix $L \in \Reals^{\I \times \I}$ is block-lower triangular with respect to this partition, with identity matrices as diagonal blocks.
	If $d(\quark, \quark)$ is a hierarchical pseudometric such that $| L_{ij} | \leq C \exp\left( -\gamma d(i,j) \right)$ (for some $C\geq 1$ and $\gamma>0$) then it holds true that
	\begin{equation}
		\left| ( L^{-1} )_{ij} \right|
		\leq 2^q \left( c_d\left(\gamma/2\right) C \right)^{q}
		\exp\left( - \frac{\gamma}{2} d(i,j) \right).
	\end{equation}
	with $c_d(\gamma) \defeq \sup_{1 \leq k \leq l \leq q} \sup_{j \in \J^{(l)}}
	\sum_{i \in \J^{(k)}} \exp\left( - \gamma d( i, j )\right)$.
\end{lemma}

\begin{proof}
	The Neumann series of a $q\times q$ block-lower-triangular matrix with identity matrices on the (block) diagonal can be written as
	\begin{equation}
		L^{-1} = \sum_{k =0}^q \left( \Id - L \right)^k\,.
	\end{equation}
	Since the sum terminates in $q$ steps, the thickening of the exponential decay can be bounded using \cref{lem-decayProduct}.
	See \cref{prv-decayTriang} for details.
\end{proof}

By applying the above results to the decomposition obtained in \cref{lem-BlockCholesky}, we conclude the proof of \cref{thm-decayAbstractCholesky}.
See \cref{prv-decayAbstractCholesky} for details.

\subsection{Complexity and error estimates}
\label{ssec-compErr}

The results of the previous sections allow us to prove the following theorem on the exponential decay of the Cholesky factors and the accuracy of their truncation:

\begin{theorem}
	\label{thm-decayCholesky}
	In the setting of \cref{examp-subsamp,examp-average} there exist constants $C,\gamma, \alpha > 0 $ depending only on $d$, $\Omega$, $s$, $\|\IK\|$, $\norm{ \IK^{-1} }$, $h$, and	$\delta$, such that the entries of the Cholesky factor $L$ of $\KM$ satisfy
	\begin{equation}
		| L_{ij} | \leq C N^\alpha \exp ( -\gamma d( i, j ) )\,,
	\end{equation}
	where $d \colon \I \times \I \to \Reals$ is the hierarchical pseudometric defined by
	\begin{equation}
		d( i, j ) \defeq h^{-\min( k, l )} \dist\left( \supp\left(\phi_i\right), \supp \left( \phi_j \right) \right) \quad \text{for all $i \in \J^{(k)}$, $j \in \J^{(l)}$.}
	\end{equation}
	As a consequence, writing
	\begin{equation}
		L^{S}_{ij} \defeq
		\begin{cases}
		L_{ij}, & \text{ for } (i,j) \in S \\
		0, & \text{ else,}
		\end{cases}
	\end{equation}
	with $S \supset S_{d,\rho} \defeq \{ (i,j) \mid d(i,j) \leq \rho \}$, we have $\bignorm{ \KM - L^{S} L^{S, \top} }_{\FRO} \leq \epsilon$ for $\rho \geq \tilde{C}(C,\gamma) \log(N/\epsilon)$.
	Furthermore, writing $E \defeq \KM - L^{S} L^{S, \top}$, using the $\epsilon$-perturbation $\KM - E$ of $\KM$ as the input to \cref{alg-ICholesky} returns $L^{S}$ as the output.
\end{theorem}

\begin{proof}
	\Cref{thm-spatloc,thm-specloc} imply that \cref{cond-spatloc,cond-specloc} are fulfilled with constants depending only on $d$, $s$, $\|\IK\|$, $\norm{ \IK^{-1} }$, $h$, and $\delta$.
	\Cref{thm-decayAbstractCholesky} concludes the exponential decay of $L$.
	The accuracy of the truncated factors follows directly from the exponential decay.
\end{proof}

\Cref{thm-decayCholeskyIntro} is a direct consequence of \cref{thm-decayCholesky}.

\begin{proof}[Proof of \cref{thm-decayCholeskyIntro}]
	As described in \cref{sssec-choleskyGamblets}, the maximin ordering can be represented as a hierarchical ordering satisfying
	the conditions of \cref{examp-subsamp}.
	The result follows from \cref{thm-decayCholesky} by observing that the sparsity pattern $S_\rho$ specified in \cref{ssec-simpleAlg} satisfies
	\begin{equation}
		S_{d, (\delta h)^{-1} \rho} \supset S_{\rho} \supset S_{d, \delta h \rho}\,.
	\end{equation}
	Scaling the weights of the measurement functions $\phi_i$ to $1$ increases the	error by a factor that is at most polynomial in $N$, which can be subsumed into the $\log(N)$-dependence of $\rho$ by increasing the constants in the decay estimates.
\end{proof}

While accurate (per \cref{thm-decayCholesky}), it is computationally inefficient to compute the full Cholesky factor first (with \cref{alg-Cholesky}) and then truncate it according to $S_{\rho}$.
Instead, we want to directly compute an approximation of $L$ from the incomplete factorization \cref{alg-ICholesky}, whose complexity is bounded by the following theorem:

\begin{theorem}
	\label{thm-complexityCholesky}
	In the setting of \cref{examp-subsamp,examp-average}, there exists a constant $C(d,\delta)$, such that, for $S \subset \{ (i,j) \mid d(i,j) \leq \rho \}$, the application of \cref{alg-ICholesky} has computational complexity 	$C(d,\delta) N q \rho^d$ in space and
	$C(d,\delta) N q^2 \rho^{2d}$ in time.
	In particular, $q\propto \log N /\ln \frac{1}{h^d}$ implies the upper bounds of $C(d,\delta, h) \rho^d N \log N$ on the space complexity, and of $C(d,\delta, h) \rho^{2d} N \log^2 N$ on the time complexity.
\end{theorem}

\begin{proof}
	Defining $m \defeq \max_{j\in \I, 1\leq k \leq q} \# \{i \in \J^{(k)} \mid i \prec j \text{ and } d(i,j) \leq \rho \}	$, $|x_i-x_j|\geq \delta^{-1}h^l$ for $i,j\in \I^{(l)}$ implies that $m \leq C(d,\delta) \rho^d$.
	Therefore $\# \{i \in \I \mid i \prec j \text{ and } d(i,j) \leq \rho \} \leq q m N$ implies the bound on space complexity.
	
	Consider the structure of the nested for-loops of \cref{alg-ICholesky} and observe that, for every $k$ in the innermost loop, the number of distinct $(i,j)$ satisfying $i\prec j \prec k$, $(j,k) \in S$ and $(i,j) \in S$ is at most $(qm)^2$.
	This implies the upper bound $N(qm)^2$ on the time complexity.
\end{proof}

\Cref{thm-decayCholesky,thm-complexityCholesky} imply that the application of \cref{alg-ICholesky} to $\KM - E$ (the $\epsilon$-perturbation of $\KM$ described in \cref{thm-decayCholesky}) returns an $\epsilon$-accurate Cholesky factorization of $\KM$ in computational complexity $\BigO (N \log^2(N) \log^{2d}(N/ \epsilon))$.
In practice we do not have access to $E$, so we need to rely on the stability of
\cref{alg-ICholesky} to deduce that $\KM$ and $\KM - E$ (used as inputs) would yield similar outputs, for sufficiently small $E$.
Even though such a stability property of {\ICH} would also be required by prior works on incomplete LU-factorization such as \cite{gines1998lu}, we did not find
this type of result in the literature.
We also found it surprisingly difficult to prove (and were unable to do so) for the maximin ordering and sparsity pattern, although we always observed stability of \cref{alg-ICholesky} in practice, for reasonable values of $\rho$.
We can however prove stability of \cref{alg-ICholesky} when using a slight modification of the ordering and sparsity pattern that compromises neither
the computational complexity nor the accuracy of the factorization.
The modified ordering and sparsity pattern, being inspired by the concepts of red-black orderings \cite{iwashita2003block} and supernodal factorizations
\cite{rothberg1994efficient,liu1993finding} also allows one to take advantage of
parallelism and dense linear algebra operations and could therefore be used to improve the practical performance of the algorithm.
For $r>0$, $1 \leq k \leq q$ and $i \in \J^{(k)}$, write
\begin{equation}
	B^{(k)}_{r} \left( i \right) \defeq \{ j \in \J^{(k)} \mid d( i, j ) \leq r \}.
\end{equation}

\begin{construction}[Supernodal multicolor ordering and sparsity pattern]
	\label{const-superMulti}
	Let $\KM \in \Reals^{\I \times \I}$ with $\I \defeq \bigcup_{1 \leq k \leq q} \J^{(k)}$ and let $d( \quark, \quark )$ be a hierarchical pseudometric.
	For $\rho \geq 1$, define the \emph{supernodal multicolor ordering} $\prec_{\rho}$ and \emph{sparsity pattern} $S_\rho$ as follows.
	For each $k\in \{1,\ldots,q\}$, select a subset $\tJ^{(k)} \subset \J^{(k)}$ of indices such that
	\begin{align}
		&\forall \ti,\tj \in \tJ^{(k)}, & \ti \neq \tj \implies
		B^{(k)}_{\rho/2}\left(\ti \right) \cap B^{(k)}_{\rho/2}\left( \tj \right)
		= \emptyset , \\
		&\forall i \in \J^{(k)}, & \exists \ti \in \tJ^{(k)}:
		i \in B_{\rho}^{(k)} \left( \ti \right) .
	\end{align}
	Assign every index in $\J^{(k)}$ to the element of $\tJ^{(k)}$ closest to it,
	using an arbitrary method to break ties.
	That is, writing $j\leadsto \tj$ for the assignment of $j$ to $\tj$,
	\begin{equation}
		\tj \in \argmin_{\tj' \in \tJ^{(k)}} d\left( j, \tj' \right)\,,
	\end{equation}
	for all	$j \in \J^{(k)}$ and $\tj \in \tJ^{(k)}$ such that $j\leadsto \tj$.
	Define $\tI \defeq \bigcup_{1 \leq k \leq q} \tJ^{(k)}$ and define the
	auxiliary sparsity pattern $\tS_{\rho} \subset \tI \times \tI$ by
	\begin{equation}
		\tS_{\rho} \defeq \left\{ \left( \ti, \tj \right) \in \tI \times \tJ \,\middle|\, \exists i \leadsto \ti, j \leadsto \tj: d(i,j) \leq \rho \right\}.
	\end{equation}
	Define the sparsity pattern $S_{\rho} \subset \I \times \I$ as
	\begin{equation}
		S_{\rho} \defeq \left\{ ( i, j ) \in \I \times \I \,\middle|\, \exists \ti, \tj \in \tI : i \leadsto \ti, j \leadsto \tj, \left(\ti, \tj \right) \in \tS_{\rho} \right\}.
	\end{equation}
	and call the elements of $\tJ^{(k)}$ \emph{supernodes}.
	Color each $\tj \in \tJ^{(k)}$ in one of $p^{(k)}$ colors such that no
	$\ti,\tj \in \tJ^{(k)}$ with $\left( \ti, \tj \right) \in \tS_{\rho}$ have the	same color.
	For $i\in \J^{(k)}$ write $\text{node}(i)$ for the $\ti \in \tJ^{(k)}$ such that $i\leadsto \ti$ and write $\text{color}(\ti)$ for the color of $\ti$.
	Define the supernodal multicolor ordering $\prec_{\rho}$ by reordering the elements of $\I$ such that
	\begin{enumerate}[label=(\arabic*)]
		\item $i\prec_{\rho} j$ for $i\in \J^{(k)}$, $j\in \J^{(l)}$ and $k<l$;
		\item within each level $\J^{(k)}$, we order the elements of supernodes colored in the same color consecutively, i.e.\ given $i, j\in \J^{(k)}$ such that $\text{color}(\text{node}(i))\not=\text{color}(\text{node}(j))$,
		$i\prec_{\rho} j \implies i'\prec_{\rho} j'$ for $\text{color}(\text{node}(i'))=\text{color}(\text{node}(i))$,
		and $\text{color}(\text{node}(j'))=\text{color}(\text{node}(j))$;
		and
		\item the elements of each supernode appear consecutively, i.e.\ given $i, j\in \J^{(k)}$ such that $\text{node}(i)\not=\text{node}(j)$,
		$i\prec_{\rho} j \implies i'\prec_{\rho} j'$ for $\text{node}(i')=\text{node}(i)$,
		and $\text{node}(j')=\text{node}(j)$.
	\end{enumerate}
\end{construction}

Starting from a hierarchical ordering and sparsity pattern, the modified ordering
and sparsity pattern can be obtained efficiently:

\begin{lemma}
	\label{lem-constSuperMulti}
	In the setting of \cref{examp-subsamp,examp-average}, given $\{ ( i, j ) \mid d( i, j ) \leq \rho \}$, there exist constants $C$ and $p_{\max}$ depending only on the dimension $d$ and the cost of computing $d( \quark, \quark )$ such that the ordering and sparsity pattern presented in \cref{const-superMulti} can be constructed with $p^{(k)} \leq p_{\max}$, for each $1 \leq k \leq q$, in computational complexity $C q \rho^d N$.
\end{lemma}

\begin{proof}
	The aggregation into supernodes can be done via a greedy algorithm by keeping track of all nodes that are not already within distance $\rho/2$ of a supernode and removing them one-at-a-time.
	We can then go through $\rho$-neighbourhoods and remove points within distance $\rho/2$ from our list of candidates for future supernodes.
	To create the coloring, we use the greedy graph coloring of \cite{husfeldt2015graph} on the undirected graph $G$ with vertices $\tJ^{(k)}$ and edges $\bigl\{ ( \ti, \tj ) \in \tS_{\rho} \,\big|\, \ti,\tj \in \tJ^{(k)} \bigr\}$.
	Defining $\deg(G)$ as the maximum number of edges connected to any vertex of $G$, the computational complexity of greedy graph coloring is bounded above by $\deg(G) \#\left( \J^{(k)} \right)$ and the number of colors used by $\deg(G) + 1$.
	A sphere-packing argument shows that $\deg(G)$ is at most a constant depending only on the dimension $d$, which yields the result.
\end{proof}

\begin{theorem}
	\label{thm-accuracyICholesky}
	In the setting of \cref{examp-subsamp,examp-average}, there exists a
	constant $C$ depending only on $d,s, \norm{ \IK }, \norm{ \IK^{-1} }$, $h$, and $\delta$ such that, given the ordering $\prec_{\rho}$ and sparsity pattern $S_{\rho}$ defined as in
	\cref{const-superMulti} with $\rho \geq C \log(N / \epsilon)$, the incomplete Cholesky factor $L$ obtained from \cref{alg-ICholesky} has accuracy
	\begin{equation}
		\norm{ LL^T - \KM }_{\FRO} \leq \epsilon.
	\end{equation}
	Furthermore, \cref{alg-ICholesky} has complexity of at most
	$C N \rho^{2d} \log^{2} N$ in time and at most $C N \rho^{d} \log N$ in space.
\end{theorem}

\begin{proof}
	The triangle inequality implies that $S_{\rho} \subset \{ (i,j) \mid d(i,j) \leq 2 \rho \}$ and hence the bound on the complexity of \cref{alg-ICholesky} follows from \cref{thm-complexityCholesky}.
	The approximation property of the incomplete factors follows from the last part of \cref{thm-decayCholesky} and a stability result for the incomplete Cholesky factorization with the supernodal multicolor ordering and sparsity pattern detailed in \cref{apsec-superMulti}.
\end{proof}

This allows us to prove the main theorem presented in the introduction.

\begin{proof}[Proof of \cref{thm-decayApproxIChol}]
	\cref{thm-decayApproxIChol} follows from \cref{thm-accuracyICholesky} since rescaling the weights of the measurements to $1$ increases bounds on errors by at most a multiplicative polynomial factor in $N$.
	By increasing the constant, this factor can be subsumed in the $N$-dependence of $\rho$.
\end{proof}

We have now established the results on exponential decay of the Cholesky factors of $\KM$ and the accuracy of \cref{alg-ICholesky}.
Before proceeding to the next section, we will quickly establish a result on low-rank approximation of the Cholesky factors.

\begin{theorem}[Approximate PCA]
	\label{thm-PCA}
	In the setting of \cref{thm-decayCholeskyIntro}, take $\rho = \infty$ and let $L^{(k)}$ be the matrix formed by the leading $k$ columns of the Cholesky factors of $\KM$ in the maximin ordering.
	Let $l[i_k]$ be as in \eqref{eq-dist-point-to-boundary}.
	Then there exists a constant $C$ depending only on $\|\IK\|$, $\norm{ \IK^{-1} }$, $d$, and $s$ such that
	\begin{equation}
		\bignorm{ \KM - L^{(k)} L^{(k),\top} } \leq C l_{i_{k+1}}^{2s - d}
	\end{equation}
\end{theorem}

\begin{proof}
	Write $\I = \I_{1} \cup \I_{2}$ with $I_1 \defeq \left\{i_{1}, \dots , i_{k}\right\}$ and
	$\I_{2} \defeq \I \setminus \I_{1}$.
	By \cref{lem-blockChol2Scale}, the approximation error made by keeping only the first $k$ columns of the Cholesky factorization is equal to the Schur complement $\KM_{2,2} - \KM_{2,1} \KM_{1,1}^{-1} \KM_{1,2}$.
	Consider the implicit hierarchy of the maximin ordering as in \cref{fig-ordering} with $h=1/2$ and let $p\in \{1,\ldots,q\}$ be such that $2^{-p}\leq l[k]/l[1]\leq 2^{-p+1}$.
	Write $\I = \I_{a} \cup \I_{b}$ with $I_a \defeq \I^{(p)}$ and $\I_{b} \defeq \I \setminus \I^{(p)}$.
	The variational property \eqref{eqkwhdhkjhd} implies that $\KM_{2,2} - \KM_{2,1} \KM_{1,1}^{-1} \KM_{1,2} \leq \KM_{b,b} - \KM_{b,a} \KM_{a,a}^{-1} \KM_{a,b}$.
	\cref{thm-specloc} (with $h=1/2$ obtained from the implicit hierarchy of \cref{fig-ordering}) implies that $\KM_{b,b} - \KM_{b,a} \KM_{a,a}^{-1} \KM_{a,b} \leq C (\frac{1}{2})^{2s (p-1)-d}$ (the extra multiplicative $(\frac{1}{2})^{-d}$ term arises because the measurement functions are scaled by $h^{kd/2}$ in \cref{examp-subsamp} with $h = \tfrac{1}{2}$).
	We conclude the proof using $2^{-p-1}\leq l[k+1]/l[1]\leq 2^{-p+1}$.
\end{proof}

%% file: sec-extensions.tex
\subsection{The cases \texorpdfstring{$s \leq d/2 \text{ or }s \notin \Naturals$}{s is small or not an integer}}

\cref{thm-decayCholeskyIntro} requires that $s>d/2$ to ensure that the elements of $H^s(\Omega)$ are continuous (by the Sobolev embedding theorem) and that pointwise evaluations of the Green's function are well defined.
The accuracy estimate of \cref{thm-decayCholeskyIntro} can be extended to $s\leq d/2$ by replacing pointwise evaluations of the Green's function by local averages and using variants of the Haar pre-wavelets of \cref{examp-average} instead of variants of the subsampled Diracs of \cref{examp-subsamp} to decompose $\Theta$ as in \eqref{eqn-gambletFact}.
Numerical experiments also suggest that the exponential decay of Cholesky factors still holds for $s\leq d/2$ if the local averages of \cref{examp-average} are sub-sampled as in \cref{examp-subsamp}, whereas the low-rank approximation becomes sub-optimal.
As illustrated in \cref{tab-maternFrac3d}, for Mat{\'e}rn kernels we observe no difference (in accuracy vs.\ complexity) between integer and non-integer values of $s$.

\subsection{Sparse factorization of \texorpdfstring{$\IKM = \KM^{-1}$}.}
\label{sssec-spFacIKM}

Let $L L^{\top} = \KM$ be the Cholesky factorization of the covariance matrix $\KM$.
Writing $\rP$ for the order-reversing permutation,
\begin{equation}
	\rP \KM^{-1} \rP = \rP L^{-\top} L^{-1} \rP = \bigl( \rP L^{-\top} \rP \bigr) \bigl( \rP L^{-1} \rP \bigr)\,.
\end{equation}
Since $\rP L^{-\top} \rP$ is lower triangular, it is the Cholesky factor of $\KM^{-1}$ in the reverse elimination ordering.
Furthermore, since $L^{-\top} = \IKM L$ and both $\IKM$ and $L$ are exponentially decaying, the Cholesky factors of $\IKM$ are also exponentially decaying if the Gaussian elimination is performed using the reverse of \cref{ssec-simpleAlg}'s ordering.
In fact, the following, stronger, theorem holds:

\begin{theorem}
	\label{thmfytfy5}
	In the setting of \cref{thm-decayCholeskyIntro}, let
	\begin{align}
		\mathring{S}_{\rho} &\defeq \left\{ ( i, j ) \in \I \times \I \,\middle|\, \dist \left( \supp( \phi_{i} ), \supp( \phi_{j} ) \right) \leq \rho \min (l[i],l[j] ) \right\},
	\end{align}
	let $L$ be the Cholesky factor of $\IKM$ in the reverse ordering, and define
	\begin{align}
		L^{\mathring{S}_\rho}_{ij} & \defeq
		\begin{cases}
			L_{ij}, & \text{ for } (i,j) \in \mathring{S}_{\rho} , \\
			0, & \text{ else.}
		\end{cases}
	\end{align}
	Then there exists a constant $C$ depending only on
	$d$, $\Omega$, $s$, $\|\IK\|$, $\|\IK^{-1}\|$ and $\delta$ such that for $\rho \geq C \log(N/\epsilon)$, we have $\bigl\|P\IKM P - L^{\mathring{S}_\rho} L^{\mathring{S}_\rho, \top} \bigr\|_{\FRO} \leq \epsilon$.
\end{theorem}

Using this result and the fact that $\# \mathring{S}_\rho$ has $\BigO (\rho^d +1)$ nonzero entries per column, one can prove that using \cref{alg-ICholesky} with a supernodal ordering as described in \cref{const-superMulti} yields an $\epsilon$-approximate Cholesky factorization of $\IKM$ in computational complexity $\BigO (N \log ( N/\epsilon )^{2d} )$ in time and $\BigO (N \log ( N/\epsilon )^{d} )$ in space.
The matrix $\IKM$ is \emph{essentially} a discretized elliptic partial differential operator, and analogous results can be obtained in the setting where $\IKM$ is obtained as a discretization of $\IK$ with regular finite elements and and $\KM$ is the inverse of that discretized operator.
Numerical experiments suggest that exponential decay properties also hold for discretized second-order elliptic equations in two or three dimensions (where $s = 1 \leq d/2$) when using subsampling as in \cref{examp-subsamp};
see \cite[Section 3.1]{schroeder1978fast} for a special case of this result on regular meshes.
Thus, by computing the incomplete Cholesky factorization, we obtain a direct solver for general elliptic PDEs with complexity $\BigO ( N \log ( N/\epsilon )^{2d} )$ in time and $\BigO (N \log ( N/\epsilon )^{d} )$ in space.
To the best of our knowledge, this is the best asymptotic complexity reported for such a solver in the literature (for elliptic PDEs with rough coefficients and rigorous a priori estimates of complexity vs.\ accuracy).
It is not surprising that we obtain a fast solver for elliptic PDEs because our work is based on the fast solvers introduced in \cite{owhadi2015multigrid,owhadi2017universal}, which in turn can be shown to be a block-wise version of the Cholesky factorization in nonstandard form introduced by \cite{gines1998lu}, where the inverses of diagonal blocks are computed using iterative methods.
By instead applying the Cholesky factorization in nonstandard form, the logarithmic factor in the complexity of the gamblet transform can be improved.
However, the error estimates of \cite{owhadi2015multigrid} and \cite{owhadi2017universal} improve significantly upon those in \cite{gines1998lu} by establishing that exponential accuracy can be obtained with a finite number of vanishing moments even for rough coefficients.
The present work further extends the results on Cholesky factorization to the setting of multiresolution schemes based on subsampling (without any vanishing moments).
For such multiresolution basis the nonstandard form just reduces to computing an ordinary incomplete Cholesky factorization with the smaller sparsity pattern $\mathring{S}_\rho$, thus greatly simplifying the implementation.
We note that by using direct inversion methods similar to \cite{lin2011algorithm} it would be possible in principle to directly compute $\epsilon$-approximations of the Cholesky factors of $\KM^{-1}$ from $\BigO ( N \log(N/\epsilon)^d )$ entries of $\KM$ at computational cost of $\BigO ( N \log(N/\epsilon)^{2d} )$, but we defer a more detailed investigation to future work.

%% file: sec-comparison.tex
\subsection{\texorpdfstring{$\mathcal{H}$}{H}-matrix approximations from sparse Cholesky factorization}

The $\mathcal{H}$-matrix data structure \cite{hackbusch1999sparse} uses low-rank
approximations for blocks $\KM_{\bI \bJ}$ ($\bI, \bJ \subset \I$) fulfilling the admissibility condition
\begin{equation}
	\min \bigl( \diam \{ x_i \}_{i \in \bI}, \diam \{ x_i \}_{i \in \bJ} \bigr) \leq \eta \dist \bigl( \{ x_i \}_{i \in \bI}, \{ x_i \}_{i \in \bJ} \bigr).
\end{equation}
The approximation property of the incomplete Cholesky factorization in maximin ordering (\cref{thm-decayCholeskyIntro}) directly implies bounds on the spectral decay of admissible blocks in the $\H$-matrix framework, as can be seen from the representation
\begin{equation}
	\KM = L L^{\top} \iff \KM = \sum_{i = 1}^N L_{:i} \otimes L_{:i}
\end{equation}
of the Cholesky factorization of $\KM$.
If $L$ is sparse according to the sparsity pattern obtained in \cref{ssec-simpleAlg} then $L_{:i} \otimes L_{:i}$ can contribute to the rank of the sub-matrix $\KM_{\bI\bJ}$ only if
{\small
\begin{equation}
	\label{eqkjhwkdodhpdj}
	2 \rho l[i] \geq \dist \left( \{ x_j \}_{j \in \bI} , \{ x_j \}_{j \in \bJ} \right) \text{ and } \max\left( \dist\left(x_i, \{ x_j \}_{j \in \bI} \right), \dist\left(x_i, \{ x_j \}_{j \in \bJ} \right) \right) \leq \rho l[i].
\end{equation}}
The number of $i \in \I$ satisfying \eqref{eqkjhwkdodhpdj} is at most $C(\eta, d) \rho^d \log N$, which recovers (up to constants) the same rank bounds as obtained in \cite{bebendorf2003existence} for second-order elliptic PDEs with rough coefficients.
However the converse is not true and most hierarchical matrix representations can not be written in terms of a sparse Cholesky factorization of $\KM$.
For example, adding a diagonal matrix to $\KM$ does not affect the ranks of admissible blocks, but it diminishes the screening effect and thus the approximation property of the incomplete Cholesky factorization as obtained in \cref{ssec-simpleAlg} (see \cref{ssec-nugget}).

\subsection{Comparison to Cholesky factorization in wavelet bases}

\cite{gines1998lu} compute sparse Cholesky factorizations of (discretized) differential/integral operators represented in a wavelet basis.
Using a \emph{fine-to-coarse} elimination ordering, they establish that the resulting Cholesky factors decay polynomially with an exponent matching the number of vanishing moments of the underlying wavelet basis.

For differential operators, this coincides algorithmically with the Cholesky factorization described in \cref{sssec-spFacIKM} and the gamblet transform of \cite{owhadi2015multigrid} and \cite{owhadi2017universal}, whose estimates guarantee exponential decay.
In particular \cite{gines1998lu} numerically observe a uniform bound on $\cond(\IKMC^{(k)})$ which they relate to the approximate sparsity of their proposed Cholesky factorization.

For integral operators, \cite{gines1998lu} use a \emph{fine-to-coarse} ordering
and we use a \emph{coarse-to-fine} ordering.
While their results rely on the approximate sparsity of the integral operator represented in the wavelet basis, our approximation remains accurate for multiresolution bases (e.g.\ the maximin ordering in \cref{ssec-simpleAlg}) in which $\KM$ is dense, which avoids the $\BigO(N^2)$ complexity of a basis transform (or the implementation of adaptive quadrature rules to mitigate this cost).

\subsection{Vanishing moments}

Let $\mathcal{P}^{s-1}(\tau)$ denote the set of polynomials of order at most $s-1$ that are supported on $\tau \subset \Omega$.
\cite{owhadi2015multigrid} and \cite{owhadi2017universal} show that \eqref{eqn-speclocup} and \eqref{eqn-specloclow} hold when $\IK$ is an elliptic partial differential operator of order $s$ (as described in \cref{sssec-classElliptic}) and the measurements are local polynomials of order up to $s-1$ (i.e.\ $\phi_{i,\alpha}=1_{\tau_i} p_\alpha$ with $p_\alpha \in \mathcal{P}^{s-1}(\tau_i)$).
Using these $\phi_{i,\alpha}$ as measurements is equivalent to using wavelets $\phi_i$ satisfying the vanishing moment condition
\begin{equation}
	\dualprod{ \phi_i }{ p } = 0 \quad \text{for all $i \in \I$, $p \in \mathcal{P}^{s-1}$.}
\end{equation}
The requirement for vanishing moments has three important consequences.
First, it requires that the order of the operator be known a priori, so that a suitable number of vanishing moments can be ensured.
Second, ensuring a suitable number of vanishing moments greatly increases the complexity of the implementation.
Third, in order to provide vanishing moments, the measurements $\phi_i$, $i \in \J^{(k)}$, have to be obtained from weighted averages over domains of size of order $h^k$.
Therefore, even computing the first entry of the matrix $\KM$ in the multiresolution basis will have complexity $\BigO(N^{2})$, since it requires taking an average over almost all of $\I \times \I$.
One of the main analytical result of this paper is to show that these vanishing moment conditions and local averages are not necessary for higher order operators (which, in particular, enables the generalization of the gamblet transform to hierarchies of measurements defined as in \cref{examp-subsamp,examp-average}).

\subsection{Comparison to Multiresolution Approximation (M-RA)}

In spatial statistics, the method most closely related to ours is the M-RA of \cite{katzfuss2016multi} where a Gaussian process is approximated by a sum, at different scales, of predictive processes described in \cite{banerjee2008gaussian}.
Following the intuition of the screening effect, these processes are assumed to be block-independent with respect to a domain decomposition at the respective scale, allowing for near-linear computational complexity.
Although the specific multiresolution scheme and its accuracy are a function of the specific choice of basis functions and of the \emph{knots} to be conditioned upon at each scale, no systematic strategy and no theoretical error bounds are provided for best accuracy.
We suspect that no scheme relying on block-sparsity assumptions can also guarantee exponential accuracy in near-linear computational complexity, though we note that the \emph{taper-M-RA} introduced by \cite{katzfuss2017multi}, independently of and after the first version of the present article, does not impose conditional block-independence and could therefore be made exponentially accurate.
While our present work and that of \cite{katzfuss2016multi} are both motivated by a hierarchical exploitation of the screening effect, we identify a concrete and simple algorithm that has a guaranteed exponential accuracy for a wide range of kernel matrices.

%% file: sec-conclusions.tex
We have shown that the dense covariance matrices obtained from a wide range of covariance functions associated to smooth Gaussian processes have almost sparse Cholesky factors.
Using this property, these matrices can be inverted in near-linear computational complexity just by applying zero fill-in incomplete Cholesky factorization with an a priori ordering and sparsity pattern.
Sparse Cholesky factorization of sparse matrices is by now a classical field, but we are not aware of prior work on the sparse factorization of dense matrices, other than for the purpose of preconditioning.
While our algorithm is subject to the curse of high dimensionality like other hierarchy-based methods, it is able to exploit low dimensionality in the data without any user intervention.
Our results are motivated by the probabilistic interpretation of Cholesky factorization and proved rigorously by using and generalizing recent results on operator-adapted wavelets.
By reversing the elimination order, we also obtain a fast direct solver for elliptic PDEs whose rigorous a priori accuracy-vs.-complexity estimates advance the current state of the art for general elliptic PDEs.